\documentclass[11pt]{amsart}
\usepackage{amssymb,latexsym}
\usepackage{amsfonts}
\usepackage{amsmath}
\usepackage{enumitem}
\usepackage{extarrows}
\usepackage{marvosym}
\usepackage{lipsum}
\usepackage{tikz-cd}
\usepackage[colorlinks=true, linkcolor=red, citecolor=blue, urlcolor=blue]{hyperref}
\usepackage[a4paper, text={160mm,250mm}, centering, footskip=20pt]{geometry}

\pretocmd{\tableofcontents}{\hypersetup{linkcolor=blue}}{}{}
\apptocmd{\tableofcontents}{\hypersetup{linkcolor=red}}{}{}

\newtheorem{theorem}{Theorem}[section]
\newtheorem{corollary}[theorem]{Corollary}
\newtheorem{lemma}[theorem]{Lemma}
\newtheorem{proposition}[theorem]{Proposition}
\newtheorem{definition}[theorem]{Definition}

\newtheorem{notation}[theorem]{Notation}
\newtheorem{setup}[theorem]{Setup}

\newtheorem{construction}[theorem]{Construction}
\newtheorem{question}[theorem]{Question}
\newtheorem{problem}[theorem]{Problem}
\newtheorem{remark}[theorem]{\bf{Remark}}

\numberwithin{equation}{section}

\newcommand{\bp}{\bar{\partial}}
\newcommand{\beq}{\begin{equation}}
	\newcommand{\eeq}{\end{equation}}
\newcommand{\beqn}{\begin{equation*}}
	\newcommand{\eeqn}{\end{equation*}}

\newcommand{\C}{\mathbb{C}}
\newcommand{\R}{\mathbb{R}}
\newcommand{\N}{\mathbb{N}}

\newcommand{\CP}{\mathbb{P}}

\newcommand{\supp}{\mathrm{supp}}
\newcommand{\pp}{\partial\bar{\partial}}
\newcommand{\p}{\partial}

\newcommand{\w}{\omega}

\newcommand{\im}{\sqrt{-1}}

\newcommand{\mF}{\mathcal{F}}
\newcommand{\mO}{\mathcal{O}}
\newcommand{\mT}{\mathcal{T}}
\newcommand{\mE}{\mathcal{E}}

\newcommand{\mQ}{\mathcal{Q}}
\newcommand{\mG}{\mathcal{G}}
\newcommand{\mS}{\mathcal{S}}
\newcommand{\reg}{\mathrm{reg}}

\newcommand{\orb}{\mathrm{orb}}
\newcommand{\ch}{\mathrm{ch}}
\newcommand{\Aut}{\mathrm{Aut}}
\newcommand{\codim}{\mathrm{codim}}
\newcommand{\tor}{\mathrm{tor}}

\DeclareMathOperator \Hom{Hom}
\DeclareMathOperator \Vol{Vol}

\DeclareMathOperator \End{End}
\DeclareMathOperator \Id{Id}

\DeclareMathOperator \rank{rank}
\DeclareMathOperator \tr{tr}

\DeclareMathOperator \imag{image}

\DeclareMathOperator{\rc}{Ric}

\title[Non-abelian Hodge correspondence over singular K\"ahler spaces]
{Non-abelian Hodge correspondence over singular K\"ahler spaces}

\author{Chuanjing Zhang}
\address{School of Mathematics and Statistics,
	Ningbo University, Ningbo, 315211, P.R. China}
\email{{\tt zhangchuanjing@nbu.edu.cn}}

\author{Shiyu Zhang}
\address{School of Mathematical Sciences, University of Science and Technology of China,
	Hefei, 230026, P.R. China}
\email{{\tt shiyu123@mail.ustc.edu.cn}}

\author{Xi Zhang}
\address{School of Mathematics and Statistics,
	Nanjing University of Science and Technology,
	Nanjing, 210094, P.R.China}
\email{{\tt mathzx@njust.edu.cn}}

\date{\today}


\begin{document}
	
	\begin{abstract}
		In this paper, we establish the non-abelian Hodge correspondence over compact K\"ahler spaces with Kawamata log terminal (klt) singularities as well as over their regular loci, thereby extending the result of Greb-Kebekus-Peternell-Taji for projective klt varieties \cite{GKPT19,GKPT20} to the context of compact K\"ahler klt spaces.
		
		The proof relies on two key ingredients: first, we establish an equivalence over the regular loci--via harmonic bundles--between polystable Higgs bundles with vanishing orbifold Chern numbers and semisimple flat bundles; second, we prove a descent (resp. ascent) result for semistable locally free Higgs sheaves with vanishing Chern classes along resolutions of singularities.
		
		As an application of our framework, we obtain a quasi-uniformization theorem for projective klt varieties with big canonical divisor that satisfy the orbifold Miyaoka-Yau equality in the sense of \cite{IJZ25}.
	\end{abstract}
	
	\maketitle
	\tableofcontents

	\section{Introduction}	 
    
    The non-abelian Hodge correspondence stems from the work of Corlette, Donaldson, Hitchin and Simpson \cite{Cor88,Don87,HIT,Simpson88} and can be summarized by the following equivalences:
    \[
    \begin{tikzcd}[column sep=large, row sep=large]
    	\left\{\substack{\text{polystable Higgs bundles}\\ \text{with vanishing Chern numbers}}\right\} \arrow[rr,leftrightarrow, "\text{Hitchin-Simpson}"] \arrow[d,leftrightarrow,"\text{Non-abelian Hodge correspondence}"] & & 
    	\left\{\text{harmonic bundles}\right\} \arrow[d, leftrightarrow,"\text{Corlette-Donaldson}"']   \\
    	\left\{\substack{\text{semisimple linear representations}\\ \text{from }\pi_1(X)\text{ to }\mathrm{GL}(r,\mathbb{C})}\right\} \arrow[rr, leftrightarrow, "\text{Riemann-Hilbert}"'] &  &
    	\left\{\text{semisimple flat bundles}\right\}
    \end{tikzcd}
    \]
    In general, Simpson \cite{SIM2} established an equivalence between the category of semistable Higgs sheaves with vanishing Chern numbers and the category of linear representations of the fundamental group over a projective manifold $X$. This equivalence also holds for compact K\"ahler manifolds \cite{NZ,Deng21} and has been significantly generalised to other settings, such as quasi-K\"ahler manifolds \cite{Simpson90,Jost-Zuo,Mo0,Mo1,Mo2} and varieties over fields of positive characteristic \cite{Fal05,OV07,AGT16,LSZ19,Langer23,SHW24}. We refer to \cite{Li19,Huang20} for an overview of its development.
    
    The minimal model program (MMP) for K\"ahler spaces has advanced substantially in recent years (see e.g. \cite{HP16,CHP16,CHP23,DO23,DHP24,Fuji22}). A key feature is that running the MMP typically introduces singularities, even when starting from a smooth K\"ahler space. Consequently, establishing a singular version of the non-abelian Hodge correspondence becomes a fundamental problem.
    
    The main results of this paper establish the non-abelian Hodge correspondence for compact K\"ahler spaces with Kawamata log terminal (klt) singularities, as well as over their regular loci. This work is motivated by two recent lines of progress: first, the successful construction of the correspondence for projective klt varieties \cite{GKPT19,GKPT20}; and second, the recent characterization of stable reflexive sheaves satisfying the orbifold Bogomolov-Gieseker inequality on compact K\"ahler klt spaces \cite{ou25,OF25} (see also \cite{GP24,GP25}). A further motivation comes from applications to the uniformization of K\"ahler klt spaces, extending analogous developments in the projective setting (cf. \cite{GKPT20,CGG24,Patel23,IMM24,Muller26} and references therein). 
    
    \subsection{Main results}
    
    General definitions and conventions follow \cite{ZZZ25}.
    
    \begin{notation}[Global notations]
    	For a normal analytic space $X$, we denote by $X_\reg$ its regular locus. Let \( \mathcal{E}_{X_{\text{reg}}} \) be a coherent sheaf defined only on $X_\reg$. We denote by \( \mathcal{E}_X \) its trivial extension to the whole space. Conversely, if \( \mathcal{E}_X \) is a sheaf over \( X \), we write \( \mathcal{E}_{X_{\text{reg}}} \) for its restriction to the regular locus. For a morphism \( f: Y \to X \) and a coherent sheaf \( \mathcal{E}_X \) over \( X \), we write \( f^{[*]} \mathcal{E}_X := (f^* \mathcal{E}_X)^{\vee\vee} \) for the reflexive pullback (i.e., the double dual of the pullback).
    \end{notation}
    
    \begin{definition}\label{defn-locallyfree-categories}
    	Let $X$ be a compact K\"ahler klt space of dimension $n$. We define
    	\begin{itemize}
    		\item $\mathrm{Higgs}_{X}$: locally free Higgs sheaves $(\mE_{X},\theta_{X})$ of rank $r$ over $X$ such that $(\mE_{X},\theta_{X})$ is $(\w_0,\cdots,\w_{n-2})$-semistable with vanishing Chern classes
    		\begin{equation}\label{equa-vanishingcondition-locallyfree}
    			\ch_1(\mE_X)\cdot[\w_0]\cdots[\w_{n-2}]=\ch_2(\mE_X)\cdot[\w_1]\cdots[\w_{n-2}]=0
    		\end{equation}
    		for some K\"ahler forms $\w_0,\cdots,\w_{n-2}$.
    		\item $\mathrm{LSys}_{X}$: local systems of rank $r$ on $X$.
    	\end{itemize}
    	Analogously, we can define the categories $\mathrm{pHiggs}_{X}$ and $\mathrm{sLSys}_{X}$ for polystable locally free Higgs sheaves satisfying \eqref{equa-vanishingcondition-locallyfree} and for semisimple local systems, respectively.	
    \end{definition}
    The first main result of this paper establishes the non-abelian Hodge correspondence for compact K\"ahler klt spaces.
	\begin{theorem}\label{main-theorem-locallyfree}
		Let $X$ be a compact K\"ahler klt space. There exists a natural one-to-one correspondence
		\[
		\mu_X: \mathrm{Higgs}_X \longrightarrow \mathrm{LSys}_X,
		\]
		which is compatible with resolutions of singularities in the following sense. For any resolution $\pi: \widehat{X} \to X$ and any $(\mathcal{E}_X,\theta_X) \in \mathrm{Higgs}_X$, the pullback $\pi^*(\mathcal{E}_X,\theta_X)$ belongs to $\mathrm{Higgs}_{\widehat{X}}$ and the diagram
		\begin{equation}\label{equa-diagram-1}
			\begin{tikzcd}
				\mathrm{Higgs}_X \arrow[r, "\mu_X"] \arrow[d, "\pi^*"'] & \mathrm{LSys}_X \arrow[d, "\pi^*"] \\
				\mathrm{Higgs}_{\widehat{X}} \arrow[r, "\mu_{\widehat{X}}"'] & \mathrm{LSys}_{\widehat{X}}
			\end{tikzcd}
		\end{equation}
		commutes. Here $\mu_{\widehat{X}}$ denotes the classical non-abelian Hodge correspondence for the compact K\"ahler manifold $\widehat{X}$. Moreover, the restriction of $\mu_X$ to $\mathrm{pHiggs}_X$ yields a one-to-one correspondence between $\mathrm{pHiggs}_X$ and $\mathrm{sLSys}_X$, which is likewise compatible with any resolution of singularities.
	\end{theorem}
	
	\begin{remark}
		We refer to Remark \ref{remark-pullback-Higgs} for the definition of the pullback $\pi^*(\mathcal{E}_X,\theta_X)$.
	\end{remark}
	
    It is necessary to consider a separate correspondence $\mu_{X_{\mathrm{reg}}}$ on the regular locus $X_\reg$ of a compact K\"ahler klt space \(X\), because the homomorphism \(i_{*}: \pi_{1}(X_{\mathrm{reg}}) \to \pi_{1}(X)\) induced by the inclusion map $i_*:X_\reg\rightarrow X$ is surjective (cf.~\cite[Proposition 2.10]{Kollar14}) but generally not injective.
	
	\begin{definition}\label{defn-categories}
		Let $X$ be  a compact K\"ahler klt space $X$ of dimension $n$. We define
		\begin{itemize}
			\item $\mathrm{Higgs}_{X_{\reg}}$: reflexive Higgs sheaves $(\mE_{X_{\reg}},\theta_{X_{\reg}})$ of rank $r$ over $X_{\reg}$ such that $(\mE_{X_{\reg}},\theta_{X_{\reg}})$ is $(\w_0,\cdots,\w_{n-2})$-semistable with vanishing orbifold Chern classes
			\begin{equation}\label{equa-vanishingcondition}
				\widehat{\ch}_1(\mE_X)\cdot[\w_0]\cdots[\w_{n-2}]=\widehat{\ch}_2(\mE_X)\cdot[\w_1]\cdots[\w_{n-2}]=0
			\end{equation}
			for some K\"ahler forms $\w_0,\cdots,\w_{n-2}$.
			\item $\mathrm{LSys}_{X_{\reg}}$: local systems of rank $r$ on $X_\reg$.
		\end{itemize}
		Analogously, we can define the categories $\mathrm{pHiggs}_{X_{\reg}}$ and $\mathrm{sLSys}_{X_{\reg}}$ for polystable reflexive Higgs sheaves satisfying \eqref{equa-vanishingcondition} and semisimple local systems.	
	\end{definition}
	
	\begin{remark}
		Our definitions employ the orbifold Chern classes, denoted $\widehat{\mathrm{ch}}_1$ and $\widehat{\mathrm{ch}}_2$, which coincide with $\mathbb{Q}$-Chern classes upon intersection with ample divisors. In particular, the definition of $\mathrm{Higgs}_{X_\reg}$ coincides with that in \cite{GKPT20} when $[\w_0],\cdots,[\w_{n-2}]$ are ample. We refer to \cite[Section 2]{ou25} for the definitions of orbifold Chern classes (see also \cite[Section 5]{GP25}).
	\end{remark}
	
	The second main result of this paper establishes the non-abelian Hodge correspondence for the regular loci of compact K\"ahler klt spaces.
	
	\begin{theorem}\label{main-theorem-reflexive}
		Let $X$ be a compact K\"ahler klt space. There exists a one-to-one correspondence
		\[
		\mu_{X_{\reg}}: \mathrm{Higgs}_{X_{\reg}} \longrightarrow \mathrm{LSys}_{X_{\reg}},
		\]
		such that the underlying $C^\infty$-bundle of $\mathcal{E}_{X_{\reg}}$ admits a flat connection induced by $\rho_{X_{\reg}}$. Moreover, this correspondence is compatible with the correspondence $\mu_X$ in the following sense: for any maximally quasi-\'etale cover $\gamma: Y \to X$, the reflexive pullback $\gamma^{[*]}(\mathcal{E}_{X_{\reg}},\theta_{X_{\reg}})$ extends to a locally free Higgs sheaf $(\mathcal{E}_Y, \theta_Y)$ belonging to $ \mathrm{Higgs}_Y$, and the diagram
		\begin{equation}\label{equa-diagram-2}
			\begin{tikzcd}
				\mathrm{Higgs}_{X_{\reg}} \arrow[r, "\mu_{X_{\reg}}"] \arrow[d, "\gamma^*"'] & \mathrm{LSys}_{X_{\reg}} \arrow[d, "\gamma^*"] \\
				\mathrm{Higgs}_{Y} \arrow[r, "\mu_{Y}"'] & \mathrm{LSys}_{Y}
			\end{tikzcd}
		\end{equation}
		commutes. The restriction of \(\mu_{X_{\mathrm{reg}}}\) to \(\mathrm{pHiggs}_{X_{\mathrm{reg}}}\) gives the bijection from $\mathrm{pHiggs}_{X_\reg}$ to $\mathrm{sLSys}_{X_\reg}$ that is compatible with every maximally quasi-\'etale cover; see Theorem \ref{main-theorem-polystable-equivalence} for a precise statement.
	\end{theorem}
	
	\begin{remark}\label{remark-reflexive-pullback}
		\(\gamma^{[*]}(\mathcal{E}_{X_{\mathrm{reg}}},\theta_{X_{\mathrm{reg}}})\) denotes the reflexive pullback (cf. Lemma~\ref{lem-reflexive-pullback}) and we refer to Section \ref{section-formulation} for the definition of the map $\gamma^*:\mathrm{Higgs}_{X_\reg}\rightarrow \mathrm{Higgs}_Y$ (resp. $\gamma^*:\mathrm{LSys}_{X_\reg}\rightarrow \mathrm{LSys}_Y$). As in \cite{GKPT19,GKPT20}, \emph{maximally quasi-\'etale covers} take a central role in the construction of $\mu_X$ and $\mu_{X_\reg}$. We refer to Section~\ref{subsubsection-maximally} for their definition.
	\end{remark}
	
	Together, Theorems \ref{main-theorem-locallyfree} and \ref{main-theorem-reflexive} provide the analytic counterparts of the two correspondences established for projective klt varieties in \cite{GKPT19} and \cite{GKPT20}, respectively.
	
	\subsection{Applications to quasi-uniformization problems}
	
	We now present several applications of our main results. First, we give a numerical characterization of singular quotients of complex tori, already established in \cite{CGG24,GP25} when $[\w_1]=\cdots=[\w_{n-2}]$ (see \cite{LT18,GKP16b,CGG22} for earlier developments). The smooth version of this result originates in Yau’s solution of the Calabi conjecture \cite{Yau78}.
	
	\begin{corollary}\label{coro-torus}
		Let \(X\) be a compact K\"ahler klt space of dimension \(n\) with \(K_X\) numerically trivial. Then
		\begin{equation}\label{equa-miyaoka2}
			\widehat{c}_2(X)\cdot[\omega_1]\cdots[\omega_{n-2}]\geq0,
		\end{equation}
		with equality if and only if \(X\) is a singular quotient of a complex torus \(\mathbb{T}^n\); i.e., there exists a finite group \(G\subset\operatorname{Aut}(\mathbb{T}^n)\) acting freely in codimension one such that \(X \cong \mathbb{T}^n / G\).
	\end{corollary}
	
	 When \(-K_X\) is nef, inequality \eqref{equa-miyaoka2}---known as Miyaoka's inequality, originating from \cite{Miyaoka87}---continues to hold in the K\"ahler klt setting (cf. \cite[Theorem 1.5]{IJZ25}). Its equality case has been characterized for compact K\"ahler manifolds \cite{Ou23,IM22,MMWZ25} and for projective klt varieties \cite{IMM24}. Extending these results to compact K\"ahler klt spaces requires a characterization of the semistable reflexive Higgs sheaves that satisfy the orbifold Bogomolov-Gieseker equality \cite[Question 7.5]{IJZ25}, which has been proved for stable reflexive sheaves in \cite{OF25} (see also \cite{GP24,GP25}). We now establish such a characterization based on Theorems~\ref{main-theorem-locallyfree} and~\ref{main-theorem-reflexive}, combined with the preservation of semistability under tensor products proved in \cite{ZZZ25}.
	\begin{corollary}\label{coro-projflat}
		Let \(X\) be a compact K\"ahler klt space and let \(\gamma:Y\rightarrow X\) be a maximally quasi-\'etale cover. Suppose \((\mathcal{E}_{X_{\mathrm{reg}}},\theta_{X_{\mathrm{reg}}})\) is a reflexive Higgs sheaf of rank \(r\) on the regular locus, semistable with respect to some K\"ahler forms \(\omega_0, \dots, \omega_{n-2}\). Then the following orbifold Bogomolov-Gieseker inequality holds:
		\begin{equation}\label{equa-BG-equality}
			\Bigl(2\widehat{c}_2(\mathcal{E}_X)-\frac{r-1}{r}\widehat{c}_1^2(\mathcal{E}_X)\Bigr)\cdot[\omega_1]\cdots[\omega_{n-2}]\geq0.
		\end{equation}
		Equality in \eqref{equa-BG-equality} holds if and only if the sheaf \(\mathcal{G}_X:=\gamma^{[*]}\operatorname{End}(\mathcal{E}_X)\) is locally free and the induced Higgs sheaf \((\mathcal{G}_X,\theta_{\mathcal{G}})\) corresponds, via the correspondence \(\mu_X\) of Theorem~\ref{main-theorem-locallyfree}, to a linear representation \(\rho: \pi_1(Y)\to\operatorname{GL}(r,\mathbb{C})\).
	\end{corollary}
	
	As another important application, we prove a quasi‑uniformization theorem for projective klt varieties with big canonical divisor. The Miyaoka-Yau inequality, originating in \cite{Miyaoka77,Miyaoka87}, is a cornerstone in the geometry of surfaces of general type; see the introductions of \cite{GKPT19a,GT22} for an overview. For a minimal projective klt variety \(X\) of general type, the authors of \cite{GKPT19a} established the orbifold Miyaoka-Yau inequality
	\begin{equation}\label{equa-miyaoka1}
		\bigl(2\widehat{c}_{2}(X)-\tfrac{n}{n+1}\widehat{c}_{1}^{2}(X)\bigr)\cdot K_X^{n-2}\ge 0.
	\end{equation}
    Moreover, the equality case of \eqref{equa-miyaoka1} yields a quasi‑uniformization description \cite{GKPT20}: the canonical model of \(X\) is a singular quotient of the unit ball. Making use of the framework developed in this paper, we now extend this picture to varieties whose canonical divisor is {\em big, but not necessarily nef}.
	
	\begin{theorem}\label{main-thm-unitball}
		Let \(X\) be a projective klt variety of dimension \(n\) with \(K_X\) big. The following orbifold Miyaoka-Yau type inequality holds:
		\begin{equation}\label{equa-MY}
			\bigl(2\widehat{c}_2(X)-\tfrac{n}{n+1}\widehat{c}_1^2(X)\bigr)\cdot\langle K_X^{n-2}\rangle \ge 0.
		\end{equation}
		If equality in \eqref{equa-MY} holds, then the canonical model \(X_{\mathrm{can}}\) of \(X\) is a singular quotient of a projective manifold \(Y\) that is covered by the unit ball \(\mathbb{B}^n\).
	\end{theorem}
	
	The inequality (1.8) was established in the recent work \cite{IJZ25}--a result that is new even in
	the smooth case. Here \(\langle K_X^{n-2}\rangle\) denotes the non-pluripolar product; the precise definition of the intersection number appearing in (1.8) is given in \cite[Section~4.1]{IJZ25}. Together with the very recent preprint \cite{Jinnouchi25-2}---which proves that for a \(K\)-stable projective klt variety with big anti-canonical divisor, the Miyaoka-Yau type equality forces its anti-canonical model to be a singular quotient of \(\mathbb{P}^n\)---these results confirm that the inequality \eqref{equa-MY} developed in \cite{IJZ25} provides an effective criterion for projective klt varieties with big canonical or anti-canonical divisors, and completes the picture

	\subsection{Overview and main technical components}\label{subsection-strategy}
	We now outline the proofs of Theorems~\ref{main-theorem-locallyfree} and~\ref{main-theorem-reflexive} and present the principal differences from the projective case treated in \cite{GKPT19,GKPT20}.
	
	Let \(X\) be a compact K\"ahler klt space. Following the strategy of \cite{GKPT19}, a key step in constructing \(\mu_X\) is to show that for any \((\mathcal{E}_X, \theta_X) \in \operatorname{Higgs}_X\) and any resolution \(\pi: \widehat{X} \to X\), the pullback \(\pi^*(\mathcal{E}_X, \theta_X)\) lies in \(\operatorname{Higgs}_{\widehat{X}}\). In Section~\ref{subsection-JHfiltrations} we prove the following.
	
	\begin{theorem}[cf. Theorems~\ref{thm-chernclassesvanishing} and~\ref{thm-indenpendent-polarization}]\label{main-maximallyquasietalecase}
		Let \(X\) be a compact K\"ahler klt space and \((\mathcal{E}_{X_{\mathrm{reg}}},\theta_{X_{\mathrm{reg}}})\) a reflexive Higgs sheaf over \(X_{\mathrm{reg}}\).
		\begin{itemize}
			\item[(1)] \((\mathcal{E}_{X_{\mathrm{reg}}},\theta_{X_{\mathrm{reg}}})\in\mathrm{Higgs}_{X_\mathrm{reg}}\) (resp. \(\mathrm{pHiggs}_{X_\mathrm{reg}}\)) if and only if for every maximally quasi-\'etale cover \(\gamma:Y\rightarrow X\), \(\mE_Y:=\gamma^{[*]}\mathcal{E}_X\) is locally free and \(\gamma^{[*]}(\mathcal{E}_{X_{\mathrm{reg}}},\theta_{X_{\mathrm{reg}}})\) extends to a locally free Higgs sheaf \((\mathcal{E}_Y,\theta_Y)\in\mathrm{Higgs}_Y\) (resp. \(\mathrm{pHiggs}_{Y}\)). Moreover, \((\mathcal{E}_Y,\theta_Y)\) is an extension of locally free Higgs sheaves belonging to \(\mathrm{pHiggs}_Y\).
			\item[(2)] If \(\mathcal{E}_X\) is locally free, then \((\mathcal{E}_X,\theta_X)\in\mathrm{Higgs}_X\) (resp. \(\mathrm{pHiggs}_{X}\)) if and only if for every resolution of singularities \(\pi:\widehat{X}\rightarrow X\), the pullback \(\pi^*(\mathcal{E}_X,\theta_X)\in\mathrm{Higgs}_{\widehat{X}}\) (resp. \(\mathrm{pHiggs}_{\widehat{X}}\)).
		\end{itemize}
	\end{theorem}
	
	The main task is to prove that for any Jordan-H\"older filtration \(\{\mathcal{E}_i\}_{i=0}^l\) of \(\gamma^{[*]}(\mathcal{E}_{Y_\mathrm{reg}}, \theta_{Y_\mathrm{reg}})\),  each quotient sheaf \(\mathcal{E}_i/\mathcal{E}_{i-1}\) is locally free and has vanishing Chern classes. Following the spirits of \cite[Proposition 7.10]{GKPT19}, \cite[Theorem 1.2]{HIM22} and \cite[Proposition 3.10]{IMM24}, the problem in our setting reduces to proving the following statement: if \((\mathcal{Q}_{Y_{\mathrm{reg}}}, \theta_{\mathcal{Q}_{Y_{\mathrm{reg}}}}) \in \operatorname{pHiggs}_{Y_{\mathrm{reg}}}\) is stable, then its reflexive extension \(\mathcal{Q}_Y\) is locally free and all its Chern classes vanish. In the projective case \cite[Proposition 7.10]{GKPT19}, this is handled by cutting down with hypersurface sections to reduce to surfaces, where reflexive sheaves are automatically locally free---a technique that is not available for general K\"ahler spaces. The central difficulty in the K\"ahler setting is therefore precisely to establish the local freeness of \(\mathcal{Q}_Y\) in higher dimensions. 
	
	In Sections~\ref{section-HE} and~\ref{section-harmonicbundles}, we concentrate on proving the following characterization of polystable reflexive Higgs sheaves satisfying \eqref{equa-vanishingcondition}. This result not only yields the desired local freeness of $\mQ_Y$ but also constructs the correspondence $\mu_{X_\mathrm{reg}}: \mathrm{pHiggs}_{X_\mathrm{reg}} \to \mathrm{sLSys}_{X_\mathrm{reg}}$.
	
	\begin{theorem}\label{main-theorem-polystable-equivalence}
		Let \(X\) be a compact K\"ahler klt space and $\gamma:Y\rightarrow X$ be a maximally quasi-\'etale cover.
		\begin{itemize}
			\item[(1)] A flat bundle \((E_{X_\mathrm{reg}},D_{X_\mathrm{reg}})\) is semisimple if and only if it admits a pluri-harmonic metric \(H_{X_\mathrm{reg}}\), which is unique up to a positive constant on each direct summand.
			\item[(2)] A reflexive Higgs sheaf \((\mathcal{E}_{X_{\mathrm{reg}}},\theta_{X_{\mathrm{reg}}})\) belongs to \(\mathrm{pHiggs}_{X_\mathrm{reg}}\) if and only if \(\mathcal{E}_{X_\mathrm{reg}}\) is locally free and \((\mathcal{E}_{X_{\mathrm{reg}}},\theta_{X_{\mathrm{reg}}})\) arises from a harmonic bundle \((E_{X_\mathrm{reg}},D_{X_\mathrm{reg}},H_{X_\mathrm{reg}})\); here \(H_{X_\mathrm{reg}}\) is unique up to obvious ambiguity. This is also equivalent to any one of the following three statements:
			\begin{enumerate}
				\item The reflexive pullback \(\gamma^{[*]}(\mathcal{E}_{X_{\mathrm{reg}}},\theta_{X_{\mathrm{reg}}})\) extends to a locally free Higgs sheaf \((\mathcal{E}_Y,\theta_Y)\in\mathrm{pHiggs}_Y\) with vanishing Chern classes.
				\item For any resolution of singularities $\pi:\widehat{Y}\rightarrow Y$, $\pi^*(\mathcal{E}_Y,\theta_Y)\in\mathrm{pHiggs}_{\widehat{Y}}$.
				\item \((\mathcal{E}_{X_{\mathrm{reg}}},\theta_{X_{\mathrm{reg}}})\) is polystable with respect to any K\"ahler forms \(\omega_0,\cdots,\omega_{n-2}\) and all orbifold Chern classes of \(\mathcal{E}_X\) vanish.
			\end{enumerate}
		\end{itemize}
	\end{theorem}
	
	The proof of Theorem~\ref{main-theorem-polystable-equivalence} addresses this issue and proceeds in three steps.
	
	\begin{itemize}
		\item[\textbf{Step 1.}] In Section~\ref{section-HE}, we prove the existence of a pluri\-harmonic metric for reflexive Higgs sheaves. This result was previously established in \cite{OF25} for the case where \(\theta_{X_{\mathrm{reg}}}\) is trivial (see also \cite{GP24,GP25}). Our proof adapts the Higgs version \cite{LZZ} of Bando-Siu's argument \cite{BS} to study the Hermitian-Yang-Mills flow \(\{H_{\mathrm{orb}}(t)\}\) on an orbifold modification constructed in \cite{KO25}. Combining this with \eqref{equa-vanishingcondition}, we conclude that the limiting Hermitian-Einstein metric is pluri-harmonic. The overall strategy are presented in Section~\ref{Strategy-HYM}.
		\item[\textbf{Step 2.}] In the first part of Section~\ref{subsection-harmonic}, we prove part (1) of the theorem. The existence of pluri-harmonic metrics on semisimple flat bundles over the regular locus follows from the work of Corlette and Donaldson \cite{Cor88,Don87} via maximally quasi‑\'etale covers; the remaining properties are obtained through the maximum principle, which is well adapted to our setting thanks to the extension theorem of Grauert and Remmert \cite{GR55}.
		\item[\textbf{Step 3.}] From Step 1 and Theorem~\ref{main-theorem-polystable-equivalence} (1), we conclude that \(\gamma^{[*]}(\mathcal{E}_{X_\mathrm{reg}},\theta_{X_\mathrm{reg}})\) admits a pluri-harmonic metric induced by a semisimple local system \(E_{Y_\mathrm{reg}}\). Since \(Y\) is maximally quasi-\'etale, \(E_{Y_\mathrm{reg}}\) extends to a semisimple local system \(E_Y\). To verify that \(\gamma^{[*]}(\mathcal{E}_{X_\mathrm{reg}},\theta_{X_\mathrm{reg}})\) extends to a locally free Higgs sheaf on \(Y\), it suffices to establish a descent result for Higgs bundles arising from harmonic bundles (Proposition~\ref{prop-descent-higgs}), which we prove using the period map constructed in \cite{Daniel17}. Details are given in Sections~\ref{subsection-harmonic} and~\ref{subsection-locallyfree}.
	\end{itemize}
	
	\begin{remark}	
		The careful treatment of harmonic bundles on the regular locus presented above is necessary in the K\"ahler klt context because Mochizuki's theory of tame harmonic bundles on quasi-projective manifolds \cite{Mo0, Mo1, Mo2}---which crucially underlies the construction of $\mu_{X_{\mathrm{reg}}}$ in the projective case \cite{GKPT20}---has not yet been established for general quasi-K\"ahler manifolds.
		
		It should be noted that the admissibility of the limiting Hermitian-Einstein metric constructed in Section~\ref{section-HE} remains unknown (cf. Question~\ref{ques-uniformestimate}). The key point of Steps~2 and~3 is that klt singularities are sufficiently mild to allow us, via harmonic bundles on the regular locus, to characterize polystable Higgs bundles satisfying \eqref{equa-vanishingcondition} and semisimple local systems on the regular locus, without imposing any extra assumptions on the pluri\-harmonic metrics.
	\end{remark}
	
	Building on Theorem~\ref{main-stable}, we confirm that each \((\mathcal{E}_i/\mathcal{E}_{i-1})^{\vee\vee}\) is locally free. Then, in Section~\ref{subsection-JHfiltrations}, we complete the proof of Theorem~\ref{main-maximallyquasietalecase}. This allows us to construct \(\mu_X\) and \(\mu_{X_{\mathrm{reg}}}\) in Section~\ref{section-formulation}. It remains to prove that both maps are injective; this relies on proving the following descent result as in \cite{GKPT19}.
	
	\begin{theorem}\label{main-thm-descent}
		Let \(f:Z\rightarrow X\) be a bimeromorphic holomorphic map between compact K\"ahler klt spaces. Then for any \((\mathcal{E}_Z,\theta_Z)\in\mathrm{Higgs}_Z\), there exists \((\mathcal{E}_X,\theta_X)\in\mathrm{Higgs}_X\) such that \((\mathcal{E}_Z,\theta_Z)=f^*(\mathcal{E}_X,\theta_X)\).
	\end{theorem}
	
	Theorem~\ref{main-thm-descent} was proved for locally free \(f\)-numerically flat sheaves in the projective case \cite[Theorem 1.2]{GKPT20}, building on the relative MMP for projective morphisms \cite{BCHM}. Its extension to the K\"ahler klt setting appears plausible, especially in light of recent work extending the relative MMP to the analytic context \cite{DHP24,Fuji22}. However, several technical results used in the proof of \cite{GKPT19}---notably \cite[Proposition 3.1]{HN13}, which appears in the proof of \cite[Claim 4.5]{GKPT19}---do not seem to be directly available in the analytic setting, to the best of our knowledge and based on the existing literature. Moreover, in order to maintain our focus on Higgs sheaves and to keep the paper concise, we adopt an alternative argument in Section~\ref{subsection-descent}. This argument proceeds via an orbifold modification to factor through local systems, employs Jordan-H\"older filtrations of the corresponding representation, and reconstructs the Higgs sheaf by descending the graded stable pieces based on Proposition \ref{prop-descent-higgs}.
	
	Although our overarching goal---proving Theorems~\ref{main-maximallyquasietalecase} and~\ref{main-thm-descent}---coincides with that of \cite{GKPT19}, the proofs we present in the K\"ahler setting are substantially different for the reasons outlined above.

    \vspace{0.3cm}
    \textbf{Acknowledgement.} The authors thank Xin Fu for insightful discussions on uniform geometric estimates, Osamu Fujino for kind help and sharing the preprint \cite{Fujino23} with them, Changpeng Pan and Di Wu for helpful conversations regarding harmonic bundles, and Masataka Iwai and Satoshi Jinnouchi for valuable remarks on quasi-uniformization problems. The second author also thanks Tianshu Jiang and Shin-ichi Matsumura for conversations in related topics. The research was supported by the National Key R and D Program of China 2020YFA0713100. The  authors are partially supported by NSF in China No.12141104, 12371062 and 12431004.
    
	\section{Fundamental materials}\label{fundamentalmaterial}
	
	This section reviews and proves necessary statements concerning complex orbifolds (Section~\ref{subsection-orbifolds}), Higgs sheaves (Section~\ref{subsection-Higgssheaves}), orbifold Chern classes (Section~\ref{subsection-orbifoldchernclass}) and quasi-\'etale morphisms (Section~\ref{subsection-quasietale}).

	\subsection{Complex orbifolds}\label{subsection-orbifolds} The study on orbifolds originated from \cite{Satake}. For a concise discussion to complex orbifolds, we refer to \cite[Section 3]{DO23} and \cite{Wu23}.

	 A complex space $X$ with quotient singularities equipped with an effective orbifold structure $X_\orb:=\{(V_\alpha,G_\alpha,\mu_\alpha)\}$, namely, there exists an open cover $\{X_\alpha\}$ of $X$ such that $X_\alpha\cong V_\alpha/G_\alpha$ for some complex manifold $V_\alpha$ and $G_\alpha\subset \Aut(V_\alpha)$ and $(V_\alpha,G_\alpha,\mu_\alpha)$ are compatible along the overlaps, where $\mu_\alpha:V_\alpha\rightarrow X_\alpha$ is the quotient map induced by $G_\alpha$. Effective means that every $G_\alpha$-action is faithful. Indeed, $X$ always admits a standard orbifold structure, i.e., $\pi_\alpha$ is \'etale in codimension $1$. In virtue of purity theorem of branch locus due to Grauert-Remmert \cite{GR55}, it can be easily seen that
		\begin{lemma}[c.f. {\cite[Proposition 1]{Wu23}}]\label{lem-maximalorbifold}
			Let $\{(V_\alpha,G_\alpha,\mu_\alpha)\}$ and $\{(V_\beta,G_\beta,\mu_\beta)\}$ be two standard orbifold structures. Then $V_{\alpha,\beta}:=\{(z_\alpha,z_\beta)\in V_\alpha\times V_\beta:\mu_\alpha(z_\alpha)=\mu_{\beta}(z_\beta)\}$ is smooth and the natural projections $V_{\alpha,\beta}\rightarrow V_\alpha$ and $V_{\alpha,\beta}\rightarrow V_\beta$ are \'etale.
		\end{lemma}
		 The definition gives $V_{\alpha,\beta}/(G_\alpha\times G_\beta)\cong X_\alpha\cap X_\beta$. Then according to Zorn's Lemma, there exists a standard orbifold structure that is maximal. By a standard orbifold structure we will always mean the maximal one.

	\begin{lemma}[cf. {\cite[Lemma 3.10]{DO23}}]\label{lem-orbireso}
		Let $W_{orb}=\{(V_i,G_i,\mu_i)\}$ be a compact complex orbifold, and $\mE_{orb}=\{\mE_i\}$ be a torsion-free orbi-sheaf on $W_{orb}$. Then there exists an orbifold morphism $h_{orb}:Z_{orb}\rightarrow W_{orb}$ from a compact complex orbifold $Z_{orb}=\{(\widehat{V}_i,G_i,\widehat{\mu}_i)\}$ so that the following holds:
		\begin{itemize}
			\item[(1)] Set $h_{orb}:=\{h_i:\widehat{V}_i\rightarrow V_i\}$, then $h_i$ is obtained by a sequence of blowups of $G_i$-invariant smooth centers contained in the non-locally-free locus of $\mE_i$.
			\item[(2)] The coherent orbi-sheaf $E_\orb:=\{h_i^*\mE_i/(\mathrm{torsion})\}$ is a vector orbi-bundle.
			\item[(3)] If $W_{orb}$ is K\"ahler, then $Z_{orb}$ is also K\"ahler.
		\end{itemize}
	\end{lemma}
	
	All objects (e.g., orbi sheaves, orbifold differential forms, etc.) in the context of orbifolds are those defined on each $V_\alpha$ that satisfying compatibility conditions. For a concise introduction to Higgs orbi-sheaves, we refer to \cite[Section 2]{ZZZ25}.

	\subsection{Higgs sheaves}\label{subsection-Higgssheaves}
	We refer to \cite[Section 4]{ZZZ25} for a concise discussion on Higgs sheaves over the regular locus of a compact complex normal space $X$. All discussions are still valid for locally free Higgs sheaves over $X$ by the following Remark \ref{remark-equivariant}.

	\subsubsection{Higgs sheaves over the whole space}

	We follow \cite[Section 5]{GKPT19a} to introduce Higgs sheaves over normal spaces (see also \cite[Section 3.2.2]{IJZ25}).
	\begin{definition}[Higgs sheaves, {cf. \cite[Definition 5.1]{GKPT19a}}]
		Let $X$ be a normal space. A Higgs sheaf $(\mE_X,\theta_X)$ on $X$ consists of a coherent sheaf and a morphism $\theta_X:\mE_X\rightarrow \mE_X\otimes\Omega_X^{[1]}$ such that the composed morphism
		$$\theta\wedge\theta:\mE_X\xrightarrow{\theta}\mE_X\otimes \Omega_X^{[1]}\xrightarrow{\theta\otimes\Id} \mE_X\otimes \Omega_X^{[1]}\otimes \Omega_X^{[1]}\xrightarrow{\Id\otimes[\wedge]} \mE_X\otimes \Omega_X^{[2]}$$
		vanishes. $\theta$ is called the Higgs field of $X$.
	\end{definition}
	
	\begin{lemma}\label{lem-Higgs-extension}
		Let  $\mE_X$ be a locally free sheaf on a normal space $X$. Suppose that $\theta_W$ is a Higgs field of $(\mE_X)|_W$ on an analytic subspace $W$ of codimension at least $2$, then $\theta_W$ extends to a Higgs field $\theta_X$ of $\mE_X$.
	\end{lemma}
	\begin{proof}
		This is because $\Hom(\mE_X,\mE_X\otimes\Omega_X^{[1]})$ and the target of the composed morphism $\theta\wedge\theta$ are reflexive.
	\end{proof}
	
	\begin{lemma}[cf. {\cite[Lemma 4.12]{GKPT19a}}]\label{lem-saturation-invariant}
		Let $(\mE_X,\theta_X )$ be a torsion-free Higgs sheaf on a normal space $X$. Suppose that $\mF_X$ is a subsheaf of $\mE_X$ and there exists a Zariski dense open subset $W\subset X_\reg$ such that $\mF_X$ is $\theta_{X}$-invariant on $W$, then its saturation $\mF^{sat}$ is generically $\theta_X$-invariant.
	\end{lemma}
	
	\begin{remark}\label{remark-saturation-stability}
		Recall that the saturation of a $\theta_{X_\reg}$-invariant torsion-free subsheaf increases the slope (see e.g. \cite[Lemma 2.30]{ZZZ25}) and is again $\theta_{X_\reg}$-invariant by Lemma \ref{lem-saturation-invariant}. Thus, for any nef polarization $(\alpha_0,\cdots,\alpha_{n-2})$, $(\mE_{X_\reg},\theta_{X_\reg})$ is $(\alpha_0,\cdots,\alpha_{n-2})$-stable if and only if 
		\begin{equation}\label{equa-defn-stability}
			\mu_{(\alpha_0,\cdots,\alpha_{n-2})}(\mF_{X})< \mu_{(\alpha_0,\cdots,\alpha_{n-2})}(\mE_{X}).
		\end{equation}
		holds for any saturated subsheaf $\mF_X\subsetneq \mE_X$ that is $\theta_{X_\reg}$-invariant. Smilar statements hold for semistability and polystability.
	\end{remark}
	
	Combining Lemma \ref{lem-saturation-invariant} and Remark \ref{remark-saturation-stability} shows that
	\begin{lemma}\label{lem-stable-equivariant}
		A torsion-free Higgs sheaf $(\mE_X,\theta_X)$ on a compact normal space $X$ is stable (resp. semistable) in the sense of \cite[Definition 4.13]{GKPT19a} if and only if $(\mE_X,\theta_X)|_{X_\reg}$ is stable (resp. semistable).
	\end{lemma}
	
	\begin{remark}\label{remark-equivariant}
		Building on Lemmas \ref{lem-Higgs-extension} and \ref{lem-stable-equivariant}, all statements proved in the remainder of this subsection also hold for torsion-free Higgs sheaves over the whole space.
	\end{remark}

	Following the spirit of \cite[Page 158]{Kobayashi2014}, it can be easily seen that
	
	\begin{lemma}\label{semistability-extension}
		Let $X$ be a compact complex normal space of dimension $n$, $\alpha_0,\cdots,\alpha_{n-2}$ be nef classes and $(\mE_{X_\reg},\theta_{\mE_{X_\reg}})$ be a torsion-free Higgs sheaf on $X_\reg$. Suppose that
		$$0\rightarrow \mF_X\rightarrow \mE_X\rightarrow\mQ_X\rightarrow 0$$
		is a short exact sequence of torsion-free sheaves such that $\mF_{X_\reg}$ is $\theta_{\mE_{X_\reg}}$-invariant. If the induced Higgs sheaves $(\mF_{X_\reg},\theta_{\mF_{X_\reg}}), (\mQ_{X_\reg},\theta_{\mQ_{X_\reg}})$ are $(\alpha_0,\cdots,\alpha_{n-2})$-semistable and $\mu_{(\alpha_0,\cdots,\alpha_{n-2})}(\mF_X)=\mu_{(\alpha_0,\cdots,\alpha_{n-2})}(\mE_X)$, then $(\mE_{X_\reg},\theta_{\mE_{X_\reg}})$ is $(\alpha_0,\cdots,\alpha_{n-2})$-semistable.
	\end{lemma}

	\subsubsection{Pullback of Higgs sheaves}\label{section-2-pullback}
	
	The existence of the pullback of Higgs sheaves is essential and relies on Kebekus-Schnell's work \cite{KS21} on constructing the pullback of reflexive differentials.
	\begin{lemma}[Reflexive pullback of Higgs field, cf. {\cite[Lemma 4.9]{ZZZ25}}]\label{lem-reflexive-pullback}
		Let $f:Y\rightarrow X$ be a holomorphism between complex spaces with klt singularities and $(\mE_{X_\reg},\theta_{X_\reg})$ be a reflexive Higgs sheaf on $X_\reg$. Let $\mE_{Y_\reg}:=(f^*\mE_X)^{\vee\vee}|_{Y_\reg}$ and $\Sigma$ be the non-locally-free locus of $\mE_{X_\reg}$. If $f^{-1}(X_\reg)\cap Y_\reg\neq \emptyset$, then $\mE_{Y_\reg}$ admits a Higgs field $\theta_{Y_\reg}$ such that $\theta_{Y_\reg}=f^*\theta_{X_\reg}$ on $f^{-1}(X_\reg\setminus\Sigma)$.
	\end{lemma}
	
	\begin{notation}[Reflexive pullback of Higgs sheaves]
		We call $(\mE_{Y_\reg},\theta_{Y_\reg})$ the {\em reflexive pullback} of $(\mE_{X_\reg},\theta_{X_\reg})$, and denote it by $f^{[*]}(\mE_{X_\reg},\theta_{X_\reg})$.
	\end{notation}
	
	\begin{remark}\label{remark-pullback-Higgs}
		If $\mE_X$ is locally free, then $\mE_Y=f^*\mE_X$ is locally free and Lemma \ref{lem-Higgs-extension} implies $\theta_{Y_\reg}$ extends to a Higgs field $\theta_Y$ of $\mE_Y$. We call $(\mE_Y,\theta_Y)$ the {\em pullback} of $(\mE_X,\theta_X)$, and denote it by $f^*(\mE_X,\theta_X)$.
	\end{remark}

	\subsection{Orbifold Chern classes}\label{subsection-orbifoldchernclass}
	 We refer to \cite[Section 2]{ou25} for a concise introduction to the definition of orbifold Chern classes via orbifold modifications \cite{KO25}. A compact complex klt space $X$ has only quotient singularities in codimension $2$ (cf. \cite[Lemma 5.8]{GK20}) and therefore admits an orbifold modification by \cite{KO25}, i.e., there exists a bimeromorphic map $f:Y\rightarrow X$ from a complex space with only quotient singularities to $X$ such that the indeterminacy locus of $f^{-1}$ has codimension at least $3$. This allows us to define the orbifold first and second Chern classes as follows.
	
	\begin{definition}[cf. {\cite[Section 2]{ou25}}]\label{defn-orbifold}
		Let $X$ be a compact complex klt space, $f:Y\rightarrow X$ be an orbifold modification and $Y_\orb=\{(V_i,G_i,\mu_i)\}$ be the standard orbifold structure of $Y$. For any reflexive sheaf $\mE$ on $X$, $\widehat{c}_2(\mE)\in H_{2n-4}(X,\R)$ is defined by
		$$\widehat{c}_2(\mE)\cdot\sigma:=c_2^\orb(f^{[*]}_\orb\mE)\cdot f^*\sigma,\ \forall\ \sigma\in H^{2n-4}(X,\R),$$
		where $f^{[*]}_\orb\mE=\{(f\circ \mu_i)^{[*]}\mE\}$ is a reflexive orbi-sheaf on $Y_\orb$. The classes $\widehat{c}_1(\mE)\in{H_{2n-2}(X,\R)}$ and $\widehat{c}_1^2(\mE)\in H_{2n-4}(X,\R)$ are defined analogously.
	\end{definition}

	Definition \ref{defn-orbifold} is independent of the choice of $f$. We refer to \cite[Section 2]{ou25} for a detailed discussion. When $\mE=\mT_X$, it  coincides with the definition of $\widehat{c}_2(X)$ in \cite{LT18,GK20,CGG24} (see also \cite{GP25}).
	
    \begin{remark}\label{remark-MTTW25}
    	Here, \(c_2^{\mathrm{orb}}(\cdot)\) denotes the orbifold Chern class in Bott-Chern cohomology for general coherent orbifold sheaves, introduced in \cite{MTTW25}; for vector bundles this coincides with the definition via metrics given in \cite{Ma05}. The Grothendieck-Riemann-Roch formula for orbifold embeddings is established in \cite{MTTW25}.
    \end{remark} 
	
	The following proposition is analogous to \cite[Lemma 2.13]{IMM24} for $\mathbb{Q}$-Chern classes.
	
	\begin{proposition}\label{prop-orbifold-exact}
		Let $\w_1,\cdots,\w_{n-2}$ be $n-2$ nef and big classes on a compact complex klt space $X$ of dimension $n$. Then for any short exact sequence of coherent sheaves		
		$$0\rightarrow \mF\rightarrow \mE\rightarrow\mQ\rightarrow 0$$
		on $X$ with $\mE$ reflexive and $\mQ$ torsion-free, we have
		\begin{equation}\label{equa-orbifoldsecond2}
			\widehat{\ch}_2(\mE)\cdot[\w_1]\cdots[\w_{n-2}]\leq \widehat{\ch}_2(\mF)\cdot[\w_1]\cdots[\w_{n-2}]+\widehat{\ch}_2(\mQ^{\vee\vee})\cdot[\w_1]\cdots[\w_{n-2}],
		\end{equation}
		where $\widehat{ch}_2(\cdot):=\frac{1}{2}\widehat{c}_1^2(\cdot)-\widehat{c}_2(\cdot).$ If $\w_0,\cdots,\w_{n-2}$ are K\"ahler forms and the equality in \eqref{equa-orbifoldsecond2} holds, then the natural map $\mQ \rightarrow \mQ^{\vee\vee}$ is an isomorphism in codimension $2$.
	\end{proposition}
	
	Proof of Proposition \ref{prop-orbifold-exact} essentially follows by considering the equality case of the argument of \cite[Proposition 4.23]{IJZ25} concerning $\widehat{c}_2(\mE)$ (see \cite[Lemma 3.2]{Kawa92} and \cite[Lemma 2.13]{IMM24} for the projective case). We  recall the following necessary lemmas in our proof.
	\begin{lemma}[cf. {\cite[Lemma 2.2]{ou25}}]\label{lem-excision}
		Let $f:Y\rightarrow X$ be a projective bimeromorphism between compact complex spaces of dimension $n$. Suppose that $\sigma_1,\sigma_2\in H_{2n-2k}(Y,\R)$ agrees outside $f^{-1}(Z)$ for some analytic subspace $Z\subset X$ of codimension at least $k$. Then $f_*\sigma_1=f_*\sigma_2\in H_{2n-2k}(X,\R)$.
	\end{lemma}
	
	\begin{lemma}[cf. {\cite[Lemma 2.17]{ZZZ25}}]\label{lem-compatible-orbifold-pullback}
		Let $g_\orb:Y_\orb\rightarrow X_\orb$ be an orbifold bimeromorphism between compact complex orbifolds. Suppose that $\mF_\orb$ and $\mE_\orb$ are coherent orbi-sheaves on $Y_\orb$ and $X_\orb$, respectively, such that $(g_\orb)_*\mF_\orb=\mE_\orb$ in codimension $k$ and the indeterminacy locus of $g_\orb^{-1}$ has codimension at least $k+1$. Then $g_*\ch_k^\orb(\mF_\orb)=\ch_k^\orb(\mE_\orb)$; that is, $\ch_k^\orb(\mF_\orb)\cdot g^*\gamma=\ch_k^\orb(\mE_\orb)\cdot \gamma$ for any $\gamma \in H^{2n-2k}(X,\R)$, where $g$ is the induced morphism between the underlying complex spaces $Y$ and $X$ of $Y_\orb$ and $X_\orb$, respectively.
	\end{lemma}
	
	\begin{lemma}\label{lem-modification-chernclass}
	Let the notation be as in Definition \ref{defn-orbifold}.	 For a reflexive sheaf $\mE$ on $X$, apply Lemma \ref{lem-orbireso} to the reflexive orbi-sheaf $\mE_\orb:=f^{[*]}_\orb\mE$ on $Y_\orb$ to obtain an orbifold morphism $h_{orb}:Z_{orb}\rightarrow Y_{orb}$ such that $E_\orb=h_\orb^*\mE_\orb/\tor:=\{h_i^*\mE_i/\tor\}$ is a vector orbi-bundle on $Z_\orb$. Then for any $\sigma\in H^{2n-4}(X,\R)$,
	$$\  \widehat{c}_2(\mE)\cdot\sigma =c_2^\orb((f\circ h)^{[*]}_\orb\mE)\cdot (f\circ h)^*\sigma=c_2^\orb(E_\orb)\cdot h^*f^*\sigma,$$
	where $h: Z \rightarrow Y$ is the induced map on underlying complex spaces. The same holds for $\widehat{c}_1(\mE)$ and $\widehat{c}_1^2(\mE)$.
	\end{lemma}
	
	\begin{proposition}\label{prop-orbifoldkahlerform}
		Suppose that $X$ is a compact complex normal space of dimension $n$ equipped with a K\"ahler form $\w_X$ and $f:Y\rightarrow X$ is a bimeromorphism from a compact complex space $Y$ with only quotient singularities to $X$. Let $Z$ be the indeterminacy locus of $f^{-1}$ and $Y_\orb:=\{(V_i,G_i,\mu_i)\}$ be a complex orbifold structure of $Y$. Then for any $V\subset\subset X\setminus Z$, there exists a smooth orbifold semipositive $(1,1)$-form $\w_{V,\orb}\in [f^*_\orb\w_X]\in H^{1,1}_{\rm{BC}}(Y_\orb,\R)$ such that $\w_{V,\orb}$ is strictly positive on $f_\orb^{-1}V:=\{(f\circ\mu_i)^{-1}V\}$.
	\end{proposition}
	
	This lemma follows from \cite[Section 6.1]{CGNPPW23}. In particular, taking $f$ to be the identity map $Y\rightarrow Y$, we conclude that $Y_\orb$ is K\"ahler if $Y$ is K\"ahler.	
	
	\begin{proof}[Proof of Proposition \ref{prop-orbifold-exact}]
	 Let $g:W\rightarrow X$ be a partial orbifold resolution by Lemma \ref{lem-orbireso} and $W_\orb:=\{(V_i,G_i,\mu_i)\}$ be the standard orbifold structure of $W$. Denote by $Z$ the indeterminacy locus of $g^{-1}$. Let $g_i:=g\circ\mu_i$, $\mE_\orb:=g_\orb^{[*]}\mE$ and $\mF_\orb$ be the saturation of $\{\imag(g_i^{[*]}\mF\rightarrow g_i^{[*]}\mE)\}$, where $g_i^{[*]}\mF\rightarrow g_i^{[*]}\mE$ is induced by the inclusion morphism $\mF\rightarrow \mE$. Since $\mu_i$ is \'etale in codimension $1$, $\mF_\orb=g_\orb^{[*]}\mF$ outside $g_\orb^{-1}(Z)$ due to reflexivity. Then we have
		\begin{equation}\label{equa-orbifold-reduction1}
			\begin{split}
				(\widehat{\ch}_2(\mE)-\widehat{\ch}_2(\mF))\cdot[\w_1]\cdots[\w_{n-2}]
				=&(\ch_2^\orb(\mE_\orb)-\ch_2^\orb(\mF_\orb))\cdot[g_\orb^*\w_1]\cdots[g_\orb^*\w_{n-2}]\\
				=&\ch_2^\orb(\mE_\orb/\mF_\orb)\cdot [g_\orb^*\w_1]\cdots[g_\orb^*\w_{n-2}],
			\end{split}
		\end{equation}
		where  the first equality follows from Definition \ref{defn-orbifold} and Lemma \ref{lem-compatible-orbifold-pullback}. 
		
		Let $\mQ_\orb:=\mE_\orb/\mF_\orb$. Then $\mT_\orb:=\mQ_\orb^{\vee\vee}/\mQ_\orb$ is a torsion orbi-sheaf supported in some orbifold analytic subspace $S_\orb$.  Based on the GRR formula for orbifold embeddings (cf. \cite[Theorem 1.1]{MTTW25}) and the fact that $\codim(S_{\text{orb}}) \geq 2$, it follows that
		\begin{equation}\label{equa-orbifoldsecond-torsion}
		 \ch_2^\orb(\mT_\orb)\cdot[g_\orb^*\w_1]\cdots[g_\orb^*\w_{n-2}]\geq 0,	
		\end{equation}
	    as can be shown by adapting the reasoning in \cite[Section 7]{CHP16}. This is shown as follows. First, we take an orbifold embedded resolution $\pi_\orb:(\widehat{W}_\orb,\widehat{S}_\orb)\rightarrow(W_{\text{orb}},S_{\text{orb}})$ such that all irreducible components $\widehat{S}_{i,\text{orb}}$ of $\widehat{S}_{\text{orb}}$ smooth with transversal intersection, which exists by the functorial embedded resolution (cf. \cite[Theorem 2.0.2]{W08}). Then, an application of the GRR formula for orbifold embeddings gives
	    \begin{equation}\label{equa-secondchernclass-computation1}
	    	\begin{split}
	    		&\ch_2^\orb(\mT_\orb)\cdot[g_\orb^*\w_1]\cdots[g_\orb^*\w_{n-2}]=\pi^*\ch_2^\orb(\mT_\orb)\cdot\pi^*[g_\orb^*\w_1]\cdots\pi^*[g_\orb^*\w_{n-2}]\\
	    		&=\sum\limits\rank(\mT_\orb|_{{\widehat{S}}_{i,\orb}})\int_{{\widehat{S}}_{i,\orb}}\pi_\orb^*g_\orb^*\w_1\wedge\cdots\wedge\pi_\orb^*g_\orb^*\w_{n-2}\\
	    		&=\sum\limits\rank(\mT_\orb|_{{\widehat{S}}_{i,\orb}})\int_{{S}_{i,\orb}}g_\orb^*\w_1\wedge\cdots\wedge g_\orb^*\w_{n-2}
	    		\geq0
	    	\end{split}
	    \end{equation}
	    since the image of the intersection of $\widehat{S}_\orb$ and the exceptional divisor of $\pi_\orb$ has codimension at least $3$, where $\pi:\widehat{W}\rightarrow W$ is the holomorphic map between the underlying complex spaces induced by $\pi_\orb$. 
	    
	    Combining \eqref{equa-orbifold-reduction1} and \eqref{equa-secondchernclass-computation1}, we obtain
	    \begin{equation}\label{equa-secondchernclass-computaion2}
	    	\begin{split}
	    		(\widehat{\ch}_2(\mE)-\widehat{\ch}_2(\mF))\cdot[\w_1]\cdots[\w_{n-2}]
	    		=& (-\ch_2^\orb(\mT_\orb)+\ch_2^\orb(\mQ_\orb^{\vee\vee}))\cdot[g_\orb^*\w_1]\cdots[g_\orb^*\w_{n-2}]\\
	    		\leq &\ch_2^\orb(\mQ_\orb^{\vee\vee})\cdot[g_\orb^*\w_1]\cdots[g_\orb^*\w_{n-2}]\\
	    		=&\widehat{\ch}_2(\mQ^{\vee\vee})\cdot[\w_1]\cdots[\w_{n-2}],
	    	\end{split}
	    \end{equation}
	    where the last equality follows from the fact that $\mQ_\orb^{\vee\vee}=g_\orb^{[*]}\mQ^{\vee\vee}$ outside $g_\orb^{-1}(Z)$. Then the inequality \eqref{equa-orbifoldsecond2} was proved. 
	    
	    Now we assume that $\w_0,\cdots,\w_{n-2}$ are K\"ahler forms and that the equality in \eqref{equa-orbifoldsecond2} holds. Then for any open subset $V$ compactly contained in $W\setminus g^{-1}(Z)$, Proposition \ref{prop-orbifoldkahlerform} means that there exists smooth orbifold semipositive $(1,1)$-forms $\w_{1,\orb}\in [g_\orb^*\w_1],\cdots,\w_{n-2,\orb}\in [g_\orb^*\w_{n-2}]$ on $W_\orb$ which are strictly positive on $g_\orb^{-1}(V)$. Moreover, the above argument implies that  the inequality in the third line of \eqref{equa-secondchernclass-computation1} is in fact an equality, which yields that $$\w_{1,\orb}\wedge\cdots\wedge\w_{n-2,\orb}=0$$
	    on each $S_{i,\orb}$. Hence $S_{i,\orb}\cap g_\orb^{-1}(V)=\emptyset$. Since $V$ is arbitrary, we conclude that $\left(W_\orb\setminus g_\orb^{-1}(Z)
	    \right)\cap S_\orb=\emptyset$ and consequently $\mQ_\orb^{\vee\vee}=\mQ_\orb$ on $W_\orb\setminus g_\orb^{-1}(Z)$. This gives rise to a short exact sequence of orbi-sheaves
	    $0\rightarrow \mF_\orb\rightarrow \mE_\orb\rightarrow \mQ_\orb^{\vee\vee}\rightarrow0$
	    over $W_\orb\setminus g_\orb^{-1}(Z)$. By \cite[Lemma A.3 and Lemma A.4]{GKKP11}, it induces a short exact sequence of reflexive sheaves
	    $0\rightarrow\mF\rightarrow\mE\rightarrow\mQ^{\vee\vee}\rightarrow 0$
	    on $X\setminus Z$. In particular, $\mQ^{\vee\vee}=\mQ$ outside $Z$. This completes the proof.
	\end{proof}
	
	\subsection{Quasi-\'etale morphisms}\label{subsection-quasietale}
	
	\begin{definition} Let $\gamma:Y\rightarrow X$ be a morphism between normal complex spaces.
		\begin{itemize}
			\item[(1)] $\gamma$ is called {\em Galois} if there exists a finite group $G\subset \Aut(Y)$ such that $\gamma:Y\rightarrow X$ is isomorphic to the quotient map.
			
			\item[(2)] $\gamma$ is called {\em quasi-\'etale} if $\dim X=\dim Y$ and \'etale in codimension $1$.
		\end{itemize}
	\end{definition}
	
	We have the following elementary property (cf. {\cite[Proposition 5.6]{GK20} for $\mE=\mT_X$}), which is analogous to \cite[Section 3.8]{GKPT19a} for $\mathbb{Q}$-Chern classes.
	\begin{proposition}[Orbifold Chern classes under quasi-\'etale morphisms]\label{lem-quasietale-chernclasses}
		Let $X$ be a compact klt space of dimension $n$ and $\gamma:X'\rightarrow X$ be a quasi-\'etale Galois morphism induced by a finite group $G\subset \Aut(X')$. Then $X'$ is also klt, and for any reflexive sheaf $\mE$ on $X$ and any class $\sigma\in H_{2n-2}(X,\R)$, we have
		$$\widehat{c}_2(\gamma^{[*]}\mE)\cdot\gamma^*\sigma=\deg \gamma \cdot \widehat{c}_2(\mE)\cdot \sigma.$$
		The same holds for $\widehat{c}_1(\mE)$ and $\widehat{c}_1^2(\mE)$.
	\end{proposition}
	\begin{proof}
		The fact that $X'$ is klt follows from the argument of \cite[Proposition 5.20]{KM98}. Suppose $f:Y\rightarrow X$ and $f':Y'\rightarrow X'$ are partial orbifold resolutions of $X$ and $X'$, respectively. Let $Y_\orb:=\{(V_\alpha,G_\alpha,\mu_\alpha)\}$ and $Y_\orb':=\{(V_\beta',G_\beta',\mu_\beta')\}$ be the maximal standard orbifold structures of $Y$ and $Y'$, respectively. Denote by $Z$ the union of the indeterminacy locus of $f^{-1}$ and the image under $\gamma$ of the indeterminacy locus of $(f')^{-1}$, a set which has codimension at least $3$. Recalling Lemma \ref{lem-modification-chernclass}, we may assume that there exist orbi-bundles $E_\orb$ on $Y_\orb$ and $E_\orb'$ on $Y_\orb'$ such that $E_\orb=f_\orb^*\mE$ on $Y_\orb\setminus f_\orb^{-1}(Z)$ and $E_\orb'=(f_\orb')^*\gamma^{[*]}\mE$ on $Y_\orb'\setminus {(\gamma\circ f')}_\orb^{-1}(Z)$, and that the following equalities hold:
		$$c_2^\orb(E_\orb)\cdot f^*\sigma=\widehat{c}_2(\mE)\cdot \sigma,\ c_2^\orb(E_\orb')\cdot (f')^*\gamma^*\sigma=\widehat{c}_2(\gamma^{[*]}\mE)\cdot\gamma^*\sigma.$$
		Thus it suffices to show that 
		$$\deg (\gamma)\cdot c_2^\orb(E_\orb)\cdot f^*\sigma= c_2^\orb(E_\orb')\cdot (f')^*\gamma^*\sigma.$$
		Let $g_{\alpha,\beta}:=W_{\alpha,\beta}\rightarrow V_\alpha\times_X V_\beta'$ be a functorial resolution of singularities (cf. \cite[Theorem 3.10]{DO23}), then there is an induced action $G_\alpha\times G_\beta'$ on $W_{\alpha,\beta}$ such that $g_{\alpha,\beta}$ is $(G_\alpha\times G_\beta')$-invariant, which induces an effective complex orbifold $W_\orb:=\{(W_{\alpha,\beta},G_\alpha\times G_\beta')\}$. Let $W$ be the underlying space and we have the following commutative diagram
		\[
		\begin{tikzcd}
			& & Y' \arrow[r,"f'"] & X' \arrow[d,"\gamma"] \\
			W \arrow[r,"g"] & Y\times_X Y' \arrow[r,"pr_1"] \arrow[ru,"pr_2"] &  Y \arrow[r,"f"] & X,
		\end{tikzcd}
		\]
		where $g$ is the morphism induced by $g_{\alpha,\beta}$. Since $\gamma$ is quasi-\'etale, $\{(V_\beta',GG_\beta', f^{-1}\circ\gamma\circ f'\circ\mu_\beta')\}$ is a non-maximal standard orbifold structure of $Y\setminus f^{-1}(Z)$ with the natural group action $GG_\beta'$. Applying Lemma \ref{lem-maximalorbifold}, we conclude that $V_\alpha'\times_X V_\beta'$ is smooth over $\left(Y\times_X Y'\right)\setminus\left(pr_1^{-1}(f^{-1}(Z))\right)$ and thus the indeterminacy locus of $g^{-1}$ is contained in $pr_1^{-1}(f^{-1}(Z))$. Let $g_{\orb,1}:W_\orb\rightarrow Y_\orb$ and $g_{\orb,2}:W_\orb\rightarrow Y_\orb'$ be the orbifold holomorphisms induced by $pr_1\circ g$ and $pr_2\circ g$, respectively. Therefore,  Lemma \ref{lem-compatible-orbifold-pullback} implies that $$c_2^\orb(g_{\orb,2}^*E_\orb')\cdot g^*pr_2^* (f')^*\gamma^*\sigma=c_2^\orb(g_{\orb,1}^*E_\orb)\cdot g^*pr_1^* f^*\sigma$$
		because $g_{\orb,1}^*E_\orb$ and $g_{\orb,2}^*E_\orb'$ coincide outside the preimage of $Z$. Note that $pr_2\circ g$ is generically biholomorphic, and $pr_1\circ g$ is generically finite and \'etale with degree $\deg (\gamma)$. The proof is complete.
	\end{proof}
	
	\begin{remark}[Locally free sheaf]\label{rem-locallyfree-orbifoldchernclass}
		Using an argument similar to that in the proof of Proposition \ref{lem-quasietale-chernclasses}, we find that for any locally free sheaf $\mE$, the classes $\widehat{c}_1(\mE),\widehat{c}_1^2(\mE)$ and $\widehat{c}_2(\mE)$ coincide, respectively, with the topological Chern classes of $\mE$  as elements of the homology groups.
	\end{remark}

    \begin{lemma}[Comparison]\label{lem-comparison}
    	Let $\mF$ be a torsion-free sheaf on a compact normal space $X$. Then the following statements hold.
    	\begin{itemize}
    		\item[(1)] For any $\sigma\in H^{2n-2}(X,\R)$ and finite quasi-\'etale morphism $\gamma:Y\rightarrow X$ between compact normal spaces, we have $$c_1(\gamma^*\mF)\cdot\gamma^*\sigma=\deg \gamma \cdot c_1(\mF)\cdot \sigma.$$
    		\item[(2)] If $X$ is klt, $c_1(\mF)=\widehat{c}_1(\mF)$.
    	\end{itemize}
    \end{lemma}
    
    \begin{proof}
    	Statement (1) can be deduced by an argument similar to that of Proposition \ref{lem-quasietale-chernclasses}. Statement (2) follows from the fact that for any orbifold modification of singularities $f:Y\rightarrow X$, one has $f_*c_1(f^{[*]}\mF)=c_1(\mF)$ (see e.g. \cite[Lemma 2.31]{ZZZ25}), together with \cite[Lemma 2.37]{ZZZ25}.
    \end{proof}

	\begin{lemma}[Stability under quasi-\'etale morphisms]\label{lem-quasietale-stability}
	  Let $(\mE_{X_\reg},\theta_{X_\reg})$ be a torsion-free Higgs sheaf on a compact complex normal space $X$ of dimension $n$ with rational singularities and $\alpha_0,\cdots,\alpha_{n-2}$ be nef classes. Suppose that $\gamma$ is a quasi-\'etale morphism induced by some group $G$. Then
	  \begin{itemize}
	  	\item[(1)] $(\mE_{X_\reg},\theta_{X_\reg})$ is $(\alpha_0,\cdots,\alpha_{n-2})$-stable if and only if $\gamma^{[*]}(\mE_{X_\reg},\theta_{X_\reg})$ is $G$-stable with respect to $\gamma^*\alpha_0,\cdots,\gamma^*\alpha_{n-2}$. The same holds for polystability.
	  	\item[(2)] $(\mE_{X_\reg},\theta_{X_\reg})$ is $(\alpha_0,\cdots,\alpha_{n-2})$-semistable if and only if $\gamma^{[*]}(\mE_{X_\reg},\theta_{X_\reg})$ is $(\gamma^*\alpha_0,\cdots,\gamma^*\alpha_{n-2})$-semistable.
	  \end{itemize}
	\end{lemma}
	
	\begin{proof}
			Based on the above lemmas, statement (1) follows directly from an argument similar to that of \cite[Proposition 4.14]{ZZZ25}. Regarding (2), it suffices to notice that the uniqueness of the maximal destablizing subsheaf $\mF_{Y_\reg}$ of $\gamma^{[*]}(\mE_{X_\reg},\theta_{X_\reg})$ implies that $\mF_{Y_\reg}$ is $G$-equivariant and any $G$-equivariant $\gamma^{[*]}\theta_{X_\reg}$-invariant subsheaf of $\gamma^{[*]}(\mE_{X_\reg},\theta_{X_\reg})$ descends to a $\theta_{X_\reg}$-invariant subsheaf of $(\mE_{X_\reg},\theta_{X_\reg})$.
	\end{proof}

	\section{Pluri-harmonic metrics on Higgs sheaves over the regular locus}\label{section-HE}
	In this section, we prove the existence of pluri-harmonic metrics on polystable reflexive Higgs sheaves with vanishing orbifold Chern number conditions \eqref{equa-vanishingcondition}. 
	
	We introduce some basic definitions to fix the conventions as follows. The Hitchin-Simpson connection \cite{Simpson88} is defined by
	\begin{equation*}
		\bar{\partial}_{\theta}:=\bar{\partial}_{E}+\theta , \quad D_{H,  \theta }^{1, 0}:=\partial_H  +\theta^{* H}, \quad D_{H,  \theta }:= \bar{\partial}_{\theta}+ D_{H,  \theta }^{1, 0},
	\end{equation*}
	where $\partial_H$ is the $(1, 0)$-part of the Chern connection $D_{H}$ of $(E,\bar{\partial }_{E}, H)$ and $\theta^{* H}$ is the adjoint of $\theta $ with respect to $H$.
	The curvature of Hitchin-Simpson connection is
	\begin{equation*}
		F_{H,\theta}=F_H+[\theta,\theta^{* H}]+\partial_H\theta+\bar{\partial}_E\theta^{* H},
	\end{equation*}
	where $F_H$ is the curvature of $D_{H}$. 
	
	\begin{definition}[Pluri-harmonic metric on Higgs bundles]
		Let $(E,\bp,\theta)$ be a Higgs bundle over a complex manifold $M$. We say that a Hermitian metric $H$ on $(E,\bp,\theta)$ is pluri-harmonic if $F_{H,\theta}=0$.
	\end{definition}
	
	The main result of this section is the following.
	
	\begin{theorem}\label{main-stable}
		Let $X$ be a compact K\"ahler klt space of dimension $n$. Suppose that $(\mE_{X_\reg},\theta_{X_\reg})$ is semistable with respect to some K\"ahler forms $\w_0,\cdots,\w_{n-2}$, then the following orbifold Bogomolov-Gieseker inequality holds:
		\begin{equation}\label{equa-BG}
			\left(2\widehat{c}_2(\mE_X)-\frac{r-1}{r}\widehat{c}_1^2(\mE_X)\right)\cdot[\w_1]\cdots[\w_{n-2}]\geq0.
		\end{equation}
		Moreover, if  $(\mE_{X_\reg},\theta_{X_\reg})\in\mathrm{pHiggs}_{X_\reg}$, then $(\mE_{X_\reg},\theta_{X_\reg})$ admits a pluri-harmonic metric $H_{X\setminus \Sigma}$ on $X\setminus\Sigma$, where $\Sigma\subsetneq X$ is an analytic subspace of codimension at least $2$ containing $X_{sing}$, such that $\mE_{X_\reg}|_{X\setminus \Sigma}$ is locally free.
	\end{theorem}
	
	\begin{remark}		
		When $X$ is smooth and $\mE_{X_\reg}$ is locally free, Theorem \ref{main-stable} builds on the Donaldson-Uhlenbeck-Yau theorem \cite{DON85,UY86} and its generalization to Higgs bundles \cite{Simpson88}; further developments appear in \cite{Bando94,BS,Li20,Mo1,Mo2,Jacob2014,LZZ,CGNPPW23,CW1,JL25,Jinnouchi25-1}, etc. We refer to the introduction of \cite{ou25} for an overview of developments on the orbifold Bogomolov-Gieseker inequality.
	\end{remark}
	
	\subsection{Strategy of Theorem \ref{main-stable}} \label{Strategy-HYM}
	In this subsection, we outline the strategy of proving Theorem \ref{main-stable}. The basic idea is to adapt the Higgs version \cite{LZZ} of Bando-Siu's argument to an orbifold modification $h:W\rightarrow X$ as in \cite{OF25} for stable reflexive sheaves. The main ingredients are:
	\begin{enumerate}
		\item The construction of an orbifold Higgs field of the pullback of a reflexive Higgs sheaf via Kebekus-Schnell's work \cite{KS21} on pulling back reflexive differentials (cf. Proposition \ref{prop-HEequation-setting}).
		\item The uniform Sobolev inequality for the degenerating family of orbifold K\"ahler metrics (cf. Proposition \ref{prop-uniformsobolev}), which follows from the smooth approximation argument in \cite[Section 3]{OF25} and uniform geometric estimates in \cite{GPS24,GPSS23,GPSS24}.
	\end{enumerate}
	
	Once these ingredients are in place, the subsequent proof at this stage can be reduced to studying the limiting behavior of the Hermitian-Yang-Mills flow over a noncompact balanced manifold $(X\setminus\Sigma,\w)$ and the strategy in \cite[Section 4]{LZZ} originated from \cite{Simpson88} remains valid.
	   
     \subsubsection{Pullback of the Higgs field over orbifold modifications}\label{section-pullbackhiggs}
     We first review the construction in \cite[Section 5.1]{ZZZ25}. 
	\begin{proposition}\label{prop-HEequation-setting}
		 Suppose that $(\mE_{X_\reg},\theta_{X_\reg})$ is a reflexive Higgs sheaf on the regular locus of a compact complex klt space $X$ and $\w_0,\cdots,\w_{n-2}$ are K\"ahler forms, then there exists a projective bimeromorphism $f:Y\rightarrow X$ from a compact K\"ahler space $Y$ with only quotient singularities to $X$ with the following data:
		\begin{itemize}
			\item[(a)] A complex effective orbifold structure $Y_\orb:=\{(U_\alpha,G_\alpha,\mu_\alpha)\}$ of $Y$.
			\item[(b)] An analytic subspace $\Sigma\subset X$ of codimension at least $2$ containing $X_{sing}$ such that each $f_\alpha:=f\circ\mu_\alpha$ is \'etale outside $f_\alpha^{-1}(\Sigma)$ and $\mE_{X_\reg}$ is locally free outside $\Sigma$.
			\item[(c)] A Higgs orbi-bundle $(E_\orb,\theta_\orb)$ over $Y_\orb$ with
			\begin{equation}\label{equa-firstchernclass-reduction}
				c_1^\orb(E_\orb)\cdot f^*\sigma=\widehat{c}_1(\mE_X)\cdot \sigma,\ \forall\sigma\in H_{2n-2}(X,\R)
			\end{equation}
			and
			\begin{equation}\label{equa-secondchernclass-reduction}
				{\ch}_2^\orb(E_\orb)\cdot f^*\beta=\widehat{\ch}_2(\mE_X)\cdot\beta,\ \forall \beta\in H_{2n-4}(X,\R).
			\end{equation}
			such that
			\begin{equation}\label{prop-higgs-compatible}
				(E_\orb,\theta_\orb)=\{f_\alpha^*(\mE_{X_\reg},\theta_{X_\reg})\}
			\end{equation} outside $f_\orb^{-1}(\Sigma):=\{f_\alpha^{-1}(\Sigma)\}$. The equality \eqref{equa-secondchernclass-reduction} also holds for $\widehat{c}_1^2(\mE_X)$.
		\end{itemize}
	\end{proposition}
	
	\begin{proof}
		Let $h:W\rightarrow X$ be an orbifold modification given by \cite{KO25} and $W_{\orb}:=\{(V_\alpha,G_\alpha,\mu_\alpha)\}$ be the standard orbifold structure of $W$. Set $\mE_{\orb}=h_\orb^{[*]}\mE_X:=\{\big((h\circ\mu_\alpha)^*\mE_X\big)^{\vee\vee}\}.$
		According to Lemma \ref{lem-orbireso}, there exists a projective orbifold morphism $g_{\orb}:Y_{\orb}\rightarrow W_{\orb}$ from a compact complex effective orbifold $Y_{\orb}:=\{(U_\alpha,G_\alpha,\mu_\alpha)\}$ to $W_{\orb}$ such that $E_{\orb}=\{g_\alpha^*\mE/(\mathrm{torsion})\}$ is a vector orbi-bundle. Let $g$ be the induced morphism from the quotient space $Y$ of $Y_\orb$ to $W$, then $f:=h\circ g$ is a projective bimeromorphism.

     	Building on Lemma \ref{lem-reflexive-pullback}, we conclude that there exists a Higgs field $\theta_{\orb}=\{\theta_\alpha\}$ of $E_{\orb}$ such that $\theta_\alpha|_{f_\alpha^{-1}(X_\reg)}$ is the pullback of $\theta_{X_\reg}$ via $f_\alpha$. We summarize the properties resulting from the constructions as follows:
		\begin{itemize}
			\item[(\romannumeral1)] There exists an analytic subspace $Z\subset X$ of codimension at least $3$ such that $h$ is biholomorphic outside $h^{-1}(Z)$.
			\item[(\romannumeral2)] The orbifold structure $W_{\orb}$ is standard, i.e., $\mu_\alpha$ is finite and quasi-\'etale for each $\alpha$.
			\item[(\romannumeral3)] The morphism $g_\alpha$ is biholomorphic outside the preimage of the non-locally free locus $S_\alpha$ of $\mE_\alpha$, where $S_\alpha$ has codimension at least $3$ due to the reflexivity of $\mE_\alpha$.
		\end{itemize}	
		Combining Lemma \ref{lem-modification-chernclass} and  \ref{lem-comparison}, the proof is complete by choosing $\Sigma=Z\cup X_{sing}\cup(\cup_\alpha h_\alpha(S_\alpha))$.
	\end{proof}

	\subsubsection{Bando-Siu's methods}\label{section-bandosiu-strategy}
	Suppose that $\w_0,\cdots,\w_{n-2}$ are K\"ahler forms on $X$. By Proposition \ref{prop-orbifoldkahlerform}, $Y_\orb$ admits an orbifold K\"ahler form $\w_{Y_\orb}=\{\w_\alpha\}$. Set
	\begin{equation}\label{omegai}
		\omega_{\orb,k, \epsilon }:=f_\orb^{\ast }\omega_{k} +\epsilon \omega_{Y_\orb}=\{f_\alpha^*\w_k+\epsilon\w_\alpha\}
	\end{equation}
	for all $0\leq k \leq n-2$, where $f_\orb^*\w_k:=\{f_\alpha^*\w_k\}$. Recall that for each $\epsilon$, there exists a smooth orbifold Hermitian metric $\w_{\orb,\epsilon}$ (see e.g. \cite[Pages 279]{Mi}) such that
	\begin{equation}\label{omegas}
		\omega_{\orb,\epsilon}^{n-1}=\omega_{\orb,0, \epsilon }\wedge  \cdots \wedge \omega_{\orb,n-2, \epsilon }.
	\end{equation}
	Note that $\w_{\orb,\epsilon}$ is an orbifold balanced metric (i.e. $\p\w_{\orb,\epsilon}^{n-1}=0$) that is not K\"ahler in general. Given a smooth Hermitian metric $\hat{H}$ on the Higgs orbi-bundle $(E_\orb, \bar{\partial }_{E_\orb}, \theta_\orb )$, it is easy to see that there exists a uniform constant $\hat{C}_{0}$ such that
	\begin{equation}\label{initial1}
		\int_{Y_\orb}( |\Lambda_{\omega_{\orb,\epsilon }}F_{\hat{H}}|_{\hat{H}}+|\theta |_{\hat{H}, \omega_{\orb, \epsilon}}^{2})\frac{\omega_{\orb,\epsilon}^{n}}{n!}\leq \hat{C}_{0}
	\end{equation}
	for all $ \epsilon $. Consider the Hermitian-Yang-Mills flow on the Higgs orbi-bundle $(E_\orb, \bar{\partial }_{E_\orb}, \theta_\orb  )$ with the fixed initial metric $\hat{H}_\orb:=\{\hat{H}_\alpha\}$ and with respect to the Hermitian metric $\omega_{\orb,\epsilon}$,
	\begin{equation}\label{DDD1}
		\left \{\begin{aligned} &H_{\orb,\epsilon}(t)^{-1}\frac{\partial H_{\orb,\epsilon}(t)}{\partial
				t}=-2(\sqrt{-1}\Lambda_{\omega_{ \orb,\epsilon}}(F_{H_{\orb,\epsilon}(t)}+[\theta_\orb , \theta_\orb^{\ast H_{\orb,\epsilon}(t)}])-\lambda_{\epsilon }\Id_{E_\orb}),\\
			&H_{ \orb,\epsilon}(0)=\hat{H}_\orb,\\
		\end{aligned}
		\right.
	\end{equation}
	where  $\lambda_{ \epsilon}  =\frac{2\pi}{\Vol(Y_\orb, \omega_{\orb,\epsilon})} \mu_{\omega_{ \orb,\epsilon }} (E_\orb)$. Following Simpson's argument \cite{Simpson88}, we obtain the long time existence and uniqueness of the solution to the above flow (\ref{DDD1}). Actually, the needed estimates are valid on each orbifold chart $U_\alpha$, where the computations are consistent with those in the manifold case (see e.g. \cite{Faulk22} for the existence of Hermitian-Einstein metrics on stable vector orbi-bundles, and \cite[Section 3]{ZZZ25} for the general case of stable Higgs orbi-bundles).
	
	In Section \ref{section-HYM}, building on the uniform Sobolev inequality (Proposition \ref{prop-uniformsobolev}) and following Bando-Siu's argument \cite{BS}, by choosing a subsequence, we shall prove that $H_{\orb,\epsilon}(t)$ converges to smooth metric $H_\orb(t):=\{H_\alpha(t)\}$ in $C_{loc}^{\infty}(Y_\orb\setminus f_\orb^{-1}(\Sigma))$-topology as $\epsilon \rightarrow 0$ and thus $H_\orb(t)$ can descend to a long time solution $H(t)$ of the Hermitian-Yang-Mills flow on $(X\setminus \Sigma)\times[0,+\infty)$, i.e., $H(t)$ satisfies:
	\begin{equation}\label{SSS1}
		\left \{\begin{aligned} &H(t)^{-1}\frac{\partial H(t)}{\partial
				t}=-2(\sqrt{-1}\Lambda_{\omega}(F_{H(t)}+[\theta , \theta ^{\ast H(t)}])-\lambda \Id_{\mathcal{E}}),\\
			&H(0)=\hat{H},\\
		\end{aligned}
		\right.
	\end{equation}
	because $\widehat{H}_\alpha$ is $G_\alpha$-invariant and $f_\alpha$ is \'etale outside $f_\alpha^{-1}(\Sigma)$, where $\w$ is a smooth Hermitian metric on $X$ such that
	\begin{equation}\label{equa-balanced}
		\omega^{n-1}=\omega_{0}\wedge  \cdots \wedge \omega_{n-2}
	\end{equation} 
	and $\lambda=\frac{2\pi}{\int_{Y_\orb}(f_\orb^*\w)^{n}}\mu_{(\w_0,\cdots,\w_{n-2})}(\mE_{X})$. Observe that $\{H_\alpha(t)\}= \{f_\alpha^*H(t)\}.$
	
	In Section \ref{section-HYM2}, we study the limiting behavior of the Hermitian-Yang-Mills flow (\ref{DDD1}), which enables us to obtain the following theorem.
	
	\begin{theorem}\label{thm1}
		In the setting of Proposition \ref{prop-HEequation-setting}, if $(\mE_{X_\reg},\theta_{X_\reg})$ is $(\w_0,\cdots,\w_{n-2})$-stable, then by choosing a subsequence, we have $H(t)\rightarrow H_\infty$ in $C^\infty_{loc}(X\setminus \Sigma)$ and
		$$\im\Lambda_\w(F_{H_\infty,\theta_{X_\reg}})=\lambda\cdot \Id_{\mE_{X_\reg}}$$
		on $X\setminus \Sigma$. In the semistable case, we have
		\begin{equation}\label{equa-approximate-HE}
			\sup_{x\in X\setminus \Sigma}|\sqrt{-1} \Lambda_\omega (F_{H(t), \theta_{X_\reg} })-\lambda\Id_{\mathcal{E}_{X_\reg}}|_{H(t)}^2\rightarrow 0.
		\end{equation}
	\end{theorem}
	
	To establish the Donaldson-Uhlenbeck-Yau theorem in our context, what remains open is the admissibility of $H_\infty$. The primary obstacle is proving a uniform estimate of $\theta_{X_\reg}$, a necessary step that is currently not attainable (cf. the smooth case in \cite[Section 3]{LZZ}).
	\begin{question}\label{ques-uniformestimate}
	  For $t_0>0$, is $\|\theta_{X_\reg}\|_{H(t),\w}$ uniformly bounded for $t>t_0$ along the Hermitian-Yang-Mills flow $H(t)$?
	\end{question}
	
	\subsubsection{Completion of proof of Theorem \ref{main-stable}}\label{subsection-computation}
	Building on Theorem \ref{thm1}, the proof of Theorem \ref{main-stable} can be completed.  It suffices to consider the stable case. We borrow the ideas of \cite{CW1,Chen25} using the Hodge index Theorem. 
	
	Let $X$ be a compact K\"ahler klt space of dimension $n$ and $\w_0,\cdots,\w_{n-2}$ be $n-1$ K\"ahler forms. Suppose that $(\mE_{X_\reg},\theta_{X_\reg})$ is an $(\w_0,\cdots,\w_{n-2})$-stable reflexive Higgs sheaf on the regular locus of $X$ with vanishing orbifold Chern numbers \eqref{equa-vanishingcondition}. We adapt the notations in Section \ref{section-pullbackhiggs} and Section \ref{section-bandosiu-strategy}. It follows from \eqref{equa-firstchernclass-reduction} that
	\begin{equation}\label{equa-stable-slope}
		\lambda=\frac{2\pi}{\int_{Y_\orb}(f^*_\orb\w)^n}\mu_{(\w_0,\cdots,\w_{n-2})}(\mE_X)=\frac{2\pi}{\int_{Y_\orb}(f^*_\orb\w)^n}\frac{\widehat{c}_1(\mE_X)\cdot[\w_0]\cdots[\w_{n-2}]}{\rank(\mE_X)}=0.
	\end{equation}
    Since $\w_\epsilon^{n-1}=\w_{0,\epsilon}\wedge\cdots\wedge\w_{n-2,\epsilon}$, we have
	\begin{align*}
		0&=(F_{H_{\orb,\epsilon}(t),\theta_{\orb}}-\frac{\Lambda_{\w_{\orb,\epsilon}}F_{H_{\orb,\epsilon}(t)}}{n}\cdot\w_{\orb,\epsilon})\wedge\w_{\orb,0,\epsilon}\wedge\cdots\wedge\w_{\orb,n-2,\epsilon}\\
		&=(F_{H_{\orb,\epsilon}(t),\theta_{\orb}}-\frac{\Lambda_{\w_{\orb,\epsilon}}F_{H_{\orb,\epsilon}(t)}}{n}\cdot \gamma_{\orb,\epsilon}\w_{\orb,0,\epsilon})\wedge\w_{\orb,0,\epsilon}\wedge\cdots\wedge\w_{\orb,n-2,\epsilon},
	\end{align*}
	where $\gamma_{\orb,\epsilon}:=\frac{\w_{\orb,\epsilon}^n}{\w_{\orb,0,\epsilon}^2\wedge\cdots\wedge\w_{\orb,n-2,\epsilon}}$ ($\epsilon>0$) is a smooth function on $Y_\orb$. Recall the following elementary statement \cite{Timorin98}.
	
	\begin{lemma}\label{lem-riemannbilinear}
		Let $V$ be a complex vector space of dimension $n$. Let $\w_0,\cdots,\w_{n-2}$ be $n-1$ positive $(1,1)$-forms. We define a Hermitian form on $\Lambda^{1,1}V$ by
		$$Q(\xi,\sigma):=*(\xi\wedge\bar{\sigma}\wedge\w_1\wedge\cdots\wedge\w_{n-2}), \ \forall \ \xi,\sigma\in\Lambda^{1,1}V,$$
		and the mixed primitive subspace $P^{1,1}$ by
		$$P^{1,1}:=\{\xi\in \Lambda^{1,1}V:\xi\wedge\w_0\wedge\cdots\wedge\w_{n-2}=0\}.$$
		Then $Q$ is negative definite on $P^{1,1}$.
	\end{lemma}

	By applying Lemma \ref{lem-riemannbilinear} to $\im F_{H_{\orb,\epsilon}(t),\theta_{\orb}}-\frac{\Lambda_{\w_{\orb,\epsilon}}(\im F_{H_{\orb,\epsilon}(t)})}{n} \gamma_{\orb,\epsilon}\w_{\orb,0,\epsilon}$ with respect to $\w_{\orb,0,\epsilon},\cdots,\w_{\orb,n-2,\epsilon}$, we conclude that
	\begin{align*}
		\tr\bigg(&(F_{H_{\orb,\epsilon}(t),\theta_{\orb}}-\frac{\Lambda_{\w_{\orb,\epsilon}}F_{H_{\orb,\epsilon}(t),\theta_{\orb}}}{n}\cdot \gamma_{\orb,\epsilon}\w_{\orb,0,\epsilon})\\
		&\wedge(F_{H_{\orb,\epsilon}(t),\theta_{\orb}}-\frac{\Lambda_{\w_{\orb,\epsilon}}F_{H_{\orb,\epsilon}(t),\theta_{\orb}}}{n}\cdot \gamma_{\orb,\epsilon}\w_{\orb,0,\epsilon})\bigg)\wedge\Omega_{\orb,\epsilon}
	\end{align*}
	is positive outside $f_\orb^{-1}(\Sigma)$ for $0<\epsilon\ll1$, and note that its descent via $f_\orb$ converges to 
	$$\tr\left((F_{H(t),\theta_{X_\reg}}-\frac{\Lambda_{\w}F_{H(t),\theta_{X_\reg}}}{n}\cdot \gamma_0\w_{0})\wedge(F_{H(t),\theta_{X_\reg}}-\frac{\Lambda_{\w}F_{H(t),\theta_{X_\reg}}}{n}\cdot \gamma_0\w_{0})\right)\wedge\Omega$$
	as $\epsilon\rightarrow0$ in $C^\infty_{loc}( X\setminus \Sigma)$, where $\Omega_{\orb,\epsilon}=\w_{\orb,1,\epsilon}\wedge\cdots\wedge\w_{\orb,n-2,\epsilon},\Omega=\w_1\wedge\cdots\wedge\w_{n-2}$ and $\gamma_0=\frac{\w^n}{\w_0^2\wedge\w_1\cdots\wedge\w_{n-2}}$. Combining \eqref{equa-secondchernclass-reduction} and Remark \ref{remark-MTTW25}, and applying the Fatou Lemma, we have
	\begin{align*}
		&-8\pi^2\widehat{\ch}_2(\mE_X)\cdot[\w_1]\cdots[\w_{n-2}]
		=-8\pi^2{\ch}_2^\orb(E_\orb)\cdot[f_\orb^*\w_{1}]\cdots[f_\orb^*\w_{n-2}]\\
		=&-8\pi^2\lim\limits_{\epsilon\rightarrow0}\ch_2^\orb(E_\orb)\cdot[\w_{\orb,1,\epsilon}]\cdots[\w_{\orb,n-2,\epsilon}]
		=\lim\limits_{\epsilon\rightarrow0}\int_{Y_\orb} \tr(F_{H_{\orb,\epsilon}(t),\theta_\orb}\wedge F_{H_{\orb,\epsilon}(t),\theta_\orb})\wedge\Omega_{\orb,\epsilon}\\
		=&\lim\limits_{\epsilon\rightarrow0}\int_{Y_\orb}\bigg(\tr\big((F_{H_{\orb,\epsilon}(t),\theta_{\orb}}-\frac{\Lambda_{\w_{\orb,\epsilon}}F_{H_{\orb,\epsilon}(t),\theta_{\orb}}}{n}\cdot \gamma_{\orb,\epsilon}\w_{\orb,0,\epsilon})\wedge (F_{H_{\orb,\epsilon}(t),\theta_{\orb}}\\
		&\ \ \ \ \ \ \ \ \ \ \ \ \ \ \ \ \ \ -\frac{\Lambda_{\w_{\orb,\epsilon}}F_{H_{\orb,\epsilon}(t),\theta_{\orb}}}{n}\cdot \gamma_{\orb,\epsilon}\w_{\orb,0,\epsilon})\big) -\frac{1}{n^2}|\Lambda_{\w_{\orb,\epsilon}}F_{H_{\orb,\epsilon}(t),\theta_\orb}|^2\gamma_{\orb,\epsilon}^2\w_{\orb,0,\epsilon}^2\bigg)\wedge\Omega_{\orb,\epsilon}\\
		\geq& \int_{X\setminus \Sigma}\tr\left((F_{H(t),\theta_{X_\reg}}-\frac{\Lambda_{\w}F_{H(t),\theta_{X_\reg}}}{n}\cdot \gamma_0\w_{0})\wedge(F_{H(t),\theta_{X_\reg}}-\frac{\Lambda_{\w}F_{H(t),\theta_{X_\reg}}}{n}\cdot \gamma_0\w_{0})\right)\wedge\Omega\\
		&-\frac{1}{n^2}\int_{X\setminus\Sigma}|\Lambda_\w F_{H(t),\theta_{X_\reg}}|^2\gamma_0\w^n.
	\end{align*}
	Similarly, as $t\rightarrow+\infty$, it follows from Theorem \ref{thm1} and the equation \eqref{equa-stable-slope} that
	$$0\geq \int_{X\setminus\Sigma}\tr(F_{H_\infty,\theta_{X_\reg}}\wedge F_{H_\infty,\theta_{X_\reg}})\wedge\w_1\wedge\cdots\wedge\w_{n-2}.$$
	Since \eqref{equa-stable-slope} implies $F_{H_\infty,\theta_{X_\reg}}\wedge\w_0\wedge\cdots\wedge\w_{n-2}=0$ outside $\Sigma$, Lemma \ref{lem-riemannbilinear} then yields    
	$F_{H_\infty,\theta_{X_\reg}}\equiv0$ outside $\Sigma$.
	
	\vspace{0.1cm}
	
	To complete the proof, it remains to prove Bogomolov-Gieseker inequality \eqref{equa-BG}. Recall from \cite[Corollary 1.3]{ZZZ25} that for an $(\w_0,\cdots,\w_{n-2})$-semistable reflexive Higgs sheaf $(\mE_{X_\reg},\theta_{\reg})$,  $$\mG_{X_\reg}:=\End(\mE_{X_\reg})=((\mE_{X_\reg})^\vee\otimes\mE_{X_\reg})^{\vee\vee}$$ with the induced Higgs sheaf $\theta_{\mG_{X_\reg}}$ is $(\w_0,\cdots,\w_{n-2})$-semistable. It follows from \cite[Lemma 2.45]{ZZZ25} that $\widehat{c}_1(\mG_X)=0$ and
	\begin{equation}\label{equa-end-chernclass-resuction}
		\widehat{\Delta}(\mE_X)=\frac{1}{\rank(\mE_X)^2}\widehat{\Delta}(\mG_X)=-\frac{1}{\rank(\mE_X)}\widehat{\ch}_2(\mG_X).
	\end{equation}
	Then the Bogomolov-Gieseker inequality \eqref{equa-BG} can be concluded by applying the proceeding argument to $(\mG_{X_\reg},\theta_{X_\reg})$.
	
	\subsection{$C_{loc}^\infty$-convergence of {$H_{\orb,\epsilon}(t)\rightarrow H_{\orb}(t)$}}\label{section-HYM}
	We adopt the notation of Section \ref{Strategy-HYM} throughout Sections \ref{section-HYM} and \ref{section-HYM2}. Recall that $H_{\orb,\epsilon}(t)$ is the long time solution of the Hermitian-Yang-Mills flow (\ref{DDD1}) on the Higgs orbi-bundle $(E_\orb, \theta_\orb)$ with the fixed smooth initial metric $\hat{H}_\orb$ and with respect to the Hermitian metric $\omega_{\orb,\epsilon}$. We shall prove the following result.
	
	\begin{proposition}\label{lem 2.6}
		By choosing subsequences,   $H_{\orb,\epsilon}(t)$ converges  to $H_{\orb}(t)=\{f_\alpha^*H(t)\}$ in $C_{loc}^\infty(Y_\orb\setminus f_\orb^{-1}(\Sigma))$ as $\epsilon\rightarrow 0$, where $H(t)$ is a smooth solution of the evolution equation (\ref{SSS1}) on $(X\setminus \Sigma) \times [0, +\infty)$.  Furthermore,  $H(t)$ satisfies:
		\begin{equation}\label{H0013}
			\int_{X\setminus \Sigma}|\Phi (H(t), \omega) |_{H(t)}\frac{\omega^{n}}{n!}\leq \int_{X\setminus \Sigma}|\Phi (\hat{H}, \omega) |_{\hat{H}}\frac{\omega^{n}}{n!}\leq \widetilde{C}_{1},
		\end{equation}
		and
		\begin{equation}\label{H00013}
			|\Phi (H(t+s), \omega ) |_{H(t+s)}(x) \leq \tilde{C}_{2}(\frac{1}{s^{2n+1}}+1)\int_{X\setminus \Sigma}|\Phi (H(t), \omega ) |_{H(t)}\frac{\omega^{n}}{n!}
		\end{equation}
		for all $x\in X\setminus \Sigma $, $s> 0$ and $t\geq 0$, where $\Phi (H, \omega )=\sqrt{-1}\Lambda_{\omega}(F_{H}+[\theta , \theta ^{\ast H}])-\lambda_{ \omega }\Id_{E}$, $\widetilde{C}_{1}$ and $\tilde{C}_{2}$ are positive constants.
	\end{proposition}
	\subsubsection{Uniform Sobolev inequality for $\w_{\orb,k,\epsilon}$}
	Thanks to Guo-Phong-Song-Sturm's uniform geometric estimates (\cite{GPS24, GPSS23, GPSS24}) and the argument of \cite[Section 3]{OF25}, we have the following uniform Sobolev inequalities for orbifold K\"ahler metrics $\w_{\orb,k,\epsilon}=f_\orb^*\w_k+\epsilon\w_{Y_\orb}$ (see \eqref{omegai}).
	\begin{proposition}\label{prop-uniformsobolev}
		For any $q\in(1,\frac{n}{n-1})$, there exists a uniform constant $C_S>0$ independent of $\epsilon$ such that
		\begin{equation}\label{equa-uniformsoblev}
			(\int_{Y_\orb}|u|^{2q}\frac{\w_{\orb,k,\epsilon}^n}{n!})^{\frac{1}{q}}\leq C_S\int_{Y_\orb}(|du|_{\w_{\orb,k,\epsilon}}^2+|u|^2)\frac{\w_{\orb,k,\epsilon}^n}{n!}
		\end{equation}
		for any $u\in L_1^2(Y_\orb,\w_{\orb,k,\epsilon})$ and $0<\epsilon\leq 1$.
	\end{proposition}
	
	Recall the uniform Sobolev inequality proved in \cite[Section 4]{GPSS23} (see also \cite[Section 2.3]{Pan} for general compact Hermitian manifolds).
	\begin{theorem}\label{thm-Kclass}(cf. \cite[Theorem 4.1]{GPSS23})
		Let $(\widehat{Y},\theta_{\widehat{Y}})$ be a compact K\"ahler manifold. For some fixed positive constants $A,K,\delta$ and a nonnegative continuous function $\gamma$ on $\widehat{Y}$, the set $\mathcal{K}(\widehat{Y},\theta_{\widehat{Y}}, \gamma, A, K, \delta)$ is defined to be the set of K\"ahler metrics $\w$ on $\widehat{Y}$ satisfying
		$$e^{F_\w}\geq \gamma,\quad \|e^{F_\w}\|_{L^{1+\delta}}\leq K,\quad [\w]\cdot[\theta_{\widehat{Y}}]^{n-1}\leq A\ \text{ and }\ [\w]^n\geq A^{-1},$$ 
		where $F_\w:=\log(\frac{1}{V_\w}\frac{\w^n}{\theta_{\widehat{Y}}^n})$ for $V_\w:=[\w]^n$. For any $q\in(1,\frac{n}{n-1})$, there exists a constant $C=C(\gamma, A, K, \delta, q)>0$ such that for any $\w\in\mathcal{K}(\widehat{Y}, \theta_{\widehat{Y}}, \gamma, A, K, \delta)$, we have
		$$(\int_{\widehat{Y}}|u|^{2q}\w^n)^{\frac{1}{q}}\leq C\int_{\widehat{Y}}(|du|_{\w}^2+|u|^2)\frac{\w^n}{n!}$$
		for any $u\in L_1^2(\widehat{Y},\w)$.
	\end{theorem}
	
	Building on Theorem \ref{thm-Kclass}, Proposition \ref{prop-uniformsobolev} can be directly concluded by applying the smooth approximation argument from \cite[Section 3]{OF25}. For the reader's convenience, we include a brief explanation.
	
	\begin{proof}[Proof of Proposition \ref{prop-uniformsobolev}]
		For simplicity, we omit the subscript `k' for now. Note that $\mu_\alpha:U_\alpha\rightarrow U_\alpha/G_\alpha$ is \'etale outside a proper analytic subspace $Z$ of codimension at least $2$. Recall from Proposition \ref{prop-orbifoldkahlerform} that there exists a smooth K\"ahler form $\w_0$ on the underlying space $Y$ such that $\w_{Y_\orb}:=\{\mu_\alpha^*\w_0+\im\pp\varphi_\alpha\}$ for some $G_\alpha$-invariant smooth function $\varphi_\alpha$ on every chart $U_\alpha$. Then $\varphi_\alpha$ induces a bounded continuous function $\varphi$ on $Y$ that is smooth outside $Z$ and in particular, $\w_{Y_\orb}=\{\mu_\alpha^*\w_Y\}$ for $\w_Y=\w_0+\im\pp\varphi$. Thus, $\w_{\orb,\epsilon}:=\{f_\orb^*\w+\epsilon(\mu_\alpha^*\w_0+\im\pp\varphi_\alpha)\}$ can also be viewed as a K\"ahler current $\w_\epsilon=f^*\w+\epsilon\w_0+\epsilon\im\pp\varphi$ on the underlying complex space $Y$ and $\w_\epsilon$ is smooth outside $Z$. Let $g:\widehat{Y}\rightarrow Y$ be a resolution of singularities of $Y$ and $D$ be the exceptional divisor of $f\circ g$. We use the notation of Proposition \ref{prop-uniformsobolev} and Theorem \ref{thm-Kclass}.
		
		As in \cite[Section 7]{GPSS23}, the basic strategy is to consider a smooth approximation of $g^*\w_\epsilon$ on $\widehat{Y}$. The situation differs slightly from that in \cite[Section 7]{GPSS23} because $\log(\frac{g^*\w_\epsilon^n}{\theta_{\widehat{Y}}^n})$ does not, in general, have log type analytic singularities (see \cite[Definition 7.2]{GPSS23}). In fact, it is the sum of a function $F_{\log}$ with log type analytic singularities and a continuous function $G_\epsilon$ that is smooth outside $D$ (cf. \cite[Lemma 3.9]{OF25}). The main consideration in \cite[Section 3]{OF25} is addressing this issue. According to \cite[Lemma 3.12, Lemma 3.13 and Lemma 3.15]{OF25}, there exists a sequence $\w_{\epsilon,j}$ of K\"ahler metrics on $\widehat{Y}$ such that $\w_{\epsilon,j}\in\mathcal{K}(\widehat{Y}, \theta_{\widehat{Y}}, \gamma, A, K, \delta)$ for some $A, K, \delta$ and $\gamma$ independent of $j$ and $\epsilon$, and for any integer $k\geq0$, we have $\w_{\epsilon,j}\rightarrow g^*\w_\epsilon$ in $C_{loc}^k(\widehat{Y}\setminus D)$. Then by applying the argument on \cite[Page 34]{GPSS23} and Theorem \ref{thm-Kclass}, the uniform Sobolev inequality \eqref{equa-uniformsoblev} for $\w_{\orb,\epsilon}$ can be proved because $\mu_\alpha$ is unramified outside $Z$.
	\end{proof}
	
	\begin{remark}
		It might be possible to develop an orbifold analogue of \cite{GPSS23} to establish Proposition \ref{prop-uniformsobolev}; see e.g. \cite{Sche24} for the orbifold version of the mean value inequality.
	\end{remark}
	
	It is easy to see that there exists some constant $C_k>0$ independent of $\epsilon$ such that
	\begin{equation}\label{equa-metric-localbound}
		C_{k}^{-1}\omega_{\orb,k, \epsilon}\leq \omega_{\orb,\epsilon} \leq C_{k}\omega_{\orb,k, \epsilon}
	\end{equation}
     from the construction (see \eqref{omegai} and \eqref{omegas}). We have
	the following uniform Sobolev inequalities for the Hermitian metrics $\omega_{\orb,\epsilon}$ by applying Proposition \ref{prop-uniformsobolev} to  $q=\frac{2n}{2n-1}$.

	\begin{proposition}
		There exists a uniform constant $C_{S}>0$ such that
		\begin{equation}\label{Sobolev}
			(\int_{Y_\orb} |u|^{\frac{4n}{2n-1}}\frac{\omega_{\orb,\epsilon}^{n}}{n!})^{\frac{2n-1}{2n}}\leq C_{S} \int_{Y_\orb} (|du |^{2}_{\omega_{\orb,\epsilon }} +|u|^{2})\frac{\omega_{\orb,\epsilon}^{n}}{n!} 
		\end{equation}
		for all $u \in L_{1}^{2}(Y_\orb,\w_{\orb,\epsilon})$ and all $0<\epsilon \leq 1$.
	\end{proposition}

	\subsubsection{Mean value inequality on orbifolds }\label{subsubsection-meanvalueinequality}
	
	We have the following standard result based on the Stokes formula for orbifolds (see e.g. \cite[Theorem 3.4.2]{CF2019}) and the fact that $d\w_\orb^{n-1}=0$.
	
	\begin{lemma}\label{ml1}
		Let $(Y_\orb, \omega_\orb)$ be an $n$-dimensional compact complex Gauduchon orbifold. Suppose that the nonnegative function $\phi(x,t)\in C^{\infty}(Y_\orb\times[0,\infty))$ is a solution of the following heat flow \begin{equation}\label{htf011}
			(2\sqrt{-1}\Lambda_{\w_\orb}\partial\overline{\partial}-\frac{\partial}{\partial t})\phi\geq 0. \end{equation} Then $$ \|\phi(t)\|_{L^1(Y_\orb,\omega_\orb)}\leq \|\phi(t_0)\|_{L^1(Y_\orb,\omega_\orb)}$$ for any $t>t_{0}\geq 0$.
	\end{lemma}

	The proof of the following lemma proceeds by directly adapting the argument in \cite[Appendix]{CZ25} to the orbifold setting. The only difference is that the constant $\widehat{C}_1$ does not depend on $d\w_\orb^{n-1}$ since we assume that $\w_\orb$ is balanced (see \cite[(A.27)]{CZ25}.
	
	\begin{lemma}\label{l1}
		Let $(Y_\orb, \omega_\orb)$ be an $n$-dimensional compact balanced orbifold. Suppose that $\phi_\orb(x,t)\in C^{\infty}(Y_\orb\times[0,\infty))$ is a nonnegative orbifold function satisfying  the following inequality \begin{equation}\label{htf01}
			(2\sqrt{-1}\Lambda_{\w_\orb}\partial\overline{\partial}-\frac{\partial}{\partial t})\phi_\orb\geq -\breve{C}\phi_\orb 
		\end{equation} 
		with $\phi_\orb(0)\in L^1(Y_\orb,\omega_\orb)$ and $\|\phi_\orb(t)\|_{L^1(Y_\orb,\omega_\orb)}\leq \|\phi_\orb(t_0)\|_{L^1(Y_\orb,\omega_\orb)}$ for any $t>t_{0}\geq 0$. Then we have the following statement:
		
		(1) For any $s>0$ and $t_{0}\geq 0$,
		\begin{equation}\label{mean01}
			\|\phi_\orb(t_{0}+s)\|_{L^{\infty}(Y_\orb)}\leq \breve{C}_{1}(\frac{1}{s^{2n+1}}+1)\|\phi_\orb(t_{0})\|_{L^1(Y_\orb,\w_\orb)},
		\end{equation}where $\breve{C}_{1}$ is a positive constant depending only on $n$, $\breve{C}$ and the uniform Sobolev constant $C_S$.
		
		(2) Let $Z_\orb \subset Y_\orb$ be a closed
		subset, $\bar{\Omega}_\orb$ be a compact subset of $Y_\orb\setminus Z_\orb$, and $\hat{\omega}_\orb$ be another orbifold Hermitian metric on $Y_\orb$. If there exist positive constants $\delta_{0}$ and $C_{\bar{\Omega}_\orb}$ such that $\overline{B_{\delta_{0}}(\bar{\Omega}_\orb)}\subset Y_\orb\setminus Z_\orb$ and $C_{\bar{\Omega}_\orb}^{-1}\hat{\omega}_\orb\leq \omega_\orb \leq C_{\bar{\Omega}_\orb} \hat{\omega}_\orb$ on $B_{\delta_{0}}(\bar{\Omega}_\orb)$, where $B_{\delta}(\bar{\Omega}_\orb)$ is the $\delta$-neighborhood of $\bar{\Omega}_\orb$ with respect to $\hat{\omega}_\orb$. Then for all $0\leq \delta_{1}< \delta_{2} \leq \delta_{0}$,
		\begin{equation}\label{221}
			\sup_{\overline{B_{\delta_{1}}(\bar{\Omega}_\orb)}\times [0, T]}\phi_\orb (x, t)\leq \breve{C}_{2}((\delta_{2}-\delta_{1})^{-1}+1)^{2n+1}(\int_{0}^{T}\|\phi_\orb(t)\|_{L^1(Y_\orb,\omega_\orb)}dt+\|\phi_\orb(0)\|_{L^{\infty}(B_{\delta_{2}}(\bar{\Omega}_\orb))}),
		\end{equation}
		 where $\breve{C}_{2}$ is a constant depending only on $n$, $\breve{C}$, the Sobolev constant $C_S$ and $\Vol(Y_\orb,\w_\orb)$.
	\end{lemma}

	\subsubsection{Proof of Proposition \ref{lem 2.6}}
	The following estimates were proved by Simpson for manifolds, and the proof extends to the orbifold setting by working locally on each orbifold chart $U_\alpha$ (cf. \cite[Lemma 6.1]{Simpson88}).
	
	\begin{lemma}
		Along  the heat flow (\ref{DDD1}), we have
		\begin{equation}\label{F1}
			(\Delta_{ {\w_{\orb,\epsilon}}}-\frac{\partial }{\partial t})\tr (\Phi (H_{\orb,\epsilon}(t), \omega_{\orb,\epsilon}))=0,
		\end{equation}
		\begin{equation}\label{F2}
			(\Delta_{ \w_{\orb,\epsilon}}-\frac{\partial }{\partial t} )|\Phi (H_{ \orb,\epsilon}(t), \omega_{\orb,\epsilon})|_{H_{\orb,\epsilon}(t)}^{2}=2|D_{H_{\orb,\epsilon}, \phi} (\Phi (H_{\orb,\epsilon}(t), \omega_{\orb,\epsilon}))|^{2}_{H_{\orb,\epsilon}(t), \omega_{\orb,\epsilon}},
		\end{equation}
		and
		\begin{equation}\label{H00}
			(\Delta_{\w_{\orb,\epsilon}} -\frac{\partial }{\partial t}) |\Phi (H_{\orb,\epsilon}(t), \omega_{\orb,\epsilon})|_{H_{\orb,\epsilon}(t)}\geq  0.
		\end{equation}
	\end{lemma}

	By (\ref{initial1}) and (\ref{H00}), we have
	\begin{equation}\label{H001}
		\int_{Y_\orb}|\Phi (H_{\orb,\epsilon}(t_{2}), \omega_{\orb,\epsilon}) |_{H_{\orb,\epsilon}(t_{2})}\frac{\omega_{\orb,\epsilon }^{n}}{n!}\leq \int_{Y_\orb}|\Phi (H_{\orb,\epsilon}(t_{1}), \omega_{\orb,\epsilon}) |_{H_{\orb,\epsilon}(t_{1})}\frac{\omega_{\orb,\epsilon }^{n}}{n!}\leq \hat{C}_{1}
	\end{equation}
	for any $0\leq t_{1}\leq t_{2}$, where $\hat{C}_{1}$ is a constant independent of $\epsilon $. Since $\w_{\orb,k,\epsilon}$ is strictly positive outside $f_\orb^{-1}(\Sigma)$, the same holds for $\w_{\orb,\epsilon}$. By combining the uniform Sobolev inequalities (\ref{Sobolev}), (\ref{equa-metric-localbound}) with \eqref{H001}, and applying Lemma \ref{ml1} and Lemma \ref{l1}, we have
	
	\begin{lemma}\label{m02} There exists a constant $\grave{C}_{1}$  which is independent of $\epsilon $,  such that
		\begin{equation}
			\|\Phi (H_{\orb,\epsilon}(t_{0}+s), \omega_{\orb,\epsilon})\|_{L^{\infty}(Y_\orb,H_{\orb,\epsilon}(t_{0}+s))}\leq \grave{C}_{1}(\frac{1}{s^{2n+1}}+1)\|\Phi (H_{\orb,\epsilon}(t_{0}), \omega_{\orb,\epsilon})\|_{L^1(Y_\orb,H_{\orb,\epsilon}(t_{0}))}
		\end{equation}
		for any $s>0$ and $t_{0}\geq 0$. Furthermore, for any compact subset $\bar{\Omega}_\orb$ of $Y_\orb\setminus f_\orb^{-1}(\Sigma) $, there exists a constant $\grave{C}_{2}(\bar{\Omega}_\orb)$ independent of $\epsilon $, such that
		\begin{equation}\label{220}
			|\Phi (H_{\orb,\epsilon}(t), \omega_{\orb,\epsilon})|_{H_{\orb,\epsilon}(t)}(x)\leq \grave{C}_{2}(\bar{\Omega}_\orb)
		\end{equation}
		for all $(x, t)\in \bar{\Omega}_\orb\times [0, \infty )$.
	\end{lemma}
	
	Setting $\exp (s_{\orb,\epsilon}(t))=h_{\orb,\epsilon}(t)=\hat{H}_\orb^{-1}H_{\orb,\epsilon }(t),$
	we have
	\begin{equation}\label{c01}
		\begin{split}
			\frac{\partial}{\partial t}\log(\tr h_{\orb,\epsilon}(t)+\tr h_{\orb,\epsilon}^{-1}(t))
			=&\ \frac{\tr(h_{\orb,\epsilon}(t)\cdot h_{\orb,\epsilon}^{-1}(t)\frac{\partial h_{\orb,\epsilon}(t)}{\partial t})-\tr(h_{\orb,\epsilon}^{-1}(t)\frac{\partial h_{\orb,\epsilon}(t)}{\partial t}\cdot h_{\orb,\epsilon}^{-1}(t))}{\tr h_{\orb,\epsilon}(t)+\tr h_{\orb,\epsilon}^{-1}(t)}\\
			\leq &\ 2|\Phi (H_{\orb,\epsilon}(t), \omega_{\orb,\epsilon})|_{H_{\orb,\epsilon}(t)},
		\end{split}
	\end{equation}
	and
	\begin{equation}\label{c02}
		\log (\frac{1}{2r}(\tr h_{\orb,\epsilon}(t) + \tr h_{\orb,\epsilon}(t)^{-1}))\leq |s_{\orb,\epsilon}(t)|_{\hat{H}}\leq r^{\frac{1}{2}}\log (\tr h_{\orb,\epsilon}(t) + \tr h_{\orb,\epsilon}(t)^{-1}),
	\end{equation}
	where $r=\rank (\mE_X)$. By (\ref{H001}) and (\ref{220}), we have
	\begin{equation}\label{C0a}
		\int_{Y_\orb}\log(\tr h_{\orb,\epsilon}(t)+\tr h_{\orb,\epsilon}^{-1}(t))-\log(2r)\frac{\omega_{\orb,\epsilon}^{n}}{n!}\leq 2\hat{C}_{1}t,
	\end{equation}
	\begin{equation}\label{C0b}
		\log(\tr h_{\orb,\epsilon}(t)+\tr h_{\orb,\epsilon}^{-1}(t))-\log(2r)\leq 2\grave{C}_{2}(\bar{\Omega}_\orb)t
	\end{equation}
	and\begin{equation}\label{C01}
		|s_{\orb,\epsilon}(t)|_{\hat{H}}(x)\leq 2r^{\frac{1}{2}}\tilde{C}_{2}(\bar{\Omega}_\orb)t +r^{\frac{1}{2}}\log(2r)
	\end{equation}
	for all $(x, t) \in \bar{\Omega}_\orb\times [0, \infty )$.
	
	
	\medskip


	By using the uniform local $C^{0}$-estimate (\ref{C01})  and following the argument in \cite[Lemma 2.5]{LZZ},
	we can derive uniform local  $C^{\infty}$-estimates of $h_{\orb,\epsilon}(t)$. The only difference is that the Riemann curvature is replaced by the Chern curvature of $\omega_{\orb,\epsilon}$, so we omit the proof.
	
	\begin{lemma}\label{lem 2.5}
		Let $\bar{\Omega}_\orb$ be a compact subset of $Y_\orb\setminus f_\orb^{-1}(\Sigma)$ and  $T>0$. Then for every integer $j\geq 0$, there exists a constant $\bar{C}_{j+2}= \bar{C}_{j+2}(\bar{\Omega}_\orb, T)$ which is independent of $\epsilon $, such that
		\begin{equation}
			\max_{(x, t)\in \bar{\Omega}_\orb\times [0, T]}|\nabla_{\hat{H}_\orb}^{j+2}h_{\orb,\epsilon}(t)|_{(\hat{H}_\orb, \omega_{\orb,\epsilon})}(x)\leq \bar{C}_{j+2}
		\end{equation}
		for all $0< \epsilon \leq 1$, where $\nabla_{\hat{H}_\orb}$ is the covariant derivative with respect to the Chern connection $D_{\hat{H}_\orb}$ of $\hat{H}$ and the Chern connection $D_{\omega_{\orb,\epsilon}}$ of $\omega_{\orb,\epsilon}$.
	\end{lemma}

	Note that $\w_{\orb,\epsilon}$ converges to $f_\orb^*\w$ by the construction (\ref{omegai}) in $C_{loc}^\infty(Y_\orb\setminus f_\orb^{-1}(\Sigma))$ and
	$$\lambda=\lim\limits_{\epsilon\rightarrow0}\lambda_{\epsilon}=\frac{2\pi}{\int_{Y_\orb}(f_\orb^*\w)^{n}}\mu_{(f^*\w_0,\cdots,f^*\w_{n-2})}(E_\orb)=\frac{2\pi}{\int_{Y_\orb}(f_\orb^*\w)^{n}}\mu_{(\w_0,\cdots,\w_{n-2})}(\mE_{X}),$$
	where the last equality follows from \eqref{equa-firstchernclass-reduction} and $\hat{H}$ denotes the descent of $\hat{H}_\orb|_{Y_\orb\setminus f_\orb^{-1}(\Sigma)}$.
	From the above local uniform $C^{\infty}$-bounds on $H_{\orb,\epsilon}(t)$, we get, by choosing a subsequence, that $H_{\orb,\epsilon}(t)$ converges in $C_{loc}^\infty(Y_\orb,f_\orb^{-1}(\Sigma))$ to smooth metric $H_\orb(t)$ as $\epsilon\rightarrow 0$, which satisfies \eqref{SSS1}. Since $f_\alpha$ is \'etale outside $f_\alpha^{-1}(\Sigma)$, by descending $H_\orb(t)|_{Y_\orb\setminus f_\orb^{-1}(\Sigma)}$ to a smooth Hermitian metric $H(t)$ on $\mE_{X_\reg}|_{X\setminus \Sigma}$, the proof of Proposition \ref{lem 2.6} is thereby completed.

	\subsection{Limiting behavior of $H(t)$}\label{section-HYM2}
	Note that $(X\setminus\Sigma,\w)$ is a non-compact balanced manifold since $\w^{n-1}=\w_0\wedge\cdots\wedge\w_{n-2}$. Then $(\mathcal{E}_{X_\reg}, \theta_{X_\reg})|_{X\setminus \Sigma}$ can be seen as a Higgs bundle over $(X\setminus \Sigma,\omega)$. For simplicity, we write $\theta_{X_\reg}$ as $\theta$ for now. Let us recall some basic lemmas.

	\begin{lemma}[Cutoff functions, see e.g. {\cite[Lemma 5.6]{Pan}}]\label{cutoff}
		Let $(X, \omega)$ be a compact Hermitian space of dimension $n$, and let $\Sigma \subset X$ be an analytic subspace of codimension $\geq k$. Then there exist cutoff functions $\{\rho_{i} \}_{i\in \N}$ with $supp(\rho_{i})\subset X\setminus \Sigma$, increasing to $1_{X\setminus \Sigma }$, and a positive constant $C'$ such that
		$\int_{X}\left(|d\rho_{i}|^{2k}_{\omega}+|\partial\bar{\partial }\rho_{i}|^{k}_{\omega}\right)\omega^{n}\leq C'$ for all $i\in \N$, and
		\begin{equation}\label{cutoff2}\lim_{i\rightarrow +\infty}\int_{X}\left(|d\rho_{i}|^{2k-\delta }_{\omega}+|\partial\bar{\partial }\rho_{i}|^{k-\delta }_{\omega}\right)\omega^{n}=0
		\end{equation} for any sufficiently small $\delta >0$.
	\end{lemma}
	
	\begin{lemma}\label{lem-sobolev}
		In light of Proposition \ref{prop-uniformsobolev} and Lemma \ref{cutoff}, the uniform Sobolev inequality holds for $(X\setminus\Sigma,\w)$.
	\end{lemma}
	
	Using the above cutoff functions and Lemma \ref{lem-sobolev}, we obtain the following mean value inequality by applying the standard Moser iteration as in \cite[Lemma 5.2]{LZZ} and the Stokes formula.
	\begin{lemma}[Mean value inequality]\label{lemma-meanvalue}
		Let $(X,\w)$ be a compact balanced space of dimension $n$ and $\Sigma\subset X$ be an analytic subspace of codimension $\geq2$. Suppose that $u\in C^2(X\setminus \Sigma)$ is a bounded positive function satisfying $\Delta_\w u\geq -A_0$ for some positive constant $A_0$. Then there exists a positive constant $\widetilde{C}$ depending only on $C_S$, $A_0$ and $\Vol(X,\w)=\int_{X_\reg}\frac{\w^n}{n!}$ such that
		$$\|u\|_{L^\infty(X\setminus\Sigma)}\leq \widetilde{C}(\|u\|_{L^1(X\setminus \Sigma,\w)}+1).$$
	\end{lemma}
	
	\begin{lemma}[Stokes formula]\label{stokes1}
		Let $(X, \omega )$ be a compact Hermitian space of dimension $n$ and $\Sigma \subset X$ be an analytic subspace of codimension $\geq2$. Suppose $\eta$ is a $(2n-1)$-form with
		$\int_{X\setminus \Sigma}|\eta|_{\omega }^2\frac{\omega ^{n}}{n!}< \infty$. Then if $\mathrm{d}\eta$ is integrable, we have $\int_{X}\mathrm{d}\eta=0.$
	\end{lemma}
	
	These lemmas are used to fulfill the roles of three assumptions in \cite{Simpson88} and then we can follow the strategy of \cite{Simpson88} as in \cite[Section 4]{LZZ}. The remainder of the proof follows by a slight adaptation of the argument in\cite[Section 4]{LZZ} to our setting. We sketch it here for completeness.
	
	\subsubsection{Donaldson's functional }
	
	Let $K$ and $H$ be two Hermitian metrics on $\mE |_{X\setminus \Sigma}$, and set  $\exp(s)=h=K^{-1}H$.
	By Lemma 3.1 in \cite{Simpson88}, we have
	\begin{equation}\label{la02}
		\Delta_{\omega }\log(\tr h+\tr h^{-1})\geq -2|\Lambda_{\omega }(F_{H, \theta })|_{H}-2|\Lambda_{\omega }(F_{K, \theta })|_{K},
	\end{equation}
	and
	\begin{equation}\label{la021}
		\Delta_{\omega , K}'h=h\sqrt{-1}\Lambda_{\omega }(F_{K, \theta }-F_{H, \theta })-\sqrt{-1}\Lambda_{\omega }(\bar{\partial }_{\theta }h \wedge h^{-1}D_{K, \theta}^{1,0}h),
	\end{equation}
	where $\Delta_{\omega , K}'=-\sqrt{-1}\Lambda_{\omega }\bar{\partial }_{\theta }D_{K, \theta}^{1,0}$.
	
	Let $S_{K}=S_{K}(\mathcal{E}_X|_{X\setminus \Sigma})$ be the real vector bundle of $K$-self-adjoint endomorphisms of $\mathcal{E}|_{X\setminus \Sigma}$, and $\mathcal{P}(S_{K})$ be the normed space of smooth sections $s\in \Gamma (S_{K})$ with norm
	\begin{equation}
		\|s\|_{P}=\sup_{X\setminus \Sigma}|s|+\|\bar{\partial }_{\theta }s\|_{L^{2}(X\setminus\Sigma)}+\|\Delta_{\omega , K}'s\|_{L^{1}(X\setminus\Sigma)}.
	\end{equation}
	Let $\mathcal{P}$ denote the space of smooth Hermitian metrics $K$ such that $\int_{X\setminus \Sigma }|\Lambda_{\omega }F_{K, \theta }|_{\omega}\frac{\omega^{n}}{n!}<\infty.$ Due to Proposition 4.1 in \cite{Simpson88}, we can endow $\mathcal{P}$ with the structure of an analytic manifold with local charts $s\mapsto Ke^{s}$, where $s\in \mathcal{P}(S_{K})$. Let $\mathcal{P}_{0}$ be one of the components covered by these charts.
	Let us recall Donaldson's functional  defined on the space $\mathcal{P}_{0}$ (see Section 5 in \cite{Simpson88} for details),
	\begin{equation}\label{7}
		\mu_{\omega} (K, H) = \int_{X\setminus \Sigma } \tr (s \sqrt{-1}\Lambda_{\omega }F_{K, \theta })+ \langle \Psi (s) (\overline{\partial }_{\theta} s) , \overline{\partial}_{\theta} s \rangle_{K}\frac{\omega^{n}}{n!},
	\end{equation}
	where $\Psi (x, y)= (x-y)^{-2}(e^{y-x }-(y-x)-1)$.
	
	The following proposition is standard because the argument of \cite[Section 5]{Simpson88} can be directly applied when $d\w^{n-1}=0$.
	
	\begin{proposition}
		Let $H(t)=Kh(t)$ for $0\leq t\leq 1$ be a piecewise differentiable curve connecting metrics $K$ and $H$ in the space $\mathcal{P}_{0}$. If $d\omega^{n-1}\equiv 0$, then
		\begin{equation}\label{D01}
			\mu_{\omega } (K, H)=\int_{0}^{1}\int_{X\setminus \Sigma }\sqrt{-1}\tr (h^{-1}\frac{\partial h}{\partial t} \cdot F_{H(t), \theta })\wedge \frac{\omega^{n-1}}{(n-1)!} dt,
		\end{equation}
		and
		\begin{equation}\label{D02}
			\frac{d}{dt}\mu_{\omega } (K, H(t))=\int_{X\setminus \Sigma }\sqrt{-1}\tr (h^{-1}\frac{\partial h}{\partial t} \cdot F_{H(t), \theta })\wedge \frac{\omega^{n-1}}{(n-1)!}
		\end{equation}
	\end{proposition}

	\subsubsection{Approximate Hermitian-Einstein structure}
	
	Let $H(t)$ be the long time solution of (\ref{SSS1}) on the Higgs bundle $(\mathcal{E}_{X_\reg}, \theta_{\reg})|_{X\setminus\Sigma}$.
	Since we know that $|\Lambda_{\omega } F_{H(t), \theta_{X_\reg} }|_{H(t)}$ is uniformly bounded for $t\geq t_{0}>0$, it follows from (\ref{la021}) and Lemma \ref{cutoff} that $H(t)$ (for every $t>0$)  belongs to the space $\mathcal{P}_{0}$. According to (\ref{D02}),
	a formula for the derivative of Donaldson's functional with respect to $t$ is given by
	\begin{eqnarray}\label{F5}
		\frac{d}{dt}\mu_{\omega } (H(t_{0}), H(t)) = -2\int_{X\setminus \Sigma }|\Phi (H(t), \omega )|_{ H(t)}^{2}\frac{\omega^{n}}{n!}.
	\end{eqnarray}
	Set $\exp{s(t)}=h(t)=\hat{H}^{-1}H(t)$ and $\exp{s(t_{1}, t_{2})}=h(t_{1}, t_{2})=H^{-1}(t_{1})H(t_{2}).$
	By the inequalities (\ref{la02}), (\ref{c02}), (\ref{H0013}), (\ref{H00013}) and Lemma \ref{lemma-meanvalue},  we have
	\begin{equation}\label{L101}
		r^{-\frac{1}{2}}\|s(t_{1}, t_{2})\|_{L^1 ( X\setminus \Sigma, \omega , H(t_{1}) )}\leq 2\tilde{C}_{1}(t_2-t_1)+\Vol(X, \omega )\log(2r),
	\end{equation}
	\begin{equation}\label{mean111}
		\|s(t_{1}, t_{2})\|_{L^{\infty}( X\setminus \Sigma , H(t_{1}) )}\leq r^{\frac{1}{2}}\{2C_{S}\tilde{C}_{1}(t_{0}^{-n}+1)(t_{2}-t_{1})+\log 2r\}
	\end{equation}
	and
	\begin{equation}\label{mean1}
		\|s(t_{1}, t_{2})\|_{L^{\infty}( X\setminus \Sigma , H(t_{1}))}\leq C_{1}(t_{0}^{-1})\|s(t_{1}, t_{2})\|_{L^{1}( X\setminus \Sigma, \omega , H(t_{1}) )}+C_{2}(t_{0}^{-1})
	\end{equation}
	for $0<t_{0}\leq t_{1} \leq t_{2}$, where $C_{1}(t_{0}^{-1})$ and $C_{2}(t_{0}^{-1})$ are constants depending only on $r$, the Sobolev constant $C_{S}$, $\tilde{C}_{1}$, $t_{0}^{-1}$ and $\Vol(X,\w)$.

	\begin{proposition}\label{prop01}
		Let $H(t)$ be the long time solution of (\ref{SSS1}) on the Higgs bundle $(\mathcal{E}_{X_\reg}, \theta_{\reg})|_{X\setminus\Sigma}$ with initial metric $\hat{H}$, and suppose that
		\begin{equation}
			\int_{X\setminus \Sigma}|\sqrt{-1}\Lambda_{\omega} F_{H(t), \theta_{X_\reg}}-\lambda \Id_{\mathcal{E}_{X_\reg}} |_{H(t)}^2 \frac{\omega^{n}}{n!}\leq \hat{C}_{1}',
		\end{equation}
		where $\hat{C}_{1}'$ is a positive constant. Define the limit:
		\begin{equation}\label{semi03}
			\lim\limits_{t\rightarrow 0}\int_{X\setminus \Sigma}|\sqrt{-1}\Lambda_{\omega} F_{H(t), \theta_{X_\reg}}-\lambda \Id_{\mathcal{E}_{X_\reg}} |_{H(t)}^{2} \frac{\omega^{n}}{n!}= C^{\ast}\geq 0.
		\end{equation}
		Assume further that there exists a sequence $t_{i}\rightarrow\infty$ such that
		\begin{equation}\label{CM02}
			\|s(1, t_{i})\|_{L^{1}(X\setminus \Sigma ,  H(1))}\rightarrow +\infty.
		\end{equation}
		Then the following statement holds: there is a section $ u_\infty \in L_1^2(\Gamma (S_{\hat{H}}))$ such that  $\tr u_{\infty}=0$, $\|u_\infty\|_{L^2}=1$ and the eigenvalues of $u_{\infty}$ are constants almost everywhere and not all equal. Furthermore, assume $\nu_1< \nu_2< \cdots< \nu_l$ ($l\geq 2$) are the distinct eigenvalues of $u_\infty$. For each $1\leq k \leq l-1$, define a smooth function $P_{k}: \mathbb{R}\to \mathbb{R}$ such that
			\begin{align*}
				P_{k}(x)=\left\{
				\begin{aligned}
					\ 1, \quad & x\leq \mu_{k},\\
					\ 0, \quad & x\geq \mu_{k+1}.
				\end{aligned}\right.
			\end{align*}
			Set $\pi_{k}=P_{k}(u_\infty)$. Then each $\pi_{k}$ defines a saturated $\theta_{X_\reg}$-invariant subsheaf $\mE_{X_\reg,k}$ of the Higgs sheaf $(\mathcal{E}_{X_\reg}, \theta_{X_\reg})$, and each $\mE_{X_\reg,k}$ extends to a saturated subsheaf $\mE_k$ of $\mE_X$ such that
			\begin{equation}\label{equa-contradiction}
				2\pi \sum_{k=1}^{l-1}(\mu_{k+1}-\mu_{k})\rank(\mE_k)(\mu_{(\w_0,\cdots,\w_{n-2})}(\mE_X)
				-\mu_{(\w_0,\cdots,\w_{n-2})}(\mE_k))\leq - r^{-\frac{1}{2}}\frac{C^{\ast}}{\hat{C}_{1}'}.
			\end{equation}
	\end{proposition}
	\begin{proof}
		By (\ref{mean1}) and the assumption (\ref{CM02}), it is easy to check that
		\begin{equation}\label{CM03}
			\|s(t_{0}, t_{i})\|_{L^{1}(X\setminus \Sigma ,  H(t_{0}))}\rightarrow +\infty ,
		\end{equation}
		and
		\begin{equation}\label{CM01}
			\begin{split}
				&\|s(t_{0}, t)\|_{L^{\infty}(X\setminus \Sigma , H(t_{0}))}\leq  r\|s(1, t)\|_{L^{\infty}(X\setminus \Sigma , H(1) )}+r \|s(t_{0}, 1)\|_{L^{\infty}(X\setminus \Sigma , H(1))}\\
				\leq &\ r^{2}C_{3}(\|s(t_{0}, t)\|_{L^{1}(X\setminus \Sigma ,  H(t_{0}))}+\|s(t_{0}, 1)\|_{L^{1}(X\setminus \Sigma ,  H(t_{0}))}) +r\|s(t_{0}, 1)\|_{L^{\infty}(X\setminus \Sigma , H(t_{0}))}+rC_{4}\\
			\end{split}
		\end{equation}
		for all $0<t_{0}\leq 1\leq t$, where $C_{3}$ and $C_{4}$ are constants depending only on the Sobolev constant $C_{S}$, $r$,  $\hat{C}_{1}'$ and the geometry of $(X\setminus \Sigma , \omega )$.

		Define $u_{i}(t_{0})=\|s(t_{0}, t_{i})\|_{L^{1}(X\setminus\Sigma,H(t_0))}^{-1}s(t_{0}, t_{i})\in S_{H(t_{0})}$,  then $\|u_{i}(t_{0})\|_{L^{1}(X\setminus \Sigma ,  H(t_{0}))}=1$ and
		\begin{equation}
			\int_{X\setminus\Sigma}\tr u_{i}(t_{0}) \frac{\omega^{n}}{n!} =\|s(t_{0}, t_{i})\|_{L^{1}}^{-1}\int_{t_{0}}^{t_{i}}\int_{X\setminus \Sigma }\tr (\Phi (H(t), \omega))\frac{\omega^{n}}{n!}dt=0.
		\end{equation}
		
		On the other hand, by (\ref{F5}), we have
		\begin{equation}\label{M05}
			\mu_\omega(H(t_{0}), H(t))=-2\int_{t_{0}}^t\int_{X\setminus \Sigma }|\Phi (H(s), \omega)|_{H(s)}^2\frac{\omega^n}{n!}ds\leq -2C^\ast(t-t_{0})
		\end{equation}
		for all $0<t_{0} \leq t$. Then it is clear that (\ref{L101}) implies
		\begin{equation}\label{semi01}
			\liminf_{t\rightarrow +\infty}\frac{-\mu_{\omega } (H(t_{0}), H(t))}{\|s(t_{0}, t)\|_{L^{1}(X\setminus \Sigma , H(t_{0}))}}\geq r^{-\frac{1}{2}}\frac{C^\ast}{\hat{C}_{1}'}.
		\end{equation}
		From the inequalities (\ref{semi01}), (\ref{CM02}), (\ref{CM01}) and Lemma 5.4 in \cite{Simpson88}, it follows that, after passing to a subsequence still denoted by $\{u_i(t_0)\}$, we have $u_{i}(t_{0})\rightharpoonup u_{\infty}(t_{0})$   weakly  in $L_{1}^{2}$ as $i \rightarrow \infty$, where the limit $u_{\infty}(t_{0})$ satisfies: $\|u_{\infty}(t_{0})\|_{L^{1}(X\setminus \Sigma , \omega , H(t_{0}))}=1$, $\int_{X\setminus\Sigma}\tr(u_{\infty}(t_{0}))\frac{\omega^{n} }{n!}=0$ and
		\begin{equation}\label{t01}
			\|u_{\infty}(t_{0})\|_{L^{\infty}(X\setminus \Sigma ,  H(t_{0}))}\leq r^{2}C_{3}.
		\end{equation}
		Furthermore, if $\Upsilon : \R\times \R \rightarrow \R$ is a positive smooth function such that $\Upsilon (\lambda_{1}, \lambda_{2})< (\lambda_{1}- \lambda_{2})^{-1}$ whenever $\lambda_{1}>\lambda_{2}$, then
		\begin{equation}\label{t02}
			\begin{split}
				&\int_{X\setminus \Sigma}\big(\tr (u_{\infty}(t_{0})\sqrt{-1}\Lambda_{\omega }(F_{H(t_{0}), \theta })) + \langle \Upsilon  (u_{\infty}(t_{0}))(\overline{\partial }_{\theta }u_{\infty}(t_{0})), \overline{\partial }_{\theta }u_{\infty}(t_{0}) \rangle_{H(t_{0})}\big)\frac{\omega^{n} }{n!}\leq -r^{-\frac{1}{2}}\frac{C^\ast}{\hat{C}_{1}'}.
			\end{split}
		\end{equation}

		Clearly (\ref{H0013}) and (\ref{t01}) mean that $\|\Lambda_{\omega }(F_{H(t_{0}), \theta })\|_{L^{1}(X\setminus \Sigma , H(t_{0}))}$ and $ \|u_{\infty}(t_{0})\|_{L^{\infty}(X\setminus \Sigma ,  H(t_{0}))}$  are uniformly bounded (independent of $t_{0}$). Therefore,  (\ref{t02}) yields that there exists a uniform constant $\check{C}'$ independent of $t_{0}$ such that
		$
		\int_{X\setminus \Sigma } |\overline{\partial }_{\theta}u_{\infty}(t_{0})|_{H(t_{0})}^{2}\frac{\omega^{n} }{n!}\leq \check{C}'.
		$
		From (\ref{C0b}), we know that $\hat{H}$ and $H(t_{0})$ are locally mutually bounded. By choosing a subsequence, we obtain $u_{\infty}(t_{0}) \rightarrow u_{\infty}\in L_1^2(\Gamma (S_{\hat{H}}))$ weakly in local $L_{1}^{2}$ on $X\setminus\Sigma $ as $t_{0}\rightarrow 0$, where $u_{\infty}$ satisfies
		\begin{equation}\int_{X\setminus\Sigma}\tr(u_{\infty})\frac{\omega^{n} }{n!}=0, \quad and \quad \|u_{\infty}\|_{L^{1}(X\setminus \Sigma , \hat{H})}=1.\end{equation}
		Furthermore, we have
		\begin{equation}\label{t04}
			\lim_{t_{0}\rightarrow 0 }\int_{X\setminus \Sigma}\tr (u_{\infty}(t_{0})\sqrt{-1}\Lambda_{\omega }F_{H(t_{0}), \theta}) \frac{\omega^{n} }{n!}= \int_{X\setminus \Sigma}\tr (u_{\infty}\sqrt{-1}\Lambda_{\omega }F_{\hat{H}, \theta}) \frac{\omega^{n} }{n!}.
		\end{equation}
		Set
		$\hat{u}_{\infty}(t_{0})=(h(t_{0}))^{\frac{1}{2}}\cdot u_{\infty}(t_{0}) \cdot (h(t_{0}))^{-\frac{1}{2}}.$ One can verify that: $\hat{u}_{\infty}(t_{0})\in S_{\hat{H}}(\mathcal{E}|_{X\setminus \Sigma })$ and $|\hat{u}_{\infty}(t_{0})|_{\hat{H}}=|u_{\infty}(t_{0})|_{H(t_{0})}$. Following the same argument as in the proof of {\cite[Lemma 4.2]{LZZ}}, we deduce
		\begin{equation}\label{t05}
			\lim_{t_{0}\rightarrow 0}\int_{\Omega} |\langle\tilde{\Upsilon } (u_{\infty}(t_{0}))(\overline{\partial }_{\theta }u_{\infty}(t_{0})), \overline{\partial }_{\theta }u_{\infty}(t_{0})\rangle_{H(t_{0})}-\langle\tilde{\Upsilon } (\hat{u}_{\infty}(t_{0}))(\overline{\partial }_{\theta}\hat{u}_{\infty}(t_{0})), \overline{\partial }_{\theta}\hat{u}_{\infty}(t_{0})\rangle_{\hat{H}}|\frac{\omega^{n} }{n!}=0,
		\end{equation}
		for any compact domain $\Omega \subset X\setminus \Sigma$ and any positive smooth function $\tilde{\Upsilon } : \R\times \R \rightarrow \R$.
		
	Combining (\ref{t02}), (\ref{t04}) with (\ref{t05}), for any given compact domain $\Omega \subset X\setminus \Sigma$ and any $\tilde{\epsilon}>0$, we get
		\begin{equation}\label{t06}
			\int_{X\setminus\Sigma}\tr (u_{\infty}\sqrt{-1}\Lambda_{\omega }F_{\hat{H}, \theta } )\frac{\omega^{n} }{n!}+\int_{\Omega} \langle\Upsilon  (\hat{u}_{\infty}(t_{0}))(\overline{\partial }_{\theta }\hat{u}_{\infty}(t_{0})), \overline{\partial }_{\theta }\hat{u}_{\infty}(t_{0})\rangle_{\hat{H}}\frac{\omega^{n} }{n!}\leq -r^{-\frac{1}{2}}\frac{C^\ast}{\hat{C}_{1}'}+\tilde{\epsilon}
		\end{equation}
		for sufficiently small $t_{0}$. Since $\hat{u}_\infty(t_0)\rightarrow u_\infty$ in $L^2(\Omega)$ as $t_0\rightarrow 0$, with $|\widehat{u}_\infty(t_0)|_{\widehat{H}}$ uniformly bounded in $L^\infty$ and $|\bp_\theta\widehat{u}_\infty(t_0)|_{\widehat{H},\w}$ uniformly bounded in $L^2(\Omega)$, the same argument as in  \cite[Lemma 5.4]{Simpson88} yields 
		\begin{equation}\label{t07}
			\int_{X\setminus\Sigma}\tr (u_{\infty}\sqrt{-1}\Lambda_{\omega }F_{\hat{H}, \theta })\frac{\omega^{n} }{n!}+\|\Upsilon^{\frac{1}{2}}  (u_{\infty})(\overline{\partial }_{\theta}u_{\infty})\|_{L^{q}(\Omega)}^{2}\leq -r^{-\frac{1}{2}}\frac{C^\ast}{\hat{C}_{1}'}+2\tilde{\epsilon}
		\end{equation}
		for any $q<2$. Since $\tilde{\epsilon}$, $q<2$ and $\Omega$ are arbitrary,  we get
		\begin{equation}\label{semi02}
			\int_{X\setminus\Sigma}\bigg(\tr (u_{\infty}\sqrt{-1}\Lambda_{\omega }F_{\hat{H}, \theta}) + \langle\Upsilon  (u_{\infty})(\overline{\partial }_{\theta}u_{\infty}), \overline{\partial }_{\theta}u_{\infty}\rangle_{\hat{H}}\bigg)\frac{\omega^{n} }{n!}\leq -r^{-\frac{1}{2}}\frac{C^\ast}{\hat{C}_{1}'}.
		\end{equation}
		
		From the above inequality and \cite[Lemma 5.5]{Simpson88}, we can see that
		the eigenvalues of $u_{\infty}$ are  constant almost everywhere. Let $\nu_{1} < \dots <\nu _{l}$ denote the distinct eigenvalues of $u_{\infty}$.  Since $\int_{X\setminus \Sigma} \tr u_{\infty} \frac{\omega^{n} }{n!}=0$ and $\|u_{\infty}\|_{L^{1}}=1$, we must have $l\geq 2$. By Uhlenbeck and Yau's regularity theorem for $L_{1}^{2}$-subbundle \cite{UY86} and Simpson's result \cite[Page 887]{Simpson88}, each $\pi_{k}$ represents a saturated $\theta_\reg$-invarient coherent subsheaf $\mE_{X_\reg,k}$ of $(\mathcal{E}_{X_\reg}, \theta_{X_\reg})$ on the open set $X\setminus \Sigma $. Since $\mE_X$ is reflexive,  the trivial extension of $\mE_{X_\reg,k}$ is a saturated subsheaf of $\mE_X$ that is $\theta_{X_\reg}$-invariant on $X_\reg$ by Lemma \ref{lem-saturation-invariant}, which will be denoted by $\mE_k$. Let $\mE_{k,\orb}$ be the saturation of $\{\imag(f_\alpha^*\mE_k\xrightarrow{f_\alpha^*}f_\alpha^*\mE_X)\cap E_\alpha\}$ in $E_\orb$. Recall from Proposition \ref{prop-HEequation-setting} that $f_\alpha$ is \'etale outside $f_\alpha^{-1}(\Sigma)$ and thus there exists an analytic subspace $Z\subset X$ of codimension $\geq 2$ such that $\mE_{k,\orb}=\{f_\alpha^*\mE_k/(\mathrm{torsion})\}$ on $Y_\orb\setminus f_\orb^{-1}(Z)$. It follows from Lemma \ref{lem-modification-chernclass}, Lemma \ref{lem-compatible-orbifold-pullback} and the Chern-Weil formula for orbifolds (cf. \cite[Section 2.2.2]{ZZZ25}) that
		\begin{align*}
			&\deg_{(\w_0,\cdots,\w_{n-2})} (\mE_k)=\deg_{f_\orb^*\w}(\mE_{k,\orb})\\
			=&\int_{Y_\orb\setminus f_\orb^{-1}(\Sigma) } \tr (\pi_{k,\orb} \sqrt{-1}\Lambda _{f_\orb^*\w}F_{\hat{H}_\orb,\theta_\orb}) -|\overline{\partial}_{\theta_\orb} \pi_{k,\orb }|_{\hat{H}_\orb, f_\orb^*\omega}^{2} \frac{(f_\orb^*\omega)^{n}}{n!}\\
			=&\int_{X\setminus \Sigma }\bigg(\tr (\pi_{k } \sqrt{-1}\Lambda _{\omega }F_{\hat{H}, \theta_{X_\reg}}) -|\overline{\partial}_{\theta_{X_\reg}} \pi_{k }|_{\hat{H}, \omega}^{2}\bigg)\frac{\omega^{n}}{n!},
		\end{align*}
		where $\pi_{k,\orb}$ denotes the weakly holomorphic subbundle associated with $\mE_{k,\orb}$. Furthermore, applying the argument in Simpson's paper \cite[Page 888]{Simpson88} and the inequality (\ref{semi02}), we conclude
		\begin{equation}
			\begin{split}
				&\int_{X\setminus\Sigma}\bigg(\tr (u_{\infty }\sqrt{-1}\Lambda _{\omega }F_{\hat{H}, \theta }) +\langle \sum_{k =1} ^{l-1} (\mu_{k +1}-\mu_{k })(dP_{k })^{2}(u_{\infty }) (\overline{\partial }_{\theta} u_{\infty}) , \overline{\partial }_{\theta} u_{\infty}\rangle _{\hat{H}}
				\bigg)\frac{\w^n}{n!}\\
				&\leq  -r^{-\frac{1}{2}}\frac{C^\ast}{\hat{C}_{1}'}.
			\end{split}
		\end{equation}
		This completes the proof of \eqref{equa-contradiction}.
	\end{proof}
	
	\subsubsection{Proof of Theorem \ref{thm1}}	
	{Let $H(t)$ be the long time solution of (\ref{SSS1}) on the Higgs bundle $(\mathcal{E}_{X_\reg}, \theta_{\reg})|_{X\setminus\Sigma}$ with initial metric $\hat{H}$.
		By Proposition \ref{prop01}, if the Higgs sheaf $(\mE_{X_\reg},  \theta_{X_\reg})$ is $(\w_0,\cdots,\w_{n-2})$-stable, then $\| \log(H^{-1}(1)H(t))\|_{L^1(X\setminus \Sigma,H(1))}$ are uniformly bounded, i.e., we obtain uniform $C^{0}$-estimates on $X\setminus \Sigma$. From the  standard parabolic estimates, we can get uniform local $C^{\infty}$-estimates. Then after passing to a subsequence, $H(t) \rightarrow H_{\infty}$ in $C^\infty_{loc}(X\setminus\Sigma)$-topology, and (\ref{F5}) implies
		\begin{equation}
			\sqrt{-1}\Lambda_{\omega} (F_{H_{\infty}} +[\theta_{X_\reg}, (\theta_{X_\reg})^{* H_{\infty}}])
			=\lambda\cdot \mathrm{Id}_{\mE_{X_\reg}}
		\end{equation}
		on $X\setminus\Sigma$. Using Proposition \ref{prop01} again, we know that if the Higgs bundle $(\mE_{X_\reg},  \theta_{X_\reg})|_{X\setminus\Sigma}$ is $(\w_0,\cdots,\w_{n-2})$-semistable, then
		\begin{equation}\label{semi030}
			\lim_{t\rightarrow \infty }\int_{X\setminus \Sigma}|\sqrt{-1}\Lambda_{\omega} F_{H(t), \theta_{X_\reg}}-\lambda \Id_{\mathcal{E}_{X_\reg}} |_{H(t)}^{2} \frac{\omega^{n}}{n!}= 0.
		\end{equation}
		Combining this with (\ref{H00013}), we obtain \eqref{equa-approximate-HE}. This completes the proof of Theorem \ref{thm1}.}
	
	\section{Characterization of harmonic bundles over the regular locus}\label{section-harmonicbundles}
     This section begins with a review of basic facts (Section~\ref{subsubsection-maximally}). It continues in Section~\ref{subsection-harmonic} with the proof of part (1) of Theorem~\ref{main-theorem-polystable-equivalence} and of the descent result for Higgs bundles from harmonic bundles (Proposition~\ref{prop-descent-higgs}). The proof of Theorem~\ref{main-theorem-polystable-equivalence} is concluded in Section~\ref{subsection-locallyfree}.

    \subsection{Basic facts}\label{subsubsection-maximally}
    We recall the following elementary facts about resolutions of singularities and maximally quasi-\'etale covers for later use.

    \begin{lemma}[cf. {\cite[Lemma 7.8]{GKPT19}} and {\cite{Taka03}}]\label{lem-resolutionofsingularities-easy}
     	Let $\pi:\widehat{X}\rightarrow X$ be a resolution of singularities of a compact K\"ahler klt space $X$. The following hold:
     	\begin{enumerate}[label=(\thetheorem.\alph*)]
     		\item \label{lem-resolutionofsingularities-easy-1} For any locally free Higgs sheaf \((\mathcal{E}_X, \theta_X)\) such that $\pi^*(\mathcal{E}_X, \theta_X)$ is semistable and satisfies \eqref{equa-vanishingcondition-locallyfree} for some K\"ahler form $\w_{\widehat{X}}$, then $(\mathcal{E}_X, \theta_X)$ is semistable with respect to any nef forms \(\eta_0, \cdots, \eta_{n-2}\) and all Chern classes \(c_i(\mathcal{E}_X) \in H^{2i}(X, \mathbb{Q})\) vanish.
     		\item \label{lem-resolutionofsingularities-easy-2} $\pi_*:\pi_1(\widehat{X})\rightarrow\pi_1(X)$ is an isomorphism. In particular, the pullback $\pi^*$ is a one-to-one map from $\mathrm{LSys}_X$ to $\mathrm{LSys}_{\widehat{X}}$ (resp. from $\mathrm{sLSys}_X$ to $\mathrm{sLSys}_{\widehat{X}}$).
     	\end{enumerate}
     \end{lemma}
     
     \begin{proof}
     	Item \ref{lem-resolutionofsingularities-easy-1} follows from the same proof as {\cite[Lemma 7.8]{GKPT19}}. Item \ref{lem-resolutionofsingularities-easy-2} follows from Takayama's result \cite{Taka03}, which states that the induced morphism $\pi_*:\pi_1(\widehat{X})\rightarrow \pi_1(X)$ is an isomorphism. The proof is purely analytic and thus also holds for analytic klt spaces.
     \end{proof}

     Let us follow \cite{GKP16,CGGN22,IMM24,OF25} to review the concept of a maximally quasi-\'etale cover, which plays a central role in our construction.
     \begin{definition}[Maximally quasi-\'etale, {\cite[Definition 5.3]{CGGN22}}]\label{defn-maximally}
     	A normal complex space $X$ is said to be {\em maximally quasi-\'etale} if it satisfies the following equivalent conditions:
     	\begin{itemize}
     		\item[(1)] Any quasi-\'etale cover of $X$ is \'etale.
     		\item[(3)] Any \'etale cover of $X_\reg$ extends to an \'etale cover of $X$.
     		\item[(3)] The morphism of \'etale fundamental groups   $\widehat{\pi}_1(X_{\reg})\rightarrow \widehat{\pi}_1(X)$ induced by the natural inclusion $i:X_{\reg}\hookrightarrow X$ is an isomorphism.
     	\end{itemize}
     \end{definition}
     
     \begin{remark}[Extension of linear representations]\label{rem-extension}
     	If $X$ is maximally quasi-\'etale, then Malcev's theorem implies that any linear representation of $\pi_1(X_{\reg})$ factors through $\pi_1(X)$ (see e.g. the proof of \cite[Proposition 3.10]{GKP22}), i.e.,
     	\[
     	\begin{tikzcd}
     		\pi_1(X_{\reg})\arrow[r,two heads,"i_*"'] \arrow[rr,bend left=20,"\rho_0"]&\pi_1(X)\arrow[r,"\rho"']&\mathrm{GL}(r,\C).
     	\end{tikzcd}
     	\]    	
     	Since $i_*:\pi_1(X_\reg)\rightarrow \pi_1(X)$ is surjective (cf. \cite[Proposition 2.10]{Kollar14}), $\rho_0$ is semisimple if and only if $\rho$ is semisimple. This implies that any (resp. semisimple) local system on $X\setminus \Sigma$, for some analytic subspace $\Sigma\subset X$ with $\codim_X \Sigma\geq2$, extends to a (resp. semisimple) local system on $X$, based on the standard fact that $\pi_1(X_\reg\setminus \Sigma)\cong \pi_1(X_\reg)$.
     \end{remark}    
    
     \begin{theorem}[The existence of maximally quasi-\'etale covers, {see e.g. \cite[Theorems 2.4 and 2.5]{OF25}}]\label{thm-maximally}
     	Let $X$ be a compact complex klt space. Then there exists a compact klt space $Y$ and a quasi-\'etale Galois morphism $\gamma:Y\rightarrow X$ such that $Y$ is maximally quasi-\'etale.
     \end{theorem}
    \begin{remark}
	    The existence of a maximally quasi-\'etale cover for a projective klt variety was proved in \cite{GKP16}. Building on Fujino's recent work \cite{Fuji22} on the relative MMP for projective morphisms between analytic spaces, the existence of maximally quasi-\'etale covers for analytic klt spaces can be obtained by following the same method as in \cite{GKP16}.
    \end{remark}

    \begin{lemma}\label{lem-maximallyquasietale-easy}
    	Let $X$ be a compact K\"ahler klt space and $\gamma:Y\rightarrow X$ be a maximally quasi-\'etale cover induced by some subgroup $G\subset \Aut(Y)$. Then the following hold:
    	\begin{enumerate}[label=(\thetheorem.\alph*), itemsep=2pt, parsep=2pt]
    		\item \label{lem-maximallyquasietale-easy-kahler} For any K\"ahler class $\w$ on $X$, $\gamma^*\w$ is K\"ahler.
    		\item \label{lem-maximallyquasietale-easy-0} $\gamma$ is finite and \'etale over $X_\reg$. In particular, $\gamma^{-1}(X_\reg)\subset Y_\reg$.
    		\item \label{lem-maximallyquasietale-easy-1}  $(\mE_{X_\reg},\theta_{{X_\reg}})\in\mathrm{Higgs}_{X_\reg}$ if and only if $(\mE_{Y_\reg},\theta_{Y_\reg})=\gamma^{[*]}(\mE_{X_\reg},\theta_{X_\reg})\in\mathrm{Higgs}_{Y_\reg}$.
    		\item \label{lem-maximallyquasietale-easy-2} For any $E_{X_\reg}\in \mathrm{LSys}_{X_\reg}$ (resp. $\mathrm{sLSys}_{X_\reg}$), the pullback $(\gamma|_{\gamma^{-1}(X_\reg)})^*E_{X_\reg}$  extends to a $G$-equivariant local system $\rho_Y\in \mathrm{LSys}_Y$ (resp. $\mathrm{sLSys}_Y$). Conversely, the restriction of a $G$-equivariant local system $\rho_Y\in \mathrm{LSys}_Y$ (resp. $\mathrm{sLSys}_Y$) descends to a local system $\rho_{X_\reg}\in\mathrm{sLSys}_{X_\reg}$.
    	\end{enumerate}
    \end{lemma}
    
    \begin{proof}
    Item \ref{lem-maximallyquasietale-easy-kahler} can be refered to \cite[Proposition 3.6]{GK20}. Item \ref{lem-maximallyquasietale-easy-0} is a consequence of the purity of branch theorem due to Grauert-Remmert \cite{GR55}. Item \ref{lem-maximallyquasietale-easy-1} follows from Proposition \ref{lem-quasietale-chernclasses}, Lemma \ref{lem-quasietale-stability} and item \ref{lem-maximallyquasietale-easy-kahler}. Item \ref{lem-maximallyquasietale-easy-2} follows from Remark \ref{rem-extension}. The statement also holds for semisimple local systems because  $\gamma_*:\pi_1(\gamma^{-1}(X_\reg))\subset \pi_1(Y_\reg)$ is injective by item \ref{lem-maximallyquasietale-easy-0}.
    \end{proof}
    
   \subsection{Harmonic bundles and descent result}\label{subsection-harmonic}
   This subsection studies harmonic bundle structures over the  regular locus of a compact K\"ahler klt space. The following standard definition fixs the conventions (see \cite[Section 2]{WZ23} for general Riemannian setting). 
   
   Let $(E,D)$ be a flat bundle equipped with a Hermitian metric $H$ on a complex manifold $M$. There exists a unique decomposition of $D$ as
   \begin{align}\label{equa-harmonic-decomposition-0}
   	D=D_{H}+\Psi_H,
   \end{align}
   where $D_{H}$ is an $H$-connection and $\Psi_H\in \Lambda^1(M,\End(E))$ is self-adjoint with respect to $H$. We define
   \begin{equation}\label{equa-G_H}
   	G_H:=(D_H^{0,1}+\Psi_H^{1,0})^2=(D_H^{0,1})^2+D_H^{0,1}\Psi_H^{1,0}+\Psi_H^{1,0}\wedge \Psi_H^{1,0}.
   \end{equation}
   
   \begin{definition}[Harmonic bundle]
   	We say that $(E,D,H)$ is a harmonic bundle, or that the flat bundle $(E,D)$ admits a pluri-harmonic metric $H$, if $G_H=0$. This condition is equivalent to
   	\begin{equation}\label{equa-harmonic-condition}
   		(D_H^{0,1})^2=0,\quad D_H^{0,1}\Psi_H^{1,0}=0 \ \text{ and}  \ \  \Psi_H^{1,0}\wedge \Psi_H^{1,0}=0.
   	\end{equation}
   \end{definition}
   Thus, $(E,D,H)$ is a harmonic bundle if and only if $(E,\bp,\theta):=(E,D_H^{0,1},\Psi_H^{1,0})$ is a Higgs bundle and the decomposition \eqref{equa-harmonic-decomposition-0} can be written as
   \begin{align}\label{equa-harmonic-decomposition}
   	D=\p+\bp+\theta+\theta^*,
   \end{align}
   where $\p=D_H^{1,0},\bp=D_H^{0,1},\theta=\Psi_H^{1,0}$ and $\theta^*=\Psi_H^{0,1}$.
   
   \subsubsection{Pluriharmonic metrics on Harmonic bundles}
   
   This subsubsection is devoted to proving the following statement.
    
    \begin{proposition}\label{prop-harmonicmetric}
    	Let $X$ be a compact K\"ahler klt space and $Z\subset X$ be an analytic subspace with $\codim_XZ\geq2$. Suppose that $X\setminus Z$ is smooth. Then the following hold:
    	\begin{itemize}
    		\item[(1)] A flat bundle on $X\setminus Z$ is semisimple if and only if it admits a pluri-harmonic metric. Such a metric is unique up to a positive constant on each direct summand.
    		\item[(2)] Any harmonic bundle $(E_{X\setminus Z},D_{X\setminus Z},H_{X\setminus Z})$ on $X\setminus Z$ could extend to a harmonic bundle $(E_{X_\reg},D_{X_\reg},H_{X_\reg})$ on $X_\reg$.
    	\end{itemize}
    \end{proposition}
   
   Because the inclusion induces an isomorphism $i_*:\pi_1(X_\reg\setminus Z)\rightarrow\pi_1(X_\reg)$, any semisimple flat bundle on $X\setminus Z$ extends to a semisimple flat bundle on $X_\reg$. Thus (2) is a direct consequence of (1).
   
   \begin{remark}
   	  When $Z=\emptyset$, Proposition \ref{prop-harmonicmetric} rests on the work of Corlette and Donaldson on harmonic metrics \cite{Cor88,Don87}; see \cite{CJY19,Cor92,Jost-Zuo-0,Jost-Zuo,Li96,Lu99,Simpson90,Mo0,WZ23} for various generalizations.
   \end{remark}
    
    \begin{proof}[Proof of Proposition \ref{prop-harmonicmetric}]
    	The existence of a pluri-harmonic metric can be obtained via a maximally quasi-\'etale cover $\gamma: Y \rightarrow X$, induced by some subgroup $G \subset \mathrm{Aut}(Y)$, as follows. There is an extension $E_{X_\reg}\in\mathrm{sLSys}_{X_\reg}$ of the local system corresponding to the flat bundle $(E^\circ,D^\circ)$ on $X\setminus Z$  because $\codim_{X_\reg}(X_\reg\setminus Z)\geq2$. By item \ref{lem-maximallyquasietale-easy-2}, the pullback $(\gamma|_{\gamma^{-1}(X_\reg)})^*E_{X_\reg}$ extends to a $G$-equivariant local system $E_Y\in\mathrm{sLSys}_Y$. Let $\pi: \widehat{Y} \rightarrow Y$ be a functorial resolution of singularities, so $G$ lifts to a subgroup $\widehat{G}\subset\mathrm{Aut}(\widehat{Y})$. Then $\pi^*E_Y$ belongs to $\mathrm{sLSys}_{\widehat{Y}}$ by item \ref{lem-resolutionofsingularities-easy-2} and is also $\widehat{G}$-equivariant. Applying Corlette-Donaldson's result \cite{Cor88,Don87} over the compact K\"ahler manifold $\widehat{Y}$, we obtain a harmonic bundle $(E_{\widehat{Y}}, D_{\widehat{Y}}, H_{\widehat{Y}})$ corresponding to $\pi^*E_Y$. This harmonic bundle is $\widehat{G}$-equivariant due to the uniqueness of pluriharmonic metrics. Then $H_{\widehat{Y}}$ could descend to $X\setminus Z$ and induces a pluri-harmonic metric on $(E^\circ,D^\circ)$.
    	
    	It remains to verify the uniqueness and to show that the existence of a pluri-harmonic metric implies semisimplicity. The standard proof in the noncompact K\"ahler context via the integral method is not available because we do not impose any assumption such as ``admissibility'' on the pluri-harmonic metrics. We adapt a new argument via the maximum principle, which is well‑suited to our setting thanks to Grauert-Remmert's extension theorem \cite{GR55}.
    	
    	\begin{lemma}\label{lem-harmonic-decomposition}
    		Let $(E,D,H)$ be a harmonic bundle on a complex manifold $M=X\setminus Z$ for some compact analytic space $X$ and an analytic subspace $Z$ of codimension at least $2$. Then for any $D$-invariant subbundle $W$ of $E$ such that $(W,D_W)$ is simple, where $D_W=D|_W$, it holds that
    		\begin{equation}
    			(E,D,H)=(W,D_W,H_W)\oplus (W^\perp,D_{W^\perp},H_{W^\perp})
    		\end{equation}
    		splits as harmonic bundles, where $H_W=H|_{W}$, $W^\perp$ is the orthogonal complement of $W$ with respect to $H$ and $D_{W^\perp}$ (resp. $H_{W^\perp}$) is the induced connection (resp. the induced Hermitian metric) on $W^\perp$.
    	\end{lemma}
    	
    	\begin{proof}[Proof of Lemma \ref{lem-harmonic-decomposition}]
    		With respect to the orthogonal decomposition $E=W\oplus W^\perp$, the connection $D$ takes the form
    		\begin{equation}
    			D=\left(\begin{matrix}
    				D_W & \beta_W \\
    				0 & D_{W^\perp}
    			\end{matrix}\right),
    		\end{equation}
    		where $\beta_W\in \Lambda^1(M,\Hom(W^\perp,W))$. Recall from \eqref{equa-harmonic-decomposition} that $D$ admits a unique decomposition $D_H+\Psi_H$. One can verify that $D_H$ and $\Psi_H$ are then given by
    		\begin{equation}\label{equa-matrix-decompose}
    			D_H=\left(\begin{matrix}
    				D_{H_W} & \frac{1}{2}\beta_W \\
    				-\frac{1}{2}\beta^*_W & D_{H_{W^\perp}}
    			\end{matrix}\right)
    			\ \text{ and }\ 
    			\Psi_H=\left(\begin{matrix}
    				\Psi_{H_W} & \frac{1}{2}\beta_W \\
    				\frac{1}{2}\beta^*_W & \Psi_{H_{W^\perp}}
    			\end{matrix}\right),
    		\end{equation}
    		where we have used the decompositions $D_W=D_{H_W}+\Psi_{H_W}$ and $D_{W^\perp}=D_{H_{W^\perp}}+\Psi_{H_{W^\perp}}$ (cf. \cite[Section 2.3]{WZ23}). Recall that the pluri-harmonicity of $H$ is equivalent to \eqref{equa-harmonic-condition}. Combining this with \eqref{equa-matrix-decompose} gives
    		\begin{equation}\label{equa-sub-harmonic-2}
    			D_{H_W}^{0,1}\Psi_{H_W}^{1,0}-\frac{1}{4}\beta_W^{1,0}\wedge (\beta_W^{1,0})^*+\frac{1}{4}\beta_W^{0,1}\wedge(\beta_{W}^{0,1})^*=0.
    		\end{equation}
    		Recall that the simple flat bundle $(W,D_W)$ admits a pluri-harmonic metric $K$, as explained above. Then $D_{K}^{0,1}\Psi_{K}^{1,0}=0$, and together with \eqref{equa-sub-harmonic-2} this implies
    		\begin{equation}\label{equa-vanishing-0,1}
    			\begin{split}
    				\im\pp\log\det(h)=&2\im\tr(D_{H_W}^{0,1}\Psi_{H_W}^{1,0}-D_{K}^{0,1}\Psi_{K}^{1,0})\\
    				=&\frac{1}{2}\im\tr\left(\beta_W^{1,0}\wedge (\beta_W^{1,0})^*-\beta_W^{0,1}\wedge(\beta_{W}^{0,1})^*\right)
    				\geq0,
    			\end{split}
    		\end{equation}
    		where $h=K^{-1}H_W$ (cf. \cite[(2.25)]{PZZ23} for the derivation of the first equality). Consequently, $\log\det(h)$ is a plurisubharmonic function on $M$ and thus extends to a plurisubharmonic function on $X$ due to Grauert-Remmert's extension theorem  \cite[Page 181]{GR55}, since $\codim_XZ\geq2$. Hence, $\log\det(h)$ must be constant on $X$. Using \eqref{equa-vanishing-0,1} again, we deduce that $\beta_W^{1,0}=\beta_W^{0,1}=0$, and so $\beta_W=0$, which completes the proof.
    	\end{proof}
    	
    	To prove the uniqueness, it suffices to consider the case when $(E,D)$ is simple by Lemma \ref{lem-harmonic-decomposition}. Suppose that $H$ and $K$ are two pluri-harmonic metrics on  $(E,D)$. Then we have
    	\begin{equation}\label{lemma-pp}
    		\im\pp\tr(h)=\frac{1}{2}\im\langle Dh\cdot h^{-\frac{1}{2}},\overline{Dh\cdot h^{-\frac{1}{2}}}\rangle_K+2\tr(h\im (G_H-G_K)),
    	\end{equation}
    	where $h=K^{-1}H\in C^\infty(M,\End(E))$ with $M=X\setminus Z$ (cf. \cite[Section 2]{PZZ23}). Consequently, $\tr(h)$ is a global plurisubharmonic function on $M$ because $G_H=0$ and $G_K=0$ (cf. \eqref{equa-harmonic-condition}), and thus, as argued above, must be constant. Applying \eqref{lemma-pp} again gives $Dh=0$ over $M$. Then $h=c\Id_{E}$ for some constant $c>0$. Otherwise, $(E,D)$ would split according to the eigenspaces of $h$. Therefore, $H=cK$.

    	It remains to show that the existence of a pluri-harmonic metric $H$ implies the semisimplicity of $(E,D)$. This can be concluded by Lemma \ref{lem-harmonic-decomposition} and an induction argument, because 
    	any flat bundle $(E,D)$ contains a $D$-invariant simple subbundle. Hence the proof of Proposition \ref{prop-harmonicmetric} is complete.
    \end{proof}

    \subsubsection{Descent results}
    
    A key technical step in proving Theorem \ref{main-theorem-polystable-equivalence} (2) is the following descent result.
    
    \begin{proposition}\label{prop-descent-higgs}
    	Let $\pi:\widehat{X}\rightarrow X$ be a resolution of singularities of a complex klt space $X$. Then for any Higgs bundle $(\mE_{\widehat{X}},\theta_{\widehat{X}})$ arising from a harmonic bundle $(E,D,H)$, there exists a locally free Higgs sheaf $(\mE_X,\theta_X)$ on $X$ such that $(\mE_{\widehat{X}},\theta_{\widehat{X}})=\pi^*(\mE_X,\theta_X)$, where $\mE_{X}:=(\pi_*\mE_{\widehat{X}})^{\vee\vee}$.
    \end{proposition}
    
    Thanks to Daniel's work on variations of loop Hodge structures \cite{Daniel17}, we are able to study harmonic bundles using the classical tools of Hodge theory, in particular the existence of a period map, which encodes the structure of the induced Higgs sheaf as follows.
    
    \begin{lemma}\label{lem-periodmap-pullback}
    	Let $(E,D,H)$ be a harmonic bundle over a simply connected complex manifold $X$ and $f:X \rightarrow\mathcal{D}$ be the induced period map introduced in \cite[Section 3.2]{Daniel17}. Write $D=\p+\bp+\theta+\theta^*$ according to the decomposition \eqref{equa-harmonic-decomposition} and set $\mE_X:=(E,\bp)$. Then there exists a locally free coherent sheaf $\mE_D$ on $\mathcal{D}$ such that $\mE_X:=f^*\mE_D$.
    \end{lemma}
    
    To see this, we briefly outline Daniel's construction of the period map. For a concise formulation, we refer to \cite{Daniel17}; below we summarize the essential ideas.
    
    \begin{construction}[The period map, cf. {\cite[Section 3]{Daniel17}}]
    	Fix a base point $x_0\in X$ and let $\mathcal{H}=L^2(S^1, \End(E))$. Then the fiber at $x_0$ of the canonical outgoing Krein bundle $(\mathcal{H}, \mathcal{B}, \mathcal{T})$ induced by $(E, H)$  can be identified with the canonical Krein space $\widehat{H}:=(L^2(S^1,\C^r),B,T)$ via the identification of $(E, H)_{x_0}$ with the standard Hermitian inner product on $\C^r$ (see \cite[Section 2.1.3]{Daniel17}), where $r=\rank(E)$ and $S^1=\{z\in\C:|z|=1\}$. 
    	
    	When there is no risk of confusion, the trivial bundle $Y\times \widehat{H}$ over a complex manifold $Y$ will be denoted simply by $\widehat{H}_Y$. Set $$D_\lambda=\p+\bp+\lambda^{-1}\theta+\lambda\theta^*$$ for any $\lambda\in S^1$. Then $\{D_\lambda\}$ can be regarded as a flat connection $\widetilde{D}$ on $\mathcal{H}$; moreover, $\mathcal{B}$ and $\mathcal{T}$ are $\widetilde{D}$-flat (see \cite[Appendix A]{Daniel17}). In particular, the canonical Krein space with the flat connection $\widetilde{D}$ admits a trivialization
    	$$(\mathcal{H},\mathcal{B},\mathcal{T},\widetilde{D})\cong(\widehat{H}_X,B_X,T_X,d_X),$$ where $B_X$ (resp. $T_X$) is the trivial Krein metric (resp. trivial outgoing operator) by identification with the base point, and $d_X$ is the trivial connection.
    	
    	Via Fourier series expansion, any element $s\in \mathcal{H}$ can be written as $s(\lambda)=\sum\limits_{n\in\mathbb{Z}}s_n \lambda^n$ with $(s_n(x))_{n\in\mathbb{Z}}\in l^2({\C}^r)$ for each $x\in X$. The canonical outgoing subbundle is given by $\mathcal{W}=L^2_+(S^1,\End(E))$ (cf. \cite[Section 2.1.3]{Daniel17}), consisting of all Fourier series with nonnegative degrees. In particular, for any $x$, $\mathcal{W}_x$ can be identified with an outgoing subspace $W$ of $\widehat{H}$, and this subspace varies with $x\in X$. In summary, we obtain an isomorphism of variations of loop Hodge structures
    	\begin{equation}\label{equa-identification-loophodgestructure}
    		\gamma:\left((\mathcal{H},\mathcal{B},\mathcal{T}),\mathcal{W},\widetilde{D}\right)\cong(\widehat{H}_X,W_X,d_X),
    	\end{equation}
    	where $\widehat{H}_X=\widehat{H}\times X$ is trivial with respect to $d_X$ while $W_X$ is not. 
    	
    	Let $\mathrm{Gr}(\widehat{H})$ denote the outgoing Grassmannian of $\widehat{H}$, which is the set of all outgoing subspaces of $\widehat{H}$. By \cite[Proposition 3.7 and Remark 3.8]{Daniel17}, one can identify $\mathrm{Gr}(H)$ with  $\Lambda_\sigma^\infty G/K$ (see \cite[Page 619]{Daniel17} for its definition), where $G=\mathrm{GL}(r,\C)$ and $K=\mathrm{U}(r,\C)$.  Therefore, we define a map $f:X\rightarrow\Lambda_\sigma^\infty G/K$ by setting $f(x)=W_x$.  Let $\mathcal{D}=\Lambda_\sigma G/K$ denote the period domain defined in \cite[Section 3.1.4]{Daniel17}. Then $\mathcal{D}$ is an infinite-dimensional complex Hilbert manifold, and $f(x)$ indeed lies in  $\mathcal{D}$ (cf. \cite[Lemma 3.19]{Daniel17}). Thus $f$ can be regarded as a map $f:X\rightarrow\mathcal{D}$ which we define as the period map of $(E,D,H)$.
    \end{construction}
  
    \begin{proof}[Proof of Lemma \ref{lem-periodmap-pullback}]
        We first observe that the above canonical variation of loop Hodge structures $(\mathcal{H},\mathcal{B},\mathcal{T},\mathcal{W},\widetilde{D})$ encodes the information of the holomorphic structure $\bp$. Here $\mathcal{T}$ is in fact given by $(\mathcal{T}s)=\lambda\cdot s(\lambda)$ and so $\mathcal{T}^p\mathcal{W}$ is the set of all Fourier series with degrees at least $p$ (cf. \cite[2.1.3]{Daniel17}). Consequently, the underlying smooth vector bundle $E$ is isomorphic to the orthogonal complement of $\mathcal{T}\mathcal{W}$ in $\mathcal{W}$. Note that $\widetilde{D}^{0,1}\mathcal{W}\subset \mathcal{W}\otimes\mathcal{A}^{0,1}$ and $\widetilde{D}^{0,1}\mathcal{T}\mathcal{W}\subset \mathcal{T}\mathcal{W}\otimes \mathcal{A}^{0,1}$. Thus $\widetilde{D}^{0,1}$ induces a holomorphic structure on $E$ via $E\cong \mathcal{W}/\mathcal{T}\mathcal{W}$. We now verify that it coincides with $\bp$. Let $[s]\in \mathcal{W}/\mathcal{T}\mathcal{W}$ with $s=\sum\limits_{n\in\N}s_n\lambda^n$ as a representative. Then the isomorphism $E\cong \mathcal{W}/\mathcal{T}\mathcal{W}$ is given by
        $s_0\mapsto [s]$. The definition of $\widetilde{D}$ gives $\widetilde{D}^{0,1}=\bp+\lambda\theta^*$, and so $\widetilde{D}^{0,1}[s]=0$ if and only if
        $(\bp+\lambda\theta^*)(s_0+\mathcal{T}\mathcal{W})\in\mathcal{T}\mathcal{W}\otimes\mathcal{A}^{0,1}$, which is equivalent to $\bp s_0=0$. Consequently,
        \begin{equation}\label{equa-identification-holomorphicstructure}
        	(E,\bp)\cong(\mathcal{W}/\mathcal{T}\mathcal{W},\widetilde{D}^{0,1})\cong (W_X/T_XW_X,d_X^{0,1}),
        \end{equation}
        where the second isomorphism is given by \eqref{equa-identification-loophodgestructure}.
          	
    	The identification $\mathrm{Gr}(\widehat{H})=\Lambda_\sigma^\infty G/K$ implies that there exists a tautological bundle over $\mathcal{D}$, defined by
    	$(\widehat{H}^{0,\mathcal{D}})_{W\in\mathcal{D}}$, where $\widehat{H}^{0,\mathcal{D}}$ is the orthogonal complement of $TW$ in $W$ (cf. \cite[Definition 3.12]{Daniel17}). Let $d_\mathcal{D}$ be the trivial connection on the trivial bundle $\widehat{H}_\mathcal{D}$, then it induces a holomorphic structure $\bp_{\mathcal{D}}$ on $\widehat{H}^{0,\mathcal{D}}$. According to the construction of the period map $f$ explained as above, we have $$(W_X/T_XW_X,d_X^{0,1})=f^*(\widehat{H}^{0,D},d_\mathcal{D}^{0,1})$$ since $f$ is holomorphic by \cite[Theorem 3.18]{Daniel17}. Via the identification \eqref{equa-identification-holomorphicstructure}, we conclude that $\mE_X=(E,\bp)=f^*(\widehat{H}^{0,D},d_\mathcal{D}^{0,1})$. Taking $\mE_\mathcal{D}=(\widehat{H}^{0,D},d_\mathcal{D}^{0,1})$, the proof is complete.
    \end{proof}
    
    In view of Lemma \ref{lem-periodmap-pullback} and \cite[Theorem 1.2]{Daniel17}, we can adapt the arguments in \cite[Corollary 5.8]{GKPT19a}---which treats Higgs bundles from variations of Hodge structures---to prove Proposition \ref{prop-descent-higgs}. 
    
    \begin{proof}[Proof of Proposition \ref{prop-descent-higgs}]
    	Recall Takayama's result \cite{Taka03} that for any $x\in X$, there exists an open neighborhood $U$ of $x$ such that $\pi^{-1}(U)$ is simply connected. We shall show that there exists a locally free  sheaf $\mE_U$ on $U$ such that $\mE_{\widehat{X}}|_{\pi^{-1}(U)}=\pi^*\mE_U$, which suffices to conclude the desired statement by Lemma \ref{lem-Higgs-extension}. 
    	
    	Let $f:\pi^{-1}(U)\rightarrow\mathcal{D}$ be the induced period map of the harmonic bundle $(E,D,h)|_{\pi^{-1}(U)}$. By Lemma \ref{lem-periodmap-pullback}, there exists a locally free sheaf $\mE_{\mathcal{D}}$ such that $\mE_X|_{\pi^{-1}(U)}=f^*\mE_{\mathcal{D}}$. If $f$ factors through $\pi$, i.e., $f=g\circ\pi$ for some holomorphic map $g:U\rightarrow\mathcal{D}$, then $\mE_X=\pi^*(g^*\mE_{\mathcal{D}})$. It suffices to show that $f$ is constant on $\pi^{-1}(x)$ for $x\in U$.  Fix a point $x\in U$. For any rational curve $\gamma:\CP^1\rightarrow \pi^{-1}(x)$, $f\circ\gamma$ is again holomorphic and horizontal (cf. \cite[ Section 3.1.6]{Daniel17}) and so is a period map induced by a harmonic bundle structure over $\CP^1$ (cf. \cite[Theorem 1.2]{Daniel17}). However, any harmonic bundle over $\CP^1$ is trivial, and hence its period map $x\mapsto W_x$ is constant, which implies that $f\circ\gamma$ is a constant map. Recall that $\pi^{-1}(x)$ is rationally chain connected (cf. \cite{Fujino23}), i.e., any two points of $\pi^{-1}(x)$ can be connected by a chain of rational curves. We conclude that $f$ is constant on $\pi^{-1}(x)$ for any $x\in U$ and the proof is complete.
    \end{proof}

    \subsection{Proof of Theorem \ref{main-theorem-polystable-equivalence}}\label{subsection-locallyfree}
     We now complete the proof of Theorem \ref{main-theorem-polystable-equivalence}.  Let $X$ be a compact K\"ahler klt space and $(\mE_{X_\reg},\theta_{X_\reg})\in\mathrm{pHiggs}_{X_\reg}$. Building on Theorem \ref{main-stable}, $(\mE_{X_\reg},\theta_{X_\reg})$ is locally free and admits a pluri-harmonic metric $H_{X\setminus Z}$ outside an analytic subspace $Z$ with $\codim_XZ\geq2$ and $X\setminus Z$ smooth. It follows from Proposition \ref{prop-harmonicmetric} that the harmonic bundle associated to $(\mE_{X_\reg},\theta_{X_\reg})|_{X\setminus Z}$ extends to a harmonic bundle on $X_\reg$. Consequently, $(\mE_{X_\reg},\theta_{X_\reg})|_{X\setminus Z}$ extends to a locally free Higgs sheaf on $X_\reg$, which coincides with $(\mE_{X_\reg},\theta_{X_\reg})$ by reflexivity.

     Thus we have verified the ``only if " part of Theorem \ref{main-theorem-polystable-equivalence} (2). Moreover, Theorem \ref{main-theorem-polystable-equivalence} (1) follows directly from Proposition \ref{prop-harmonicmetric} by taking $X\setminus Z=X_\reg$.  The remaining statements of Theorem \ref{main-theorem-polystable-equivalence} are covered by the following proposition.
    	\begin{proposition}\label{prop-locallyfree-chernclass-vanish}
    		Let $X$ be a compact K\"ahler klt space. Suppose that  $(\mE_{X_\reg},\theta_{X_\reg})$ is a locally free Higgs sheaf equipped with a pluri-harmonic metric $H_{X_\reg}$ on $X_\reg$. Then for any maximally quasi-\'etale cover $\gamma:Y\rightarrow X$ induced by some subgroup $G\subset\mathrm{Aut}(Y)$, the following hold:
    		\begin{enumerate}[label=(\thetheorem.\alph*)]
    			\item \label{prop-locallyfree-chernclass-vanish-1} $\gamma^{[*]}\mE_X$ is locally free; in particular, $\gamma^{[*]}\theta_{{X_\reg}}$ extends to a Higgs field $\theta_Y$ of $\mE_Y=\gamma^{[*]}\mE_X$.
    			\item \label{prop-locallyfree-chernclass-vanish-2} For any resolution of singularities $\pi:\widehat{Y}\rightarrow Y$, $\pi^*(\mE_Y,\theta_Y)\in\mathrm{pHiggs}_{\widehat{Y}}$.
    			\item \label{prop-locallyfree-chernclass-vanish-3} $(\mE_Y,\theta_Y)$ is polystable with respect to any K\"ahler forms $\w_0,\cdots,\w_{n-2}$.
    		\end{enumerate}
    		In particular, $(\mE_{X_\reg},\theta_{X_\reg})$ is polystable with respect to any K\"ahler forms $\w_0,\cdots,\w_{n-2}$ and the orbifold Chern classes of $\mE_X$ vanish.
    	\end{proposition}

    	\begin{proof}
    		Let $E_{X_\reg}$ be the underlying $C^\infty$-bundle and $D_{X_\reg}=D_{H_{X_\reg}, \theta_{X_\reg}}$. Observe that there exists an analytic subspace $Z$ of $Y$ such that $\codim_YZ\geq 2$, the open set $Y^\circ:=Y\setminus Z$ is contained in $\gamma^{-1}(X_\reg)$ and $\pi$ is biholomorphic over $Y^\circ$. Set $\widehat{Y}^\circ=\pi^{-1}(Y^\circ)$. As in the proof of Proposition \ref{prop-harmonicmetric}, the pullback $(\gamma\circ\pi|_{Y^\circ})^*(E_{X_\reg}, D_{X_\reg},H_{X_\reg})$ extends to a harmonic bundle $(E_{\widehat{Y}},D_{\widehat{Y}},H_{\widehat{Y}})$, which corresponds to a locally free Higgs sheaf $(\mE_{\widehat{Y}},\theta_{\widehat{Y}})$. Applying Proposition \ref{prop-descent-higgs}, there exists a locally free Higgs sheaf $(\mE_Y,\theta_Y)$ such that $(\mE_{\widehat{Y}},\theta_{\widehat{Y}})=\pi^*(\mE_Y,\theta_Y)$. By virtue of the uniqueness of pluri-harmonic metrics on $(\gamma|_{Y^\circ})^{-1}(E_{X_\reg}, D_{X_\reg})$ (cf. Proposition \ref{prop-harmonicmetric}), $\gamma^{[*]}(\mE_{X_\reg},\theta_{X_\reg})$ and $(\mE_Y,\theta_Y)$ agree on $Y^\circ$ and hence on all of $Y$ by reflexivity. The proof of item \ref{prop-locallyfree-chernclass-vanish-1} is complete.

    		We now prove item \ref{prop-locallyfree-chernclass-vanish-3}. The following argument, based on direct computations via the Chern-Weil formula, will also imply item \ref{prop-locallyfree-chernclass-vanish-2}. Let $\mF_Y$ be an arbitrary $\theta_{Y}$-invariant saturated subsheaf of $\mE_Y$. Note that $(\mE_{\widehat{Y}},\theta_{\widehat{Y}})$ admits a pluri-harmonic metric $H_{\widehat{Y}}$ according to the above construction. Let $\mF_{\widehat{Y}}$ be the saturation of the image of $\pi^*\mF_Y\rightarrow\mE_{\widehat{Y}}$. We may assume that $\pi$ is biholomorphic over $Y_\reg$, so that $\mF_{\widehat{Y}}=\pi^*\mF_{Y}$ on $Y_\reg$. Using the Chern-Weil formula (cf. \cite[Lemma 3.2]{Simpson88}) and Lemma \ref{lem-excision}, we have
    		\begin{equation}\label{equa-check-polystable-1}
    			\begin{split}
    				&\deg_{(\w_0,\cdots,\w_{n-2})}(\mF_Y)=\deg_{(\pi^*\w_0,\cdots,\pi^*\w_{n-2})}(\mF_{\widehat{Y}})\\
    				&=-\int_{\widehat{Y}\setminus\Sigma_{\mF_{\widehat{Y}}}}\im\tr(\p\pi_{\mF_{\widehat{Y}}}^{H_{\widehat{Y}}}\wedge\bp \pi_{\mF_{\widehat{Y}}}^{H_{\widehat{Y}}})\wedge\pi^*\w_0\wedge\cdots\pi^*\w_{n-2}\\
    				&\leq0=\deg_{(\w_0,\cdots,\w_{n-2})}(\mE_Y),
    			\end{split}
    		\end{equation}
    		where we used that $F_{H_{\widehat{Y}},\theta_{\widehat{Y}}}=0$ and $c_1(\mE_Y)=0$ by \ref{lem-resolutionofsingularities-easy-1}, and $\Sigma_{\mF_{\widehat{Y}}}$ denotes the singular set of $\mF_{\widehat{Y}}$. Now suppose that $(\mE_Y,\theta_Y)$ is not $(\w_0,\cdots,\w_{n-2})$-stable. Then there exists some $\mF_Y$ that achieve equality in the last line of \eqref{equa-check-polystable-1}. Since the forms $\pi^*\w_0,\cdots,\pi^*\w_{n-2}$ are strictly positive on $\pi^{-1}(Y_\reg)$, this forces $\p\pi_{\mF_{\widehat{Y}}}^{H_{\widehat{Y}}}=0$ on $\pi^{-1}(Y_\reg)\setminus\Sigma_{\mF_{\widehat{Y}}}$. Thus
    		$$(\mE_{\widehat{Y}},\theta_{\widehat{Y}})=(\mF_{\widehat{Y}},\theta_{\widehat{Y}}|_{\mF_{\widehat{Y}}})\oplus(\mF_{\widehat{Y}}^\perp,\theta_{\widehat{Y}}|_{\mF_{\widehat{Y}}^\perp})$$
    		splits as a direct sum of locally free Higgs sheaves on $\pi^{-1}(Y_\reg)\setminus\Sigma_{\mF_{\widehat{Y}}}$, where $\mF_{\widehat{Y}}^\perp$ is the orthogonal complement with respect to $H_{\widehat{Y}}$. Consequently,
    		$$(\mE_{{Y}},\theta_{{Y}})=(\mF_{{Y}},\theta_Y|_{{\mF_Y}})\oplus(\mF_{{Y}}^\perp,\theta_Y|_{{\mF}_Y^\perp})$$
    		splits as locally free Higgs sheaves on $Y_\reg\setminus\Sigma_{\mF_Y}$, where $\mF_Y^\perp$ is the orthogonal complement with respect to $\pi_*H_{\widehat{Y}}$ and $\Sigma_{\mF_Y}$ is the singular set of $\mF_Y$. Then the induced Hermitian metric $H_{\mF_Y}$ (resp. $H_{\mF_Y^\perp}$) on $\mF_Y$ (resp. $\mF_Y^\perp$) is pluri-harmonic. Since $\Sigma_{\mF_Y}$ has codimension at least $2$, both $(\mF_Y,\theta_Y|_{\mF_Y})$ and $(\mF_Y^\perp,\theta_Y|_{\mF_Y^\perp})$ extend to locally free Higgs sheaves arising from a harmonic bundle structure, by the same argument as above. Applying an induction argument, we verify the $(\w_0,\cdots,\w_{n-2})$-polystability of $(\mE_Y,\theta_Y)$.
    		
    		Finally, we conclude the proof. Item \ref{prop-locallyfree-chernclass-vanish-2}, together with item \ref{lem-resolutionofsingularities-easy-1}, Remark \ref{rem-locallyfree-orbifoldchernclass} and Proposition \ref{lem-quasietale-chernclasses} implies that all orbifold Chern classes of $\mE_X$ vanishes. Moreover, item \ref{prop-locallyfree-chernclass-vanish-3} implies $\gamma^{[*]}(\mE_{X_\reg},\theta_{X_\reg})$ is polystable with respect to $\gamma^*\w_0,\cdots,\gamma^*\w_{n-2}$ for any K\"ahler forms $\w_0,\cdots,\w_{n-2}$ on $X$ and in particular, is $G$-polystable. By Lemma \ref{lem-quasietale-stability}, this yields the $(\w_0,\cdots,\w_{n-2})$-polystability of $(\mE_{X_\reg},\theta_{X_\reg})$. 
    	\end{proof}

    \section{Non-abelian Hodge correspondence on K\"ahler klt spaces}\label{section-nonabelian Hodge theory}
    This section first proves the two key ingredients outlined in Section~\ref{subsection-strategy}: Theorem~\ref{main-maximallyquasietalecase} (in Section~\ref{subsection-JHfiltrations}) and Theorem~\ref{main-thm-descent} (in Section~\ref{subsection-descent}). We then prove Theorems~\ref{main-theorem-locallyfree} and~\ref{main-theorem-reflexive} in Section~\ref{section-formulation}, thereby establishing the nonabelian Hodge correspondences for compact K\"ahler klt spaces. Finally, we prove Corollaries \ref{coro-torus} and \ref{coro-projflat} and Theorem \ref{main-thm-unitball} in Section~\ref{section-applications}.
    
    \subsection{Proof of Theorem \ref{main-maximallyquasietalecase}}\label{subsection-JHfiltrations}
    The proof of Theorem \ref{main-maximallyquasietalecase} will be divided into two parts. We review the following statements, which are useful in the proof of Theorems \ref{main-maximallyquasietalecase} and  \ref{main-thm-descent}.
    
    \begin{lemma}\label{lem-locallyfree-extension}
    	Let $0\rightarrow\mS\rightarrow\mE\rightarrow\mQ\rightarrow0$ be a short exact sequence of coherent sheaves on a Cohen-Macaulay, normal and irreducible complex space $X$ of dimension $n$. Suppose that $\mE$ is reflexive, $\mS$ and $\mQ^{\vee\vee}$ are locally free. If, in addition, $\mQ$ is locally free in codimension $2$, then both $\mE$ and $\mQ$ are locally free.
    \end{lemma}
    
    Lemma \ref{lem-locallyfree-extension} follows from a result of Scheja and Trautmann \cite{Scheja64,Tra67}.
    
    \begin{proposition}[cf. {\cite[Page 142]{Siu06}}]\label{lemma-siu}
    	Let $X$ be a complex space, $A\subset X$ a closed analytic subset and $\mF$ a coherent sheaf on $X$. Fix a nonnegative integer $k$. Then the following statements are equivalent:
    	\begin{itemize}
    		\item[(1)] $\dim (A\cap S_{q+k+1}(\mF))\leq q$ for all $q\in \mathbb{Z}$;
    		\item[(2)] for every open subset $U\subset X$, the restriction maps 
    		$H^i(U,\mF)\rightarrow H^i(U\setminus A,\mF)$
    		are bijective for $i<k$ and injective for $i=k$.
    	\end{itemize}
    \end{proposition}

    For the reader's convenience, we sketch how Proposition \ref{lemma-siu} implies Lemma \ref{lem-locallyfree-extension}. By the assumptions of Lemma \ref{lem-locallyfree-extension}, there exists an analytic subspace $Z\subset X$ of codimension at least $3$ such that $\mQ$ and $\mE$ are locally free outside $Z$. Since $\mQ^\vee=\mQ^{\vee\vee\vee}$ and $\mS$ are locally free, for any $x\in X$, we have
    \begin{align*}
    	\mathrm{codh}((\mS\otimes\mQ^\vee)_x)=\mathrm{codh}(\mO_{X,x})\xlongequal{\text{Cohen-Macaulay}}\dim \mO_{X,x}=\dim (X,x)\xlongequal{\text{normal}}n,
    \end{align*}
    where $\mathrm{codh}$ denotes the depth, and the third equality follows from the dimension formula. Thus the singularity set $S_{n-1}(\mS\otimes \mQ^\vee):=\{x\in X:\mathrm{codh}(\mS\otimes\mQ^\vee)_x\leq n-1\}$ of $\mS\otimes \mQ^\vee$ is empty, and $S_n(\mS\otimes \mQ^\vee)=X$. Hence
    $\dim (A\cap S_{q+3})\leq q\ \text{ for any integer }q.$
    Therefore, the restriction
    $H^1(X,\mS\otimes \mQ^\vee)\rightarrow H^1(X\setminus Z, \mS\otimes\mQ^\vee)$
    is bijective by Proposition \ref{lemma-siu}. This means that $\mE|_{X\setminus Z}$ extends to a locally free sheaf $\mE'$ on $X$ that is the extension of $\mS$ by $\mQ^{\vee\vee}$. By reflexivity of $\mE$, we have $\mE$ coincides with $\mE'$. Hence $\mE$ and $\mQ$ are locally free. This completes the proof.
    
    \begin{lemma}[Hodge Index Theorem, cf. {\cite[Lemma 3.17]{IJZ25}}]\label{lem-hodge-1}
    	Suppose that $\mQ$ is a reflexive sheaf on $X$. Then for nef and big classes $\alpha_0,\cdots,\alpha_{n-2}$, we have
    	\begin{equation}
    		\widehat{c}_1^2(\mQ)\cdot[\alpha_1]\cdots[\alpha_{n-2}]\leq \frac{(\widehat{c}_1(\mQ)\cdot[\alpha_0]\cdots[\alpha_{n-2}])^2}{[\alpha_0]^2\cdot[\alpha_1]\cdots[\alpha_{n-2}]}.
    	\end{equation}
    	The equality holds if and only if $\widehat{c}_1(\mQ)\cdot[\sigma]\cdot[\alpha_1]\cdots[\alpha_{n-2}]=0$ for every $\sigma\in H^2(X,\R)$.
    \end{lemma}

    \subsubsection{Characterization of semistable reflexive Higgs sheaf by extension}
    We first prove part (1) of Theorem \ref{main-maximallyquasietalecase}.
    
    \begin{theorem}\label{thm-chernclassesvanishing}
    	Let $X$ be a compact K\"ahler klt space of dimension $n$, $\gamma:Y\rightarrow X$ be a maximally quasi-\'etale cover and $(\mE_{X_\reg},\theta_{\mE_{X_{\reg}}})$ be a reflexive Higgs sheaf over $X_\reg$. Then the following statements are equivalent:
    	\begin{enumerate}
    		\item $(\mE_{X_\reg},\theta_{\mE_{X_{\reg}}})$ is semistable and satisfies \eqref{equa-vanishingcondition} with respect to some K\"ahler forms $\w_0,\cdots,\w_{n-2}$.
    		\item The pullback $\mE_Y:=\gamma^{[*]}\mE_X$ is locally free, and $\gamma^*\theta_{\mE_{X_{\reg}}}$ extends to a Higgs field $\theta_Y$ on $\mE_Y$. Moreover, $\mE_Y$ admits a filtration by $\theta_{Y}$-invariant locally free subsheaves
    		$$0=\mE_0\subsetneq\mE_1\subsetneq\cdots\subsetneq\mE_l=\mE_Y$$
    		such that each quotient $\mQ_k:=\mE_k/\mE_{k-1}$ is locally free with all Chern classes vanishing, and the induced locally free Higgs sheaf $(\mQ_k,\theta_k)$ is $(\gamma^*\w_0,\cdots, \gamma^*\w_{n-2})$-stable with respect to any K\"ahler forms $\w_0,\cdots,\w_{n-2}$ on $X$. 
    		\item The orbifold Chern classes $\widehat{c}_1(\mE_X),\widehat{c}_1^2(\mE_X),\widehat{c}_2(\mE_X)$ vanish, and $(\mE_{X_\reg},\theta_{\mE_{X_{\reg}}})$ is semistable with respect to any K\"ahler forms $\w_0,\cdots,\w_{n-2}$ on $X$.
    	\end{enumerate}
    	The same holds for polystability, in which case the filtration in (\romannumeral2) has length $l=1$.
    \end{theorem}
    
    \begin{proof}[Proof of Theorem \ref{thm-chernclassesvanishing}]
    The implication $(\romannumeral3)\Rightarrow (\romannumeral1)$ is obvious, and $(\romannumeral2)\Rightarrow (\romannumeral3)$ follows from Lemma \ref{semistability-extension}, Lemma \ref{lem-quasietale-stability} and Proposition \ref{lem-quasietale-chernclasses}. Thus
    it suffices to verify $(\romannumeral1)\Rightarrow (\romannumeral2)$. By the same reasoning as in the implication $(\romannumeral2)\Rightarrow (\romannumeral3)$, we may assume that $X$ itself is maximally quasi-\'etale, so that $\gamma$ is the identity. 
     
     If $(\mE_{X_\reg},\theta_{\mE_{X_\reg}})$ is $(\w_0,\cdots,\w_{n-2})$-stable, the conclusion follows directly from Theorem \ref{main-theorem-polystable-equivalence}, Proposition \ref{prop-locallyfree-chernclass-vanish} and Lemma \ref{lem-resolutionofsingularities-easy}. Note that any reflexive Higgs sheaf of $\rank 1$ is $(\omega_0, \dots, \omega_{n-2})$-stable. We will argue by induction on $\rank(\mE_X)$. Assume that Theorem~\ref{thm-chernclassesvanishing} holds for all ranks smaller than $r$, and let $\rank(\mE_X)=r$. Without loss of generality, we assume that $(\mE_{X{\mathrm{reg}}}, \theta_{\mE_{X{\mathrm{reg}}}})$ is not $(\omega_0, \dots, \omega_{n-2})$-stable. 
     
     The following construction of $\mF_X$ is standard; we include the details for the reader's convenience. It follows from Remark \ref{remark-saturation-stability} that the set
    $$\mathcal{B}:=\{0\neq\mathcal{H}_X\subsetneq\mE_X: \mathcal{H}_{X_\reg}\text{ is $\theta_{\mE_{X_\reg}}$-invariant and }\mu_{(\w_0,\cdots,\w_{n-2})}(\mathcal{H}_X)=\mu_{(\w_0,\cdots,\w_{n-2})}(\mE_X)\}$$
    is nonempty. Choose $\mF_X\in\mathcal{B}$ of maximal rank. We may assume that $\mF_X$ is saturated by Remark \ref{remark-saturation-stability}. In particular, the induced Higgs sheaf $(\mF_{X_\reg},\theta_{\mF_{X_\reg}})$ is $(\w_0,\cdots,\w_{n-2})$-semistable. Set $\mQ_{X_\reg}:=\mE_{X_\reg}/\mF_{X_\reg}$. Then a direct computation gives
    $$\mu_{(\w_0,\cdots,\w_{n-2})}(\mQ_X)=\frac{(c_1(\mE_X)-c_1(\mF_X))\cdot[\w_0]\cdots[\w_{n-2}]}{\rank(\mE_X)-\rank(\mF_X)}=\mu_{(\w_0,\cdots,\w_{n-2})}(\mE_X)=0.$$
    We claim that the induced torsion-free Higgs sheaf $(\mQ_{X_\reg}, \theta_{\mQ_{X_\reg}})$ is $(\w_0,\cdots,\w_{n-2})$-stable. Indeed, if it were not, there would exist a $\theta_{\mQ_{X_\reg}}$-invariant saturated subsheaf $\mG_X\subsetneq \mQ_X$ with $\mu_{(\w_0,\cdots,\w_{n-2})}(\mG_X)\geq \mu_{(\w_0,\cdots,\w_{n-2})}(\mQ_X).$ Consider a proper subsheaf $\mS_X\subsetneq \mE_X$ such that $\mS_X/\mF_X=\mG_X$. Then we have a short exact sequence
    $$0\rightarrow\mF_X\rightarrow\mS_X\rightarrow\mG_X\rightarrow0.$$
     The sheaf $\mS_{X_\reg}$ is $\theta_{\mE_{X_\reg}}$-invariant, and a direct computation shows $\mu_{(\w_0,\cdots,\w_{n-2})}(\mS_X)\geq \mu_{(\w_0,\cdots,\w_{n-2})}(\mE_X)$, contradicting the maximality of $\rank(\mF_X)$ in $\mathcal{B}$.

    Now, applying Proposition \ref{prop-orbifold-exact}, the Bogomolov-Gieseker inequality \eqref{equa-BG}, and Lemma \ref{lem-hodge-1}, we obtain that
    	\begin{equation}\label{equa-chernclass-sum}
    		\begin{split}
    			0=&\widehat{\ch}_2(\mE_X)\cdot[\w_1]\cdots[\w_{n-2}]\\
    			\leq& \widehat{\ch}_2(\mF_X)\cdot[\w_1]\cdots[\w_{n-2}]+\widehat{\ch}_2(\mQ_X^{\vee\vee})\cdot[\w_1]\cdots[\w_{n-2}]\\
    			\leq& \frac{1}{\rank(\mF_X)}\widehat{c}_1^2(\mF_X)\cdot[\w_1]\cdots[\w_{n-2}]+\frac{1}{\rank(\mQ_X)}\widehat{c}_1^2(\mQ^{\vee\vee})\cdot[\w_1]\cdots[\w_{n-2}]\\
    			\leq&\frac{1}{\rank \mF_X}\frac{\big(\widehat{c}_1(\mF_X)\cdot[\w_0]\cdots[\w_{n-2}]\big)^2}{[\w_0]^2\cdot[\w_1]\cdots[\w_{n-2}]}+\frac{1}{\rank \mQ_X}\frac{\big(\widehat{c}_1(\mQ_X^{\vee\vee})\cdot[\w_0]\cdots[\w_{n-2}]\big)^2}{[\w_0]^2\cdot[\w_1]\cdots[\w_{n-2}]}
    			=0
    		\end{split}
    	\end{equation}
    	since $\mu_{(\w_0,\cdots,\w_{n-2})}(\mF_X)=\mu_{(\w_0,\cdots,\w_{n-2})}(\mQ_X)=0$. Hence, every inequality in \eqref{equa-chernclass-sum} is  an equality. This yields the following three conclusions:
    	\begin{itemize}
    		\item[(\romannumeral1)] $(\mF_{X_\reg},\theta_{\mF_{X_\reg}})$ is an $(\w_0,\cdots,\w_{n-2})$-semistable reflexive Higgs sheaf satisfying \eqref{equa-vanishingcondition}. By the induction hypothesis, $(\mF_{X_\reg},\theta_{\mF_{X_\reg}})$ satisfies (\romannumeral2) of Theorem \ref{thm-chernclassesvanishing}. In particular, $\mF_X$ is locally free.
    		\item[(\romannumeral2)] $(\mQ_{X_\reg},\theta_{\mQ_{X_\reg}})^{\vee\vee}$ is an $(\w_0,\cdots,\w_{n-2})$-stable reflexive Higgs sheaf satisfying \eqref{equa-vanishingcondition}; Consequently, by the stable case already established, $\mQ_X^{\vee\vee}$ is locally free and all its Chern classes vanish.
    		\item[(\romannumeral3)] $\mQ_X^{\vee\vee}=\mQ_X$ in codimension $2$ (see Proposition \ref{prop-orbifold-exact}); Using Lemma \ref{lem-locallyfree-extension}, we conclude both $\mQ_X$ and $\mE_X$ are locally free. In particular, the Higgs fields $\theta_{\mE_{X_\reg}}$ and $\theta_{\mQ_{X_\reg}}$ extend to Higgs fields on $\mE_X$ and $\mQ_X$, respectively.
    	\end{itemize}
    	Putting together the above statements and the induction hypothesis completes the proof. 
    \end{proof}

    \subsubsection{Independence of the choice of  K\"ahler polarization}
    
    We now complete the proof of Theorem \ref{main-maximallyquasietalecase} by establishing the following result.
    \begin{theorem}\label{thm-indenpendent-polarization}
    	Let $X$ be a compact K\"ahler klt space of dimension $n$, $\pi:\widehat{X}\rightarrow X$ a resolution of singularities, and $(\mE_{X},\theta_{X})$ a locally free Higgs sheaf over $X$. Set $(\mE_{\widehat{X}},\theta_{\widehat{X}})=\pi^*(\mE_X,\theta_X)$. Then the following statements are equivalent:
    	\begin{enumerate}
    		\item There exist K\"ahler forms $\w_0,\cdots,\w_{n-2}$ such that $(\mE_{X},\theta_{X})$ is $(\w_0,\cdots,\w_{n-2})$-semistable and satisfies \eqref{equa-vanishingcondition}.
    		\item There exist nef and big forms $\alpha_0,\cdots,\alpha_{n-2}$ such that $(\mE_{\widehat{X}},\theta_{\widehat{X}})$ is $(\alpha_0,\cdots,\alpha_{n-2})$-semistable with $${\ch}_1(\mE_{\widehat{X}})\cdot[\alpha_0]\cdots[\alpha_{n-2}]=0\quad \text{and}\quad {\ch}_2(\mE_{\widehat{X}})=0.$$
    		\item All Chern classes of $\mE_{\widehat{X}}$ vanish, and $(\mE_{\widehat{X}},\theta_{\widehat{X}})$ is semistable with respect to every K\"ahler polarization $(\widehat{\w}_0,\cdots,\widehat{\w}_{n-2})$ on $\widehat{X}$.
    		\item All Chern classes of $\mE_X$ vanish, and $(\mE_{X},\theta_{X})$ is semistable with respect to every K\"ahler polarization $(\w_0,\cdots,\w_{n-2})$.
    	\end{enumerate}
    	The same equivalences hold for polystability.
    \end{theorem}
    
    On the other hand, using Theorem \ref{thm-indenpendent-polarization} together with Theorem \ref{thm-chernclassesvanishing}, Proposition \ref{lem-quasietale-chernclasses} and Lemma \ref{lem-quasietale-stability}, one can easily verify that the definitions of the categories $\mathrm{Higgs}_X, \mathrm{pHiggs}_X, \mathrm{Higgs}_{X_\reg}$ and $\mathrm{pHiggs}_{X_\reg}$ are independent of the choice of K\"ahler polarization $(\w_0,\cdots,\w_{n-2})$.
    
    \begin{proof}
    	The implication $(\romannumeral4)\Rightarrow(\romannumeral1)$ is obvious, and $(\romannumeral3)\Rightarrow (\romannumeral4)$ follows from item \ref{lem-resolutionofsingularities-easy-1}. We first prove $(\romannumeral1)\Rightarrow(\romannumeral2)$. Let $\gamma:Y\rightarrow X$ be a maximally quasi-\'etale cover of $X$, $\widehat{Y}$ be a desingularization of the (unique) irreducible component of the fibre product $\widehat{X}\times_XY$. Then we have a commutative diagram:
    	\begin{equation}\label{equa-diagram-reso-maximally}
    		\begin{tikzcd}
    			\widehat{Y} \arrow[r,"\widehat{\pi}"] \arrow[d,"\widehat{\gamma}"] & Y \arrow[d,"\gamma"] \\
    			\widehat{X} \arrow[r,"\pi"] & X.
    		\end{tikzcd}
    	\end{equation}
    	Set $(\mE_Y,\theta_Y)=\gamma^{*}(\mE_{X},\theta_{X})$. Building on Theorem \ref{thm-chernclassesvanishing}, all Chern classes of $\mE_Y$ vanish in $H^{2i}(Y,\mathbb{Q})$ and so all Chern classes of $\widehat{\gamma}^*\mE_{\widehat{X}}=\widehat{\pi}^*\mE_Y$ vanish in $H^{2i}(\widehat{Y},\mathbb{Q})$ by functoriality. Because $\widehat{\gamma}$ is generically finite, all Chern classes of $\mE_{\widehat{X}}$ vanish. Moreover, \cite[Proposition 4.8]{ZZZ25} implies $(\mE_{\widehat{X}},\theta_{\widehat{X}})$ is $(\pi^*\w_0,\cdots,\pi^*\w_{n-2})$-semistable. This establishes $(\romannumeral1)\Rightarrow(\romannumeral2)$.
    	
    	We now prove $(\romannumeral2)\Rightarrow(\romannumeral3)$. The argument relies on an iterated application of the Hodge Index Theorem. Take a Jordan-H\"older filtration of $(\mE_{\widehat{X}},\theta_{\widehat{X}})$ with respect to $(\alpha_0,\cdots,\alpha_{n-2})$ (cf. \cite[Definition 1.5.1]{HL10}), i.e., a filtration by $\theta_{\widehat{X}}$-invariant saturated subsheaves
    	$0=\mE_{0}\subsetneq\mE_{1}\subsetneq\cdots\subsetneq\mE_{l}=\mE_{\widehat{X}}$
    	such that each quotient $\mQ_k=\mE_k/\mE_{k-1}$ with the induced Higgs field $\theta_k$, is $(\alpha_0,\cdots,\alpha_{n-2})$-stable and satisfies $c_1(\mQ_k)\cdot[\alpha_0]\cdots[\alpha_{n-2}]=0$. Let $\alpha_{i,\epsilon}=\alpha_i+\epsilon\w_{\widehat{X}}$, where $\w_{\widehat{X}}$ is a fixed K\"ahler class on $\widehat{X}$. We claim that for all $1\leq k\leq l$ and $0\leq m\leq n-2$, it holds that
    	\begin{equation}\label{equa-slope-1}
    		c_1(\mQ_k)\cdot[\alpha_{0,\epsilon}]\cdots[\alpha_{m-1,\epsilon}]\cdot[\alpha_{m}]\cdots[\alpha_{n-2}]=0.
    	\end{equation}
    	Note that \eqref{equa-slope-1} holds for $m=0$ and all $1\leq k\leq l$ by construction. Assume inductively that \eqref{equa-slope-1} holds for some $m-1$ (where $1\leq m\leq n-2$). We shall prove it for $m$. Because $(\mQ_k,\theta_k)^{\vee\vee}$ is $(\alpha_{0,\epsilon},\cdots,\alpha_{m-1,\epsilon},\alpha_m+\delta\w_{\widehat{X}},\cdots,\alpha_{n-2}+\delta\w_{\widehat{X}})$-stable for $0<\delta\ll 1$ and $0<\epsilon\ll1$ (see e.g. \cite[Proposition 4.18]{ZZZ25}), the Bogomolov-Gieseker inequality \eqref{equa-BG} yields, after letting $\delta\rightarrow 0$,
    	\begin{equation}\label{equa-chern-computation2}
    		\begin{split}
    			&\left(\ch_2(\mQ_k^{\vee\vee})-\frac{1}{\rank \mQ_k^{\vee\vee}}\cdot c_1^2(\mQ_k^{\vee\vee})\right)\cdot[\alpha_{0,\epsilon}]\cdots[\alpha_{m-1,\epsilon}]\cdot[\alpha_{m+1}]\cdots[\alpha_{n-2}]\leq0.
    		\end{split}
    	\end{equation}
    Combining Proposition \ref{prop-orbifold-exact}, the induction hypothesis for $m-1$, \eqref{equa-chern-computation2} and the Hodge Index Theorem (Lemma \ref{lem-hodge-1}), we obtain
    	\begin{equation}\label{equa-chern-computation3}
    		\begin{split}
    			0&=\ch_2(\mE_{\widehat{X}})\cdot [\alpha_{0,\epsilon}]\cdots[\alpha_{m-1,\epsilon}]\cdot[\alpha_{m+1}]\cdots[\alpha_{n-2}]\\
    			&\leq\sum\limits_{k=1}^l2\cdot \ch_2(\mQ_k^{\vee\vee})\cdot[\alpha_{0,\epsilon}]\cdots[\alpha_{m-1,\epsilon}]\cdot[\alpha_{m+1}]\cdots[\alpha_{n-2}]\\
    			&\leq\sum\limits_{k=1}^l\frac{1}{\rank \mQ_k^{\vee\vee}}\cdot c_1^2(\mQ_k^{\vee\vee})\cdot[\alpha_{0,\epsilon}]\cdots[\alpha_{m-1,\epsilon}]\cdot[\alpha_{m+1}]\cdots[\alpha_{n-2}]\\
    			&\leq \sum\limits_{k=1}^l\frac{\big(c_1(\mQ_k^{\vee\vee})\cdot[\alpha_{0,\epsilon}]\cdots[\alpha_{m-1,\epsilon}]\cdot[\alpha_{m}]\cdots[\alpha_{n-2}]\big)^2}{\rank \mQ_k^{\vee\vee}\cdot [\alpha_m]^2\cdot[\alpha_{0,\epsilon}]\cdots[\alpha_{m-1,\epsilon}]\cdot[\alpha_{m+1}]\cdots[\alpha_{n-2}]}=0.
    		\end{split}
    	\end{equation}
    	Thus, inequality in the fourth line of \eqref{equa-chern-computation3} is an equality, which forces that for all $1\leq k\leq l$,
    	\begin{equation}
    		c_1^2(\mQ_k^{\vee\vee})\cdot[\alpha_{0,\epsilon}]\cdots[\alpha_{m-1,\epsilon}]\cdot[\alpha_{m+1}]\cdots[\alpha_{n-2}]=\frac{\big(c_1(\mQ_k^{\vee\vee})\cdot[\alpha_{0,\epsilon}]\cdots[\alpha_{m-1,\epsilon}]\cdot[\alpha_{m}]\cdots[\alpha_{n-2}]\big)^2}{ [\alpha_m]^2\cdot[\alpha_{0,\epsilon}]\cdots[\alpha_{m-1,\epsilon}]\cdot[\alpha_{m+1}]\cdots[\alpha_{n-2}]}.
    	\end{equation}
    	Combining this with Lemma \ref{lem-hodge-1}, we deduce that for every $1\leq k\leq l$,
    	\begin{equation}
    		c_1(\mQ_k^{\vee\vee})\cdot[\alpha_{m,\epsilon}]\cdot[\alpha_{0,\epsilon}]\cdots[\alpha_{m-1,\epsilon}]\cdot[\alpha_{m+1}]\cdots[\alpha_{n-2}]=0.
    	\end{equation} Proceeding by induction on $m=1,\cdots,n-2$, we establish the claim \eqref{equa-slope-1}. On the other hand, the openness of stability (see e.g. \cite[Proposition 4.18]{ZZZ25}) implies that $(\mQ_k,\theta_k)^{\vee\vee}$ is $(\alpha_{0,\epsilon},\cdots,\alpha_{n-2,\epsilon})$-stable for $0<\epsilon\ll1$. Combining these statements, we deduce $(\mE_{\widehat{X}},\theta_{\widehat{X}})$ is $(\alpha_{0,\epsilon},\cdots,\alpha_{n-2,\epsilon})$-semistable for $0<\epsilon\ll1$ (see e.g. \cite[Proposition 4.18]{ZZZ25}). Applying the proof of Theorem \ref{thm-chernclassesvanishing}, we conclude that each $\mE_k$ and $\mQ_k$ is locally free. Furthermore, by Proposition \ref{prop-locallyfree-chernclass-vanish}, all Chern classes of $\mQ_k$ vanish, and $(\mQ_k,\theta_k)$ is stable with respect to every K\"ahler polarization $(\widehat{\w}_0,\cdots,\widehat{\w}_{n-2})$ on $\widehat{X}$. Consequently, $(\mE_{\widehat{X}},\theta_{\widehat{X}})$ is $(\widehat{\w}_0,\cdots,\widehat{\w}_{n-2})$-semistable and all its Chern classes vanish, which verifies the implication $(\romannumeral2)\Rightarrow(\romannumeral3)$.
    	
        The proof for the polystable case is identical.
    \end{proof}
    
    Combining Theorem~\ref{thm-indenpendent-polarization} with the classical non-abelian Hodge correspondence for a single K\"ahler class \cite{Cor88,Simpson88,NZ,Deng21} (see also \cite[Theorem 1.9]{PZZ23} for projective flat bundles), we immediately obtain its extension to multiple K\"ahler polarizations on compact K\"ahler manifolds, which has already been established in the polystable case by \cite{CW2}.
    
    \begin{theorem}[Non-abelian Hodge theory for compact K\"ahler manifolds]\label{smooth-nonabelian}
    	Let $X$ be a compact K\"ahler manifold. There exists an equivalence of categories
    	$\mu_X:\mathrm{Higgs}_X\rightarrow \mathrm{LSys}_X.$
    	This equivalence restricts to an equivalence between
    	$\mathrm{pHiggs}_X$ and $\mathrm{sLSys}_X$.
    \end{theorem}
    
    \begin{remark}
    	The equivalence in Theorem \ref{smooth-nonabelian} is indeed functorial with respect to morphisms between compact K\"ahler manifolds, which can be seen directly from the construction of $\mu_X$ (see e.g. \cite[Section 9]{PZZ23}).
    \end{remark}

    \subsection{Proof of Theorem \ref{main-thm-descent}}\label{subsection-descent}
    
    We now prove Theorem \ref{main-thm-descent}. Suppose that $f:Z\rightarrow X$ is a bimeromorphic holomorphic map between compact K\"ahler klt spaces and $(\mE_Z,\theta_Z)\in\mathrm{Higgs}_Z$. We need to construct a locally free Higgs sheaf $(\mE_X,\theta_X)\in\mathrm{Higgs}_X$ such that $f^*(\mE_X,\theta_X)=(\mE_Z,\theta_Z)$.

     Let $g:W\rightarrow X$ be an orbifold modification of $X$ and $Y$ be a desingularization of the (unique) irreducible component of the fibre product $W\times_XZ$. Then we have the following commutative diagram
    \begin{equation}
    	\begin{tikzcd}
    		Y \arrow[r,"pr_1"] \arrow[d,"pr_2"] \arrow[rd,"\pi"] & W \arrow[d,"g"] \\
    		Z \arrow[r,"f"] & X,
    	\end{tikzcd}
    \end{equation}
    where $pr_1,pr_2$ are morphisms induced by natural projections and $\pi:=g\circ pr_1=f\circ pr_2$. In particular, $pr_1:Y\rightarrow W$, $pr_2:Y\rightarrow Z$ and $\pi:Y\rightarrow X$ are resolutions of singularities of $W$, $Z$ and $X$, respectively. By Theorem \ref{thm-indenpendent-polarization}, the pullback $(\mE_Y,\theta_Y):=pr_2^*(\mE_Z,\theta_Z)\in\mathrm{Higgs}_Y$. Combining this with item \ref{lem-resolutionofsingularities-easy-1}, it reduces to show the following statement.
    
    \begin{proposition}\label{prop-descent-reduce}
    There exists a locally free Higgs sheaf $(\mE_X,\theta_X)$ on $X$ such that $(\mE_Y,\theta_Y)=\pi^*(\mE_X,\theta_X)$.
    \end{proposition}
    
    We begin with some preparations. Assume that $E_Y = \mu_Y(\mathcal{E}_Y, \theta_Y) \in \mathrm{LSys}_Y$. Then by item \ref{lem-resolutionofsingularities-easy-2}, $E_X = (\pi^*)^{-1}E_Y \in \mathrm{LSys}_X$ is well-defined. Denote by $E_{\rho_X}$ and $E_{\rho_Y}$ the corresponding local systems. 
    
    Suppose that $E_X$ is induced by a linear representation $\rho_X:\pi_1(X)\rightarrow\mathrm{GL}(r,\C)$ in block upper-triangular form
    \[
    \begin{pmatrix}
    	\rho_1 & * & \cdots & * \\
    	0 & \rho_2 & \cdots & * \\
    	\vdots & \vdots & \ddots & \vdots \\
    	0 & 0 & \cdots & \rho_m
    \end{pmatrix},
    \]
    where each diagonal block $\rho_i: \pi_1(X) \to \mathrm{GL}(r_i, \mathbb{C})$ is irreducible. This induces an increasing filtration of local systems
    \[
    0 = E_{X,0} \subsetneq E_{X,1} \subsetneq \cdots \subsetneq E_{X,m} = E_{X},
    \]
    where each $E_{X,i}$ corresponds to the truncated representation
    \[
    \widehat{\rho}_i = 
    \begin{pmatrix}
    	\rho_1 & * & \cdots & * \\
    	0 & \rho_2 & \cdots & * \\
    	\vdots & \vdots & \ddots & \vdots \\
    	0 & 0 & \cdots & \rho_i
    \end{pmatrix}.
    \]
    In particular, the quotient $Q_{X,i}:=E_{X,i}/E_{X,i-1}$ is a simple local system induced by $\rho_i$. Then $E_{W}=g^*E_{X}$ admits a filtration of local systems
    \begin{equation}\label{equa-filtration-2}
    	0=E_{W,0}\subsetneq E_{W,1}\subsetneq \cdots\subsetneq E_{W,m}=E_{W}
    \end{equation}
    such that each $E_{W,i}=g^*E_{X,i}$ is induced by $\widehat{\rho_i}\circ g_*$. Because $g_*:\pi_1(W)\rightarrow \pi_1(X)$ is an isomorphism by item \ref{lem-resolutionofsingularities-easy-2}, the quotient $$Q_{W,i}=E_{W,i}/E_{W,i-1}=g^*Q_{X,i}$$ is a simple local system induced by $\rho_i\circ g_*$. Analogously, the flat bundle $(V_Y,D_Y)$ induced by $E_{Y}$ admits a filtration by flat subbundles
    \begin{equation}\label{equa-filtration-3}
    	0=V_{Y,0}\subsetneq V_{Y,1}\subsetneq \cdots\subsetneq V_{Y,m}=V_Y
    \end{equation}
    such that each $V_{Y,i}$ is $D_Y$-flat and induced by $\widehat{\rho_i}\circ\pi_*$, and the quotient $Q_{Y,i}=V_{Y,i}/V_{Y,i-1}$ with the induced flat connection $D_{Y,i}$ is induced by $\pi^*\rho_i\in\mathrm{sLSys}_Y$.
    
    By applying the non-abelian Hodge correspondence for compact K\"ahler manifolds \cite{Deng21}, $(Q_{Y,i},D_{Y,i})$ admits a pluri-harmonic metric $h_i$. Recall that there is a unique harmonic decomposition \begin{equation}
    	D_{Y,i}=\p_{Y,i}+\bp_{Y,i}+\theta_{Y,i}+\theta_{Y,i}^*,
    \end{equation} and $Q_{Y, i}$ thereby inherits a natural Higgs bundle structure $(\mQ_{Y,i},\theta_{Y,i})$.
    By Proposition \ref{prop-descent-higgs} and Theorem \ref{main-maximallyquasietalecase}, there exists a locally free Higgs sheaf $(\mQ_{X,i},\theta_{X,i})\in\mathrm{pHiggs}_X$ on $X$ such that $(\mQ_{Y,i},\theta_{Y,i})=\pi^*(\mQ_{X,i},\theta_{X,i}).$ According to the above construction, the underlying $C^\infty$-bundle of $(\mQ_{X,i},\theta_{X,i})$ is isomorphic to the underlying $C^\infty$-bundle of $Q_{X,i}$. In particular, $$(\mQ_{W,i},\theta_{W,i}):=g^*(\mQ_{X,i},\theta_{X,i})\in\mathrm{pHiggs}_W,$$ and its underlying $C^\infty$-bundle is isomorphic to the underlying $C^\infty$-bundle of $Q_{W,i}$ (the notion of a $C^\infty$-bundle makes sense via local embeddings of analytic spaces).
    
    \vspace{0.1cm}
    
    Now we complete the proof of Proposition \ref{prop-descent-reduce}. The basic idea is to construct a locally free Higgs sheaf structure $(\mE_W,\theta_W)$ on the underlying $C^\infty$-bundle of $E_{\rho_W}$ via the filtration $\{E_{W,k}\}_{k=1}^m$ and the standard orbifold structure $W_\orb=\{(U_\alpha,G_\alpha,\mu_\alpha)\}$ of $W$, and then show that $(g_*\mE_W)^{\vee\vee}$ is locally free.
    
    \begin{proof}[Proof of Proposition \ref{prop-descent-reduce}]
    	Let  $(V_\orb,D_\orb)$ be the flat orbi-bundle induced by $\{\mu_\alpha^*E_{\rho_W}\}$. Because $\mu_\alpha$ is quasi-\'etale, $(V_\orb,D_\orb)$ admits a filtration by flat orbi-bundles
    	\begin{equation}\label{equa-filtration-4}
    		0=V_{\orb,0}\subsetneq V_{\orb,1}\subsetneq\cdots\subsetneq V_{\orb,m}=V_{\orb}
    	\end{equation}
    	such that each $V_{\orb,i}$ is invariant under $D_\orb$ and the quotient $Q_{\orb,i}=V_{\orb,i}/V_{\orb,i-1}$ with the induced flat connection $D_{\orb,i}$ is induced by $\{\mu_\alpha^*Q_{W,i}\}$. Note that $Q_{\orb,i}$ carries a Higgs orbi-bundle structure $\{(Q_{\orb,i},\bp_{\orb,i},\theta_{\orb,i})\}$ induced by $\{\mu_\alpha^*(\mQ_{W,i},\theta_{W,i})\}$. Then $(Q_{\orb,i},\bp_{\orb,i},\theta_{\orb,i})$ is polystable with respect to some orbifold K\"ahler class $\w_\orb$ by Theorem \ref{thm-indenpendent-polarization} and \cite[Proposition 4.8]{ZZZ25}, and all its orbifold Chern classes vanish because $Q_{\orb,i}$ is flat. By the existence of Hermitian-Einstein metrics on stable Higgs orbi-bundles \cite{ZZZ25}, $(Q_{\orb,i},\bp_{\orb,i},\theta_{\orb,i})$ admits an orbifold pluri-harmonic metric $H_{\orb,i}$.  Since each $\mu_\alpha$ is \'etale over $W_\reg$, $H_{\orb,i}$ descends to a pluri-harmonic metric $H_i$ on $(\mQ_{W,i},\theta_{W,i})|_{W_\reg}$. By virtue of the uniqueness statement in Theorem \ref{main-theorem-polystable-equivalence} (1), the connection $D_{\theta_{W,i},H_i}$ coincides with the flat connection on $Q_{W,i}|_{W_\reg}$ and so $D_{\theta_{\orb,i},H_{\orb,i}}=D_{\orb,i}.$ 
    	
    	Note that the Hodge isomorphism theorem remains valid for compact complex orbifolds (cf. \cite[Proposition 2.1]{Ma05}). Hence we may follow the argument of \cite[Proof of Theorem A]{Deng21} to construct a Higgs orbi-bundle structure $(V_\orb,\bp_\orb,\theta_\orb)$ on $V_\orb$ such that each $V_{\orb,i}$ is invariant under $\bp_\orb$ and $\theta_\orb$, and the induced Higgs structure on each quotient $Q_{\orb,i}$ coincides with $(Q_{\orb,i},\bp_{\orb,i},\theta_{\orb,i})$. Consequently, $$(\mE_\orb,\theta_\orb):=(V_\orb,\bp_\orb,\theta_\orb)$$ gives an extension of the collection of locally Higgs sheaves $\{\mu_\alpha^*(\mQ_{W,i},\theta_{W,i})\}_{i=1}^m$.	By applying \cite[Lemma A.4]{GKKP11} and noticing that $\mu_\alpha$ is \'etale over $W_\reg$, we see that $\left((\mu_\alpha)_*(\mu_\alpha^*\mQ_{W,i})\right)^{G_\alpha}$ is reflexive and coincides with the locally free sheaf $\mQ_{W,i}$. Then $\mE_W:=\left((\mu_\alpha)_*(\mE_{\orb})_\alpha\right)^{G_\alpha}$ is reflexive and admits a filtration by saturated subsheaves
    	$$0=\mE_{W,0}\subsetneq \mE_{W,1}\subsetneq\cdots \subsetneq \mE_{W,m}=\mE_W$$
    	such that $\mE_{W,i}/\mE_{W,i-1}=\mQ_{W,i}$, where $\mE_{W,i}=\left((\mu_\alpha)_*(\mE_{\orb,i})_\alpha)\right)^{G_\alpha}$ and $\mE_{\orb,i}=(V_{\orb,i},\bp_\orb|_{V_{\orb,i}})$. Lemma \ref{lem-locallyfree-extension} implies that each $\mE_{W,i}$ is locally free, and thus so is $\mE_{W}$. Moreover, the orbifold Higgs field $\theta_\orb$ first descends to a Higgs field on $\mE_W|_{W_\reg}$, and then extends to a Higgs field $\theta_W$ on $\mE_W$ by Lemma \ref{lem-Higgs-extension}. This endows the underlying $C^\infty$-bundle of $E_{\rho_W}$ with a locally free Higgs sheaf structure. Consequently, $(V_Y,\bp_Y,\theta_Y)=pr_1^*(\mE_W,\theta_W)$ is a locally free Higgs sheaf on $(V_Y,D_Y)$. 
    	
    	From the above construction, each $V_{Y,i}$ is $\theta_Y$-invariant and $\bp_Y$-invariant, and the Higgs structure induced on the quotient $Q_{Y,i}$ coincides with $(Q_{Y,i},\bp_{Y,i},\theta_{Y,i})$ (because the two agree on a Zariski open subset of $Y$). According to the explicit construction of the non-abelian Hodge correspondence for compact K\"ahler manifolds (cf. \cite[Section 9]{PZZ23} and \cite{Deng21}), we have $\mu_Y(V_Y,\bp_Y,\theta_Y)={E_Y}=\mu_Y(\mE_Y,\theta_Y),$
    	which implies $(\mE_Y,\theta_Y)=pr_1^*(\mE_W,\theta_W)$.
    	
    	Finally, we verify that the reflexive sheaf $\mE_X:=(g_*\mE_W)^{\vee\vee}$ is locally free, which will complete the proof. The Higgs field $\theta_W$ induces a Higgs field on $\mE_X$ in codimension $2$, and hence induces a Higgs field $\theta_{X_\reg}$ on $\mE_{X_\reg}$. From the above construction, $\mE_X$ admits a filtration by reflexive subsheaves
    	$$0=\mE_{X,0}\subsetneq \mE_{X,1}\subsetneq \cdots\subsetneq \mE_{X,m}=\mE_X,$$
    	where $\mE_{X,i}:=(g_*\mE_{W,i})^{\vee\vee}$. It is straightforward to see that each $\mE_{X,i}$ is saturated in $\mE_X$. This follows from the reflexivity of these sheaves together with the fact that $\mE_{X,i}/\mE_{X,i-1}=\mQ_{X,i}$ outside the indeterminacy locus $S$ of $g^{-1}$. Again by reflexivity, we have $$(\mE_{X,i}/\mE_{X,i-1})^{\vee\vee}=\mQ_{X,i},$$ and its orbifold Chern classes vanish. Recall that $S$ has codimension at least $3$; consequently, the orbifold Chern classes of each $\mE_{X,i}$ vanish, because $g_\orb^{[*]}\mE_{X,i}$ coincides with $E_{\orb,i}$ outside $g^{-1}(S)$ and all orbifold Chern classes of $E_{\orb,i}$ vanish. Combining these facts with Proposition \ref{prop-orbifold-exact}, we deduce that $\mE_{X,i}/\mE_{X,i-1}$ is locally free in codimension $3$. Applying Lemma \ref{lem-locallyfree-extension} and proceeding by induction on $i=1,\cdots,m$ shows that each $\mE_{X,i}$ is locally free, and hence so is $\mE_X$.
    \end{proof}
    
    \subsection{Construction of the correspondence}\label{section-formulation}
    This subsection completes the proofs of Theorem \ref{main-theorem-locallyfree} and Theorem \ref{main-theorem-reflexive}. Assuming that $X$ is a compact K\"ahler klt space, we outline the construction of $\mu_X$ and $\mu_{X_\reg}$.
    
    \subsubsection{Proof of Theorem \ref{main-theorem-locallyfree}}
    We first construct the functor $\mu_X$. For any $(\mE_X,\theta_X)\in\mathrm{Higgs}_X$, Theorem \ref{thm-indenpendent-polarization} implies $\pi^*(\mE_X,\theta_X)\in\mathrm{Higgs}_{\widehat{X}}$ for a resolution of singularities $\pi:\widehat{X}\rightarrow X$. Together with item \ref{lem-resolutionofsingularities-easy-2}, this allows us to define $\mu_X$ as follows.
    
    \begin{definition}
    	The functor $\mu_X: \mathrm{Higgs}_X\rightarrow \mathrm{LSys}_X$ is defined by
    	$$(\mE_{X},\theta_X)\mapsto (\pi^*)^{-1}\left(\mu_{\widehat{X}}(\pi^*(\mE_X,\theta_X))\right),$$
    	where $\pi:\widehat{X}\rightarrow X$ is a resolution of singularities.
    \end{definition}

     It is clear that the diagram \eqref{equa-diagram-1} commutes. It is necessary to check that $\mu_X$ is independent of the choice of $\pi$. Let $\pi_1:X_1\rightarrow X$ and $\pi_2:X_2\rightarrow X$ be two resolutions of singularities. We need to verify the following lemma.
    \begin{lemma}[Well-definedness of $\mu_X$]\label{lem-welldefined-muX}
    $$(\pi_1^*)^{-1}\left(\mu_{X_1}(\pi_1^*(\mE_X,\theta_X))\right)=(\pi_2^*)^{-1}\left(\mu_{X_2}(\pi_2^*(\mE_X,\theta_X))\right)$$
    \end{lemma}
    \begin{proof}
    	Let $\widehat{X}$ be a desingularization of the (unique) irreducible component of the fibre product $X_1\times_XX_2$. We then have a commutative diagram
    	\begin{equation}\label{equa-welldefined}
    		\begin{tikzcd}
    			\widehat{X} \arrow[r,"f_1"] \arrow[d,"f_2"] \arrow[rd,"g"] & X_1 \arrow[d,"\pi_1"] \\
    			X_2 \arrow[r,"\pi_2"] & X,
    		\end{tikzcd}
    	\end{equation}
    	where $f_1$ and $f_2$ are the morphisms induced by the natural projections, and $g:=\pi_1\circ f_1=\pi_2\circ f_2$. Note that
    	\begin{align*}
    		g^*\left((\pi_1^*)^{-1}\left(\mu_{X_1}(\pi_1^*(\mE_X,\theta_X))\right)\right)&=f_1^*\left(\mu_{X_1}(\pi_1^*(\mE_X,\theta_X))\right)=\mu_{\widehat{X}}\left(f_1^*\pi_1^*(\mE_X,\theta_X)\right)\\
    		&=\mu_{\widehat{X}}\left(f_2^*\pi_2^*(\mE_X,\theta_X)\right)=f_2^*\left(\mu_{X_2}(\pi_2^*(\mE_X,\theta_X))\right)\\
    		&=g^*\left((\pi_2^*)^{-1}\left(\mu_{X_2}(\pi_2^*(\mE_X,\theta_X))\right)\right),
    	\end{align*}
    	where the second and the fourth equalities follow from the functoriality of the non-abelian Hodge correspondence for compact K\"ahler manifolds. Because $g$ is also a resolution of singularities, the induced morphism $g^*:\mathrm{LSys}_X\rightarrow \mathrm{LSys}_{\widehat{X}}$ is one-to-one by item \ref{lem-resolutionofsingularities-easy-2}. This implies \eqref{equa-welldefined} and completes the proof.
    \end{proof}
    
    According to the above construction, it is obvious that if $(\mE_X,\theta_X)\in\mathrm{pHiggs}_X$ then $\mu_X(\mE_X,\theta_X)\in\mathrm{sLSys}_X$. Moreover, using the existence of functorial resolutions of singularities (see e.g. \cite[Theorem 3.10]{DO23}) and the functoriality of $\mu_X$ for compact K\"ahler manifolds, we obtain the following corollary, which will be used to construct $\mu_{X_\reg}$.
    
    \begin{corollary}\label{coro-functoriality}
    	$\mu_X$ is functorial with respect to local analytic isomorphisms.
    \end{corollary}
    
    To prove that $\mu_X$ is one-to-one, it suffices to construct its inverse. Let $E_X\in\mathrm{LSys}_X$. Suppose that $\pi:\widehat{X}\rightarrow X$ is a resolution of singularities and then $\pi^*E_X\in\mathrm{LSys}_{\widehat{X}}$. It follows from Theorem \ref{main-thm-descent} that there exists $(\mE_X,\theta_X)\in\mathrm{Higgs}_X$ such that $\mu_{\widehat{X}}^{-1}(\pi^*E_X)=\pi^*(\mE_X,\theta_X)$.
    
    \begin{definition}
    	The functor $\eta_X:\mathrm{LSys}_X\rightarrow\mathrm{Higgs}_X$ is defined by sending $E_X$ to $(\mE_X,\theta_X)$ constructed above.
    \end{definition}
    
    It follows immediately from the above construction that $\eta_X\circ \mu_X=\Id$, $\mu_X\circ\eta_X=\Id$ and $\mu_X$ satisfies the properties listed in Theorem \ref{main-theorem-locallyfree}.
    
    \begin{remark}
    	One might naturally ask whether adapting the argument in \cite[Section 8.2]{GKPT19} would allow us to show that $\eta_X$ (respectively $\mu_\alpha$) is in fact functorial with respect to arbitrary morphisms between compact K\"ahler klt spaces. However, a key ingredient used in that proof---the analytic analogue of \cite[Corollary 1.7]{HM2}---remains unproven at present, as the proof of the latter relies on the choice of a very ample divisor, a step that is not available in the analytic setting.
    \end{remark}

    \subsubsection{Proof of Theorem \ref{main-theorem-reflexive}}
    We now construct the functor $\mu_{X_\reg}$. Let $\gamma:Y\rightarrow X$ be a maximally quasi-\'etale cover of $X$ induced by some subgroup $G\subset \mathrm{Aut}(Y)$. For any $(\mE_{X_\reg},\theta_{X_\reg})\in\mathrm{Higgs}_{X_\reg}$, the reflexive pullback $(\mE_{Y_\reg},\theta_{Y_\reg})=\gamma^{[*]}(\mE_{X_\reg},\theta_{X_\reg})$ is locally free and uniquely extends to a locally free Higgs sheaf $(\mE_Y,\theta_Y)\in\mathrm{Higgs}_Y$ by Theorem \ref{main-maximallyquasietalecase}. This also implies $\mE_{X_\reg}$ itself is locally free because $\gamma$ is \'etale over $X_\reg$.
    \begin{definition}
    	$\gamma^*:\mathrm{Higgs}_{X_\reg}\rightarrow \mathrm{Higgs}_Y$ is defined by sending $(\mE_{X_\reg},\theta_{X_\reg})$ to $(\mE_Y,\theta_Y)$ constructed above. 
    \end{definition}
    
    By reflexivity, $(\mE_Y,\theta_Y)$ is $G$-equivariant and thus the local system $\mu_Y(\mE_Y,\theta_Y)\in\mathrm{LSys}_Y$ is also $G$-equivariant by Corollary \ref{coro-functoriality}. Since $\gamma$ is finite and \'etale over $X_\reg$, the local system $\mu_Y(\mE_Y,\theta_Y)|_{\gamma^{-1}(X_\reg)}$ descends to a local system $E_{X_\reg}\in\mathrm{LSys}_{X_\reg}$.
    
    \begin{definition}
    	Set $\gamma^\circ:=\gamma|_{\gamma^{-1}(X_\reg)}$.    $\mu_{X_\reg}:\mathrm{Higgs}_{X_\reg}\rightarrow\mathrm{LSys}_{X_\reg}$ is defined by
    	$$(\mE_{X_\reg},\theta_{X_\reg})\mapsto \left(\gamma^\circ_*(\mu_Y(\gamma^*(\mE_{X_\reg},\theta_{X_\reg})))\right)^{G}.$$
    \end{definition}
    
    To illustrate that $\mu_{X_\reg}$ is well-defined, we must verify
    \begin{lemma}\label{lem-welldefined-regular}
    	The functor $\mu_{X_\reg}$ is independent of the choice of the maximally quasi-\'etale cover.
    \end{lemma}
    \begin{proof}[Proof of Lemma \ref{lem-welldefined-regular}]
    	Let $\gamma_1:Y_1\rightarrow X$ and $\gamma_2:Y_2\rightarrow X$ be two maximally quasi-\'etale covers of $X$ induced by some subgroups $G_1\subset \mathrm{Aut}(Y_1)$ and $G_2\subset\mathrm{Aut}(Y_2)$, respectively. Consider the normalization $Y$ of the fibre product $Y_1\times_X Y_2$. The two natural projections $Y\rightarrow Y_1$ and $Y\rightarrow Y_2$ are quasi-\'etale. According to Definition \ref{defn-maximally} (1), $Y$ is an \'etale cover of $Y_1$ and hence also a maximally quasi-\'etale cover of $X$. 
    	
    	Therefore, to prove the lemma, we may assume that there exists an \'etale morphism $\gamma_{12}: Y_1\rightarrow Y_2$. Let $\pi_2:\widehat{Y}_2\rightarrow Y_2$ be a resolution of singularities and $\widehat{Y}_1$ be a desingularization of the (unique) irreducible component of the fibre product $\widehat{Y}_2\times_XY_1$. Denote by $\widehat{\gamma}_{12}:\widehat{Y}_1\rightarrow \widehat{Y}_2$ and $\pi_1:\widehat{Y}_1\rightarrow Y_1$ the morphisms induced by the natural projections from the fibre product. We adopt the notation from the construction above. One can see that $(\mE_{\widehat{Y}_1},\theta_{\widehat{Y}_1})=\widehat{\gamma}^*_{12}(\mE_{Y_1},\theta_{Y_2})$ by the following commutative diagram
    	\[
    	\begin{tikzcd}
    		\widehat{Y}_1 \arrow[r,"\pi_1"] \arrow[d,"\widehat{\gamma}_{12}"] & Y_1 \arrow[d,"\gamma_{12}"] \\
    		\widehat{Y}_2 \arrow[r,"\pi_2"] & Y_2.
    	\end{tikzcd}
    	\]
    	Applying the functoriality of the non-abelian Hodge correspondence for compact K\"ahler manifolds yields $E_{\widehat{Y}_{1}}=\widehat{\gamma}_{12}^*E_{{\widehat{Y}}_{2}}$. This completes the proof.
    \end{proof}
    
     To show that $\mu_{X_\reg}$ is one-to-one, it is enough to construct its inverse. Let $E_{X_\reg}\in\mathrm{LSys}_{X_\reg}$. Item \ref{lem-maximallyquasietale-easy-2} implies that $(\gamma|_{\gamma^{-1}(X_\reg)})^*E_{X_\reg}$ uniquely extends to a $G$-equivariant local system $E_Y\in\mathrm{LSys}_Y$.
     \begin{definition}
     	We define $\gamma^*:\mathrm{LSys}_{X_\reg}\rightarrow\mathrm{LSys}_Y$ by sending $E_{X_\reg}$ to $E_Y$.
     \end{definition}
     Building on Corollary \ref{coro-functoriality} and the fact that $\mu_Y$ is one-to-one, the Higgs sheaf $\mu_Y^{-1}(E_Y)\in\mathrm{Higgs}_{Y}$ is $G$-equivariant and thus $\mu_Y^{-1}(E_Y)|_{\gamma^{-1}(X_\reg)}$ descends to some $(\mE_{X_\reg},\theta_{X_\reg})\in\mathrm{Higgs}_{X_\reg}$ (again by item \ref{lem-maximallyquasietale-easy-1}).
     
     \begin{definition} $\eta_{X_\reg}:\mathrm{LSys}_{X_\reg}\rightarrow\mathrm{Higgs}_{X_\reg}$ is defined by
     	$$E_{X_\reg}\mapsto \left(\gamma^\circ_*(\mu_Y^{-1}(\gamma^*E_{X_\reg}))\right)^G.$$
     \end{definition}

    It follows immediately from the construction that $\mu_{X_\reg}\circ \eta_{X_\reg}=\Id$, $\eta_{X_\reg}\circ\mu_{X_\reg}=\Id$ and $\mu_{X_\reg}$ satisfies the remaining properties stated in Theorem \ref{main-theorem-reflexive}. The proof is complete.
    
    \subsection{Quasi-uniformization results}\label{section-applications}
    In this subsection, we prove Corollaries \ref{coro-torus} and \ref{coro-projflat} and Theorem \ref{main-thm-unitball}.
    
    \begin{proof}[Proof of Corollary \ref{coro-torus}]
    	Let $X$ be a compact K\"ahler klt space of dimension $n$ with $K_X$ numerically trivial. Since $\mT_X$ is generically nef with respect to $(\w_0,\cdots,\w_{n-2})$ by \cite[Proposition 6.5]{IJZ25} and $c_1(K_X)=0$, $\mT_X$ is semistable with respect to $(\w_0,\cdots,\w_{n-2})$. Then \eqref{equa-miyaoka2} follows from the Bogomolov-Gieseker inequality \eqref{equa-BG}. 
    	
    	Let $\gamma:Y\rightarrow X$ be a maximally quasi-\'etale cover. We have $\mT_Y=\gamma^{[*]}\mT_X$ by reflexivity. Since $\widehat{c}_1(\mT_X)=0$, Theorem \ref{main-maximallyquasietalecase} implies that the equality in \eqref{equa-miyaoka2} holds if and only if $\mT_Y$ is locally free and flat. The Lipman-Zariski conjecture for klt spaces (cf. \cite[Theorem 1.13]{KS21}) then yields that $Y$ is smooth. By Yau's theorem \cite{Yau78}, $Y$ is a complex torus and the proof is complete.
    \end{proof}

    \begin{proof}[Proof of Corollary \ref{coro-projflat}]
    	The proof is similar to \cite[Proposition 2.20]{IMM24}. Let $(\mE_{X_\reg},\theta_{{X_\reg}})$ be a $(\w_0,\cdots,\w_{n-2})$-semistable reflexive Higgs sheaf. Then \cite[Corollary 1.3]{ZZZ25} implies that $\mG_{X_\reg}:=\End(\mE_{X_\reg})$, endowed with the induced Higgs field $\theta_{\mG_{X_\reg}}$, is $(\w_0,\cdots,\w_{n-2})$-semistable. The inequality \eqref{equa-BG-equality} was established in Theorem \ref{main-stable}. Recall from \eqref{equa-end-chernclass-resuction} that $\mE_X$ satisfies the equality in \eqref{equa-vanishingcondition} if and only if $\mG_X$ satisfies $\widehat{c}_1(\mG_X)=0$ together with \eqref{equa-vanishingcondition}. The conclusion now follows from Theorem \ref{main-maximallyquasietalecase}.
    \end{proof}
    
    The proof of Theorem~\ref{main-thm-unitball} follows the spirit of \cite{IJZ25} and \cite[Section 4]{Jinnouchi25-2}, building on the results in Sections~\ref{section-HE} and~\ref{section-harmonicbundles}. For the precise definitions of stability and of the intersection number taken with respect to the non‑pluripolar product, we refer to \cite[Section 3.2 and 4.2]{IJZ25}.
    
    \begin{proof}[Proof of Theorem \ref{main-thm-unitball}]
    	
    	Let $X$ be a projective klt variety of dimension $n$ with $K_X$ big. Consider the natural Higgs sheaf on $X$ consisting of the reflexive sheaf \(\mathcal{E}_X := \Omega_X^{[1]} \oplus \mathcal{O}_X\) together with the Higgs field \(\theta_X\) defined by
    	\begin{equation}
    		\theta_X : \Omega_X^{[1]} \oplus \mathcal{O}_X \longrightarrow (\Omega_X^{[1]} \oplus \mathcal{O}_X) \otimes \Omega_X^{[1]}, \quad a \oplus b \longmapsto (0 \oplus 1) \otimes a.
    	\end{equation} Since $K_X$ is big, $X$ admits a canonical model $X_{\text{can}}$ (see \cite[Theorem 1.2]{BCHM}). Denote by $h:X\dashrightarrow X_{\text{can}}$ the birational map induced by the $K_X$-MMP. Let $p_1:X_1\rightarrow X$ and $q_1:X_1\rightarrow X_{\text{can}}$ be resolutions of $f$. Then there exists a $q_1$-effective divisor $E$ such that
    	$$p_1^*K_X=q_1^*K_{\text{can}}+E$$
    	and $p_1(\supp(E))$ coincides with the $h$-exceptional divisor. By \cite[Lemma A.5]{DHY23} (see also \cite[Proposition 3.15]{IJZ25}), we have
    	\begin{equation}\label{equa-big-product}
    		\langle p_1^*K_X^k\rangle= q_1^*K_{X_{\text{can}}}^k,\ \forall \ k\in\N.
    	\end{equation}
    	
    	We now recall the definition of $\widehat{\Delta}(\mE_X)\cdot\langle K_X^{n-2}\rangle$ introduced in \cite[Section 4.2.1]{IJZ25}. Take an orbifold modification $f':Y'\rightarrow X$ such that $Y'$ admits a standard orbifold structure $Y_\orb':=\{(U'_\alpha,G_\alpha',\mu_\alpha')\}$. Set $\mE_{Y'}:=f'^*\mE_X/\tor$. This induces a torsion-free orbi-sheaf $\mE_\orb':=\{\widehat{\mu}_\alpha^*\mF_{Y'}/\mathrm{torsion}\}$. According to Lemma \ref{lem-orbireso}, there exists a projective bimeromorphic morphism $g_\orb:Y_\orb\rightarrow Y_\orb'$ from an effective complex orbifold $Y_\orb=\{(U_\alpha,G_\alpha,\mu_\alpha)\}$ to $Y_\orb'$ such that the torsion-free pullback of $\mE_\orb'$ is an orbi-bundle $E_\orb'$. Set $f:=f'\circ g$ and it follows from \cite[Proposition 4.18]{IJZ25} that
    	\begin{equation}\label{equa-big-chernclass}
    		\left(2\widehat{c}_2(X)-\frac{n}{n+1}\widehat{c}_1^2(X)\right)\cdot\langle K_X^{n-2}\rangle=\left(2c_2^\orb(E_\orb')-\frac{n}{n+1}({c}_1^\orb)^2(E_\orb')\right)\cdot\langle f^*K_X^{n-2}\rangle.
    	\end{equation}
    	Observe that on every orbifold chart $U_\alpha$, $(f\circ\mu_\alpha)^*\mE_X$ carries a functorial Higgs field $(f\circ\mu_\alpha)^*\theta_X$ by \cite[Theorem 1.10]{KS21} (cf. \cite[Construction 5.5]{GKPT19a}), which induces a Higgs field $\theta_\orb'$ on $E_\orb'$. 

    	Choose resolutions \( p_{2}: X_{2} \to Y \) and \( q_{2}: X_{2} \to X_{1} \) of the composition \( p^{-1}\circ f: Y \dashrightarrow X_{1} \).  
    	Let \( V_{\alpha} \) be a functorial resolution of singularities of the normalization of \( U_{\mu}\times_{Y} X_{2} \); then the group \( G_{\alpha} \) lifts to a subgroup \( \widehat{G}_{\alpha} \subset \operatorname{Aut}(V_{\alpha}) \).  
    	Denote by \( \widehat{\mu}_{\alpha}: V_{\alpha} \to V_{\alpha}/\widehat{G}_{\alpha} \) the quotient map.  
    	The triple \( \{(V_{\alpha},\widehat{G}_{\alpha},\widehat{\mu}_{\alpha})\} \) forms a complex effective orbifold $Z_\orb$; we write \( Z \) for its underlying complex space, $g_1:Z\rightarrow Y$ and $g_2:Z\rightarrow X$ for the induced bimeromorphisms.  
    	All these morphisms are summarized in the following diagram:
    	\[
    	\begin{tikzcd}
    		& X_2 \arrow[d,"p_2"] \arrow[r,"q_2"] & X_1 \arrow[d,"p_1"] \arrow[dr,"q_1"] &   \\
    		Z \arrow[r,"g_1"] \arrow[ur,"g_2"] & Y \arrow[r,"f"] \arrow[ur,dashrightarrow] & X \arrow[r,"h",dashrightarrow] & X_{\text{can}}.
    	\end{tikzcd}
    	\]
    	
    	Let $(E_\orb,\theta_\orb)$ be the pullback of the Higgs orbi-bundle $(E_\orb',\theta_\orb')$ by $g_1$. Note that \eqref{equa-big-product} implies
    	\begin{equation}\label{equa-big-nonpluri-2}
    		\langle g_1^*f^*K_X^{k}\rangle=\langle g_2^*q_2^*p_1^*K_X^{k}\rangle=g_2^*q_2^*q_1^*K_X^k.
    	\end{equation}    
    	Set $p:=f\circ g_1:Z\rightarrow X$ and $q:=q_1\circ q_2\circ g_2:Z\rightarrow X_\text{can}$. Since Poincar\'e duality holds for orbifolds, we have $\langle g_1^*f^* K_X^k\rangle=\pi_*\langle \pi^*g_1^*f^* K_X^k \rangle=g_1^* \langle f^*K_X^k\rangle$ in $H^{2n-4}(Z,\R)$ according to the definition of the non-pluripolar product \cite[Section 3.1.3]{IJZ25}. Then \eqref{equa-big-nonpluri-2} and \eqref{equa-big-chernclass} yield
    	\begin{equation}
    		\left(2\widehat{c}_2(X)-\frac{n}{n+1}\widehat{c}_1^2(X)\right)\cdot\langle K_X^{n-2}\rangle=\left(2c_2^\orb(E_\orb)-\frac{n}{n+1}({c}_1^\orb)^2(E_\orb)\right)\cdot \langle q^*K_X^{n-2}\rangle.
    	\end{equation}
        Recall \cite[Theorem A]{Gue16} that the reflexive cotangent sheaf $\Omega_{X_\text{can}}^{[1]}$ is $K_{X_{\text{can}}}^{n-1}$-semistable, which implies that the natural Higgs sheaf $(\mE_{X_{\text{can}}},\theta_{X_{\text{can}}})$ on $X_{\text{can}}$ is $K_{X_{\text{can}}}^{n-1}$-stable (cf. \cite[Proposition 2.8]{IMM24}). According to the above construction, $q:Z\rightarrow X_{\text{can}}$ satisfies all properties listed in Proposition \ref{prop-HEequation-setting} upon substituting $(\mE_{X_\reg},\theta_{X_\reg})$= $(\mE_{X_\text{can}},\theta_{X_{\text{can}}})|_{(X_\text{can})_\reg}$ and $\w_0=\cdots=\w_{n-2}=\w$, where $\w$ is a K\"ahler form in the class $K_X$. Under the assumption that the equality in \eqref{equa-MY} holds, applying Theorem \ref{thm1} and a computation similar to that in Section \ref{subsection-computation} give
    	\begin{equation}
    		0=4\pi^2\left(2c_2^\orb(E_\orb)-\frac{n}{n+1}({c}_1^\orb)^2(E_\orb)\right)\cdot \langle q^*K_X^{n-2}\rangle \geq \int_{X_{\text{can}}\setminus\Sigma}\|F_{H_\infty,\theta_{\text{can}}}^\perp\|^2_{H_\infty,\w}\cdot\frac{\w^n}{n!}
    	\end{equation}
    	and hence $F_{H_\infty,\theta_{\text{can}}}^\perp\equiv0$ on $X_{\text{can}}\setminus \Sigma$. In particular, $\mG_{X_{\text{can}}}:=\End(\mE_{X_{\text{can}}})$ together with the induced Higgs field $\theta_{\mG}$ admits a pluri-harmonic metric on $X_{\text{can}}\setminus \Sigma$.  This Higgs sheaf therefore admits a pluri-harmonic metric on the regular locus of $X_{\text{can}}$ by Proposition \ref{prop-harmonicmetric}. Then applying Proposition \ref{prop-locallyfree-chernclass-vanish}, we obtain that for a maximally quasi-\'etale cover $\gamma:W\rightarrow X_{\text{can}}$, $\gamma^{[*]}\mG_{X_{\text{can}}}$ is locally free and all its Chern classes vanish. By reflexivity,
    	$\gamma^{[*]}\mG_{X_{\text{can}}}=\End(\Omega_W^{[1]}\oplus \mO_W)$. As $\mT_W$ is a direct summand of $\End(\Omega_W^{[1]}\oplus \mO_W)$, it follows that $\mT_W$ is locally free.  The solution of Lipman-Zariski conjecture for klt spaces \cite{GKKP11,KS21} then implies that $W$ is smooth. Moreover, $W$ satisfies the Miyaoka-Yau equality with respect to $K_W^{n-2}$. Since $K_W=\gamma^*K_{X_{\text{can}}}$ is ample, the proof is complete by the classical result of Yau \cite{Yau77}.
    \end{proof}
    
    \begin{remark}
    	Although \cite[Theorem 1.2]{Jinnouchi25-2} assume that $X$ is smooth in codimension $2$, this condition is in fact unnecessary. The general case follows by combining the argument above with the methods of \cite[Section 4]{Jinnouchi25-2}.
    \end{remark}


\begin{thebibliography}{20}
		
		
		\bibitem{AGT16}
		A. Abbes, M. Gros and T. Tsuji, {\em The p-adic Simpson correspondence, Annals of Mathematics Studies,} Vol. 193,
		Princeton University Press, Princeton, NJ, 2016.
		
		
		
		\bibitem{Bando94}
		S. Bando, {\em Einstein-Hermitian metrics on noncompact K\"ahler manifolds.} Einstein metrics and Yang-Mills
		connections (Sanda, 1990), 27-33, Lecture Notes in Pure and Appl. Math. 145, Dekker. New York, 1993.
		
		\bibitem{BS} S. Bando and Y.T. Siu, {\em Stable sheaves and Einstein-Hermitian metrics}, in {\it Gemetry and Analysis on Complex Manifolds}, World Scientific, 1994, 39-50.
		
		\bibitem{BCHM}
		C. Birkar, P. Cascini, C. Hacon and J. McKernan, {\em Existence of minimal models for varieties of log general type,} Journal of the American Mathematical Society. {\bf 23} (2010), No. 2, 405-468.
		
		
		
		
		
		
		
		
		
		
		
		
		
		
		\bibitem{CHP16}
		F. Campana, A. H\"oring and T. Peternell, {\em Abundance for K\"ahler threefolds,}
		Ann. Sci. Ec. Norm. Sup\'er. {\bf 49} (2016), No. 4, 971-1025.
		
		\bibitem{CHP23}
		F. Campana, A. H\"oring and T. Peternell, {\em Erratum and addendum to the paper: Abundance for K\"ahler threefolds,} arXiv preprint: 2304.10161 (2023).
		
		\bibitem{CGNPPW23}
		J. Cao, P. Graf, P. Naumann, M. P\u{a}un, T. Peternell and X. Wu, {\em Hermite-Einstein metrics in singular settings,} Asian Journal of Mathematics {\bf 29} (2025), No. 4, 441-484.
		
		\bibitem{CF2019}
		J. Caramello, C. Francisco, {\em Introduction to orbifolds,} arXiv preprint:1909.08699 (2019).
		
		\bibitem{Chen25}
		X. Chen, {\em Admissible Hermitian–Yang–Mills connections over normal varieties,} Mathematische Annalen (2025), 1-37.
		
		\bibitem{CW1}
		X. Chen and R. A. Wentworth, {\em A Donaldson–Uhlenbeck–Yau theorem for normal varieties and semistable bundles on degenerating families,} Mathematische Annalen {\bf 288} (2024), No. 2, 1903-1935.
		
		
		
		\bibitem{CW2}
		X. Chen and R. A. Wentworth, {\em The nonabelian Hodge correspondence for balanced Hermitian metrics of Hodge-Riemann type,} Mathematical Research Letters {\bf 31} (2024), No. 3, 639-654.
		
		
		\bibitem{CZ25}
		Z. Chen and C. Zhang, {\em Stable reflexive sheaves over compact Gauduchon manifolds,} Science China Mathematics {\bf 68} (2025), No. 4, 891-916.
		
		\bibitem{CGG22}
		B. Claudon, P. Graf and H. Guenancia, {\em Numerical characterization of complex torus quotients,} Commentarii Mathematici Helvetici {\bf 97} (2022), No. 4, 769-799.
		
		\bibitem{CGG24} B. Claudon, P. Graf and H. Guenancia, {\em Equality in the Miyaoka–Yau inequality and uniformization of non-positively curved klt pairs,} Comptes Rendus. Math\'{e}matique. {\bf 362} (2024), No. S1, 55--81.
		
		\bibitem{CGGN22}
		B. Claudon, P. Graf, H. Guenancia and P. Naumann, {\em K\"ahler spaces with zero first Chern class: Bochner principle, Albanese map and fundamental groups,} Journal f\"{u}r die reine und angewandte Mathematik (Crelles Journal). (2022), No. 786, 245--275.
		
		\bibitem{CJY19}
		T.C. Collins, A. Jacob and S.T. Yau, {\em Poisson metrics on flat vector bundles over non-compact curves,} Comm. Anal. Geom. {\bf 27} (2019), No. 3, 529--597.
		
		\bibitem{Cor88}
		K. Corlette, {\em Flat G-bundles with canonical metrics,} J. Differential Geom. {\bf 28} (1988), No. 3,
		361–382,
		
		\bibitem{Cor92}
		K. Corlette, {\em Archimedean superrigidity and hyperbolic geometry,} Ann. of Math. (2).
		{\bf 135} (1992), No. 1, 165--182.
		
		\bibitem{Daniel17}
		J. Daniel, {Loop Hodge structures and harmonic bundles,} Algebraic Geometry. {\bf 4} (2017), No. 5, 603--643
		
		
		\bibitem{DHP24}
		O. Das, C. Hacon and M. P\u{a}un, {\em On the 4-dimensional minimal model program for K\"ahler varieties,} Advances in Mathematics {\bf 443} (2024), 109615.
		
		\bibitem{DHY23}
		O. Das, C. Hacon and J. I. Y\'a\~{n}ez, {\em MMP for Generalized Pairs on K\"ahler 3-folds,} arXiv: 2305.00524 (2023).
		
		\bibitem{DO23}
		O. Das and W. Ou, {\em On the Log Abundance for Compact K\"ahler threefolds II,} arXiv: 2306.00671v3 (2023).
		
		
		
		
		\bibitem{Deng21}
		Y. Deng, {\em A note on the Simpson correspondence for semistable Higgs bundles,} Pure Appl. Math. Q. {\bf 17} (2021), No. 5, 1899-1911.
		
		\bibitem{DON85}
		S.K. Donaldson, {\em Anti self-dual Yang-Mills connections over complex algebraic
			surfaces and stable vector bundles}, Proc. London Math. Soc. (3) {\bf50}(1985), 1-26.
		
		\bibitem{Don87}
		S. Donaldson, {\em Twisted harmonic maps and the self-duality equations,} Proc. London Math. Soc. (3) {\bf 55} (1987), No. 1, 127–131, 
		
		
		\bibitem{Fal05}
		G. Faltings, {\em A $p$-adic Simpson correspondence,} Adv. Math. {\bf 198} (2005), 847–862.
		
		\bibitem{Faulk22}
		M. Faulk, {\em Hermitian-Einstein metrics on stable vector bundles over compact K\"ahler orbifolds,} arxiv preprint: 2202.08885 (2022).
		
		
		\bibitem{OF25}
		X. Fu and W. Ou, {\em Orbifold Bogomolov-Gieseker inequalities on compact K\"ahler varieties,} arxiv preprint: 2511.03530 (2025).
		
		\bibitem{Fuji22}
		O. Fujino, {\em Minimal model program for projective morphisms between complex analytic spaces,} arXiv preprint: 2201.11315 (2022).
		
		\bibitem{Fujino23}
		O. Fujino, {\em Notes on rational chain connectedness,} arxiv preprint: 2602.19415 (2026).
		
		
		
		
		\bibitem{GK20} P. Graf and T. Kirschner, {\em Finite quotients of three-dimensional complex tori,} Ann.
		Inst. Fourier (Grenoble). {\bf 70} (2020), No. 2, 881–-914.
		
		
		\bibitem{GR55}
		H. Grauert and R. Remmert, {\em Zur theorie der modifikationen. I. stetige und eigentliche modifikationen komplexer r\"aume,} Mathematische Annalen {\bf 129} (1955), No. 1, 274--296.
		
		
		\bibitem{GKKP11}
		D. Greb, S. Kebekus, S. J. Kov\'acs, and T. Peternell, {\em Differential forms
			on log canonical spaces,} Publ. Math. Inst. Hautes Études Sci. {\bf 114} (2011), 87--169.
		
		
		
		
		\bibitem{GKP16}
		D. Greb, S. Kebekus and T. Peternell, {\em \'{E}tale fundamental groups of Kawamata log terminal spaces, flat sheaves, and quotients of abelian varieties,} Duke Mathematical Journal. {\bf 165} (2016), No. 10, 1965--2004.
		
		\bibitem{GKP16b}
		D. Greb, S. Kebekus and T. Peternell, "Singular spaces with trivial canonical class," , in Minimal Models and Extremal Rays, Kyoto, 2011, Adv. Stud. Pure Math., vol. 70, Mathematical Society of
		Japan, Tokyo, 2016, 67–-113.
		
		\bibitem{GKP22}
		D. Greb, S. Kebekus and T. Peternell, {\em Projective flatness over klt spaces and uniformisation of varieties with nef anti-canonical divisor,} Journal of Algebraic Geometry. {\bf 31} (2022), No. 3, 467--496.
		
		\bibitem{GKPT19a}
		D. Greb, S. Kebekus, T. Peternell and B. Taji, {\em The Miyaoka-Yau inequality and uniformisation of canonical models,} Annales Scientifiques de l'\'{E}cole Normale Sup\'{e}rieure. {\bf 52} (2019), No. 6, 1487--1535.
		
		\bibitem{GKPT19}
		D. Greb, S. Kebekus, T. Peternell and B. Taji, {\em Nonabelian Hodge theory for klt spaces and descent theorems for vector bundles,} Compositio Mathematica. {\bf 155} (2019), No. 2, 289-323.
		
		\bibitem{GKPT20}
		D. Greb, S. Kebekus, T. Peternell and B. Taji, {\em Harmonic metrics on Higgs sheaves and uniformization of spaces of general type,} Mathematische Annalen {\bf 378} (2020), 1061-1094.
		
		
		\bibitem{Gue16}
		H. Guenancia, {\em Semistability of the tangent sheaf of singular varieties,} Algebr. Geom. {\bf 3} (2016), No. 5, 508-–542.
		
		\bibitem{GP24}
		H. Guenancia and M. P\u{a}un, {\em Bogomolov-Gieseker inequality for log terminal K\"{a}hler threefolds}, Commun. Pure Appl. Math. {\bf 78} (2025), No. 11, 2206-–2244.
		
		\bibitem{GP25}
		H. Guenancia and M. P\u{a}un, {\em A complex analytic approach to orbifold Chern classes on singular varieties and its applications}, arxiv: 2601.08627 (2026).
		
		\bibitem{GT22}
		H. Guenancia, and B. Taji, {\em Orbifold stability and Miyaoka-Yau inequality for minimal pairs,} Geometry \& Topology. {\bf 26} (2022), No. 4, 1435--1482.
		
		\bibitem{GPS24} B. Guo, D. H. Phong and J. Sturm,  {\em Green's functions and complex Monge-Ampe\`re equations,}, J. Differential Geom. {\bf 127} (2024), No. 3, 1083--1119.
		
		\bibitem{GPSS23} B. Guo, D. H. Phong, J. Song and J. Sturm, {\em Sobolev inequalities on K\"ahler spaces,} arXiv:2311.00221 (2023).
		
		\bibitem{GPSS24} B. Guo, D. H. Phong, J. Song and J. Sturm, {\em Diameter estimates in K\"ahler geometry,} Comm. Pure Appl. Math. {\bf 77} (2024), no. 8, 3520--3556.
		
		
		
		
		
		
		\bibitem{HM2}
		C. D. Hacon and J. McKernan, {\em On Shokurov's rational connectedness conjecture,} Duke Math. J.
		{\bf 138} (2007), No. 1, 119–-136.
		
		\bibitem{HIT}
		N.J. Hitchin, {\em The self-duality equations on a Riemann surface}, Proc. London Math. Soc. (3) {\bf 55}(1987), 59--126.
		
		\bibitem{HIM22}
		G. Hosono, M. Iwai and S. Matsumura, {\em On projective manifolds with pseudo-effective tangent bundle,} Journal of the Institute of Mathematics of Jussieu {\bf 21} (2022), No. 5, 1801--1830.
		
		\bibitem{HN13}
		A. H\"oring and C. Novelli, {\em Mori contractions of maximal length,} Publ. Res. Inst. Math. Sci. {\bf 49} (2013), 215--228, doi:10.4171/PRIMS/103.
		
		\bibitem{HP16}
		A. H\"oring and T. Peternell, {\em Minimal models for K\"ahler threefolds,} Invent. Math. {\bf 203} (2016), No. 1, 217-–264.
		
		\bibitem{HL10}
		D. Huybrechts and M. Lehn, {\em The geometry of moduli spaces of sheaves,} Cambridge University Press, 2010.
		
		\bibitem{Huang20}
		P. Huang, {\em Non-Abelian Hodge theory and related topics,} SIGMA. Symmetry, Integrability and Geometry: Methods and Applications {\bf 16} (2020), 029.
		
		\bibitem{IJZ25}
		M. Iwai, S. Jinnouchi and S. Zhang, {\em The Miyaoka-Yau inequality for singular varieties with big canonical or anticanonical divisors}, arxiv: 2507.08522 (2025).
		
		\bibitem{IM22}
		M. Iwai and S. Matsumura, {\em Abundance theorem for minimal compact K\" ahler manifolds with vanishing second Chern class,} arXiv:2205.10613 (2022).
		
		\bibitem{IMM24}
		M. Iwai, S. Matsumura and N. M\"uller, {\em Minimal projective varieties satisfying Miyaoka's equality,} Proceedings of the London Mathematical Society {\bf 131} (2025), No. 6, e70104.
		
		
		
		
		
		
		\bibitem{JL25}
		T. Jiang and Jiayu Li, {\em Kobayashi-Hitchin Correspondence for Saturated Reflexive Parabolic Sheaves on K\"ahler manifolds,} arXiv preprint: 2506.03579 (2025).
		
		\bibitem{Jinnouchi25-1}
		S. Jinnouchi, {\em Slope stable sheaves and Hermitian-Einstein metrics on normal
			varieties with big cohomology classes,} arXiv preprint: 2501.04910 (2025).
		
		\bibitem{Jinnouchi25-2}
		S. Jinnouchi, {\em Admissible HYM metrics on klt KE varieties and the MY equality for big anticanonical K-stable varieties,} arXiv preprint : 2512.24161 (2025).
		
		\bibitem{Jacob2014}
		A. Jacob, {\em Existence of approximate Hermitian-Einstein structures on semi-stable bundles,} Asian J. Math.
		{\bf 18} (2014), No. 5, 859-–883.
		
		\bibitem{Jost-Zuo-0}
		J. Jost and K. Zuo, Harmonic maps and Sl(r, C)-representations of fundamental groups
		of quasiprojective manifolds. J. Algebraic Geom. 5(1996), no. 1, 77-106.
		
		\bibitem{Jost-Zuo}
		J. Jost and K. Zuo, {\em Harmonic maps of infinite energy and rigidity results for representations of fundamental groups of quasiprojective varieties,} Journal of Differential Geometry. {\bf 47} (1997), No. 3, 469--503.
		
		\bibitem{Kawa92}
		Y. Kawamata, {\em Abundance theorem for minimal threefolds,} Invent. Math. {\bf 108} (1992), No. 2, 229-–246.
		
		\bibitem{KS21}
		S. Kebekus and C. Schnell, {\em Extending holomorphic forms from the regular locus of a complex space to a resolution of singularities,} Journal of the American Mathematical Society. {\bf 34} (2021), No, 2, 315--368.
		
		\bibitem{Kobayashi2014}
		S. Kobayashi, {\em Differential geometry of complex vector bundles,} Princeton University Press. {\bf 793} (2014).
		
		
		\bibitem{Kollar14}
		J. Koll\'ar, {\em Shafarevich maps and automorphic forms,} In Shafarevich Maps and Automorphic Forms. Princeton University Press, 2014.
		
		\bibitem{KM98}
		J. Koll\'ar and S. Mori, {\em Birational Geometry of Algebraic Varieties,} Cambridge Tracts in Math. {\bf 134}, Cambridge Univ. Press, Cambridge, 1998.
		
		\bibitem{KO25}
		J. Koll\'ar and W. Ou, {\em Orbifold modifications of complex analytic spaces,} arxiv preprint: 2512.20708 (2025).
		
		\bibitem{Langer23}
		A. Langer, {\em Simpson's correspondence on singular varieties in positive characteristic,} arXiv preprint: 2302.05668 (2023).
		
		\bibitem{LSZ19}
		G. Lan, Mao Sheng and Kang Zuo, {\em Semistable Higgs bundles, periodic Higgs bundles and representations of algebraic fundamental groups,} Journal of the European Mathematical Society (2019) {\bf 21}, No. 10.
		
		
		\bibitem{Li96} 
		J. Li, {\em Hitchin's self-duality equations on complete Riemannian manifolds,} Math. Ann.
		{\bf 306} (1996), No. 3, 419--428.
		
		\bibitem{Li20}
		J. Li, {\em Hermitian-Einstein metrics and Chern number inequalities on parabolic stable bundles over K\"ahler manifolds,} Comm. Anal. Geom. {\bf 8} (2000), 445--475.
		
		
		
		\bibitem{LZ}
		J.Y. Li and X. Zhang, {\em Existence of approximate Hermitian-Einstein structures on semi-stable Higgs bundles},  Calc. Var. Partial Differential Equations {\bf52}(2015), 783-795.
		
		\bibitem{LZZ}
		J.Y. Li, C.J. Zhang and X. Zhang, {\em Semi-stable Higgs sheaves and Bogomolov type inequality},  Calc. Var. Partial Differential Equations {\bf56}(2017), no. 3, Paper No. 81, 33 pp.
		
		\bibitem{Li19}
		Q. Li, {\em An introduction to Higgs bundles via harmonic maps,} SIGMA. Symmetry, Integrability and Geometry: Methods and Applications. {\bf 15} (2019), 035.
		
		
		\bibitem{LT18}
		S. Lu, and B. Taji, {\em A characterization of finite quotients of abelian varieties,} International Mathematics Research Notices. (2018), No. 1, 292--319.
		
		\bibitem{Lu99}
		M. L\"ubke, {\em Einstein metrics and stability for flat connections on compact Hermitian
			manifolds, and a correspondence with Higgs operators in the surface case,} Doc. Math.
		{\bf 4} (1999), 487--512.
		
		\bibitem{Ma05}
		X. Ma, {\em Orbifolds and analytic torsions,} Transactions of the American Mathematical Society {\bf 357} (2005), No. 6, 2205-2233.
		
		\bibitem{MTTW25}
		Q. Ma, X. Tang, H.-H. Tseng and Z. Wei, {\em Superconnection and Orbifold Chern character,} arXiv preprint: 2505.13912 (2025).
		
		\bibitem{MMWZ25}
		S. Matsumura, J. Wang, X. Wu and Q. Zhang, {\em Compact K\" ahler manifolds with nef anti-canonical bundle,} arXiv preprint: 2506.23218 (2025).
		
		\bibitem{Mi} 
		M. L. Michelsohn, {\em On the existence of special metrics in complex geometry},
		Acta Math. {\bf 149}(1982), no. 3-4, 261-295.
		
		\bibitem{Miyaoka77}
		Y. Miyaoka, {\em On the Chern numbers of surfaces of general type,} Invent. Math. {\bf 42} (1977), 225-–237.
		
		\bibitem{Miyaoka87}
		Y. Miyaoka, {\em The Chern classes and Kodaira dimension of a minimal variety,} In
		Algebraic geometry, Sendai, 1985, volume 10 of Adv. Stud. Pure Math., pages 449–476. North-Holland, Amsterdam, 1987.
		
		
		\bibitem{Mo0}
		T. Mochizuki, {\em A characterization of semisimple local system by tame pure imaginary pluri-harmonic metric,} arXiv preprint: 0402122 (2024).
		
		\bibitem{Mo1}
		T. Mochizuki, {\em Kobayashi-Hitchin correspondence for tame harmonic bundles and an application}, Ast\'erisque {\bf 309} (2006), viii+117 pp. ISBN: 978-2-85629-226-6.
		
		
		\bibitem{Mo2}
		T. Mochizuki,  {\em Kobayashi-Hitchin correspondence for tame harmonic bundles II}, Geom. Topol. {\bf 13}(2009), 359-455.
		
		\bibitem{Muller26}
		N. M\"uller, {\em Inequalities of Miyaoka-Yau type $\& $ Uniformisation of varieties of intermediate Kodaira Dimension,} arXiv preprint:2601.15138 (2026).
		
		
		\bibitem{NZ}
		Y.C. Nie and X. Zhang,  {\em A note on semistable Higgs bundles over compact K\"ahler manifolds,} Annals of Global Analysis and Geometry {\bf 48} (2015), No. 4, 345--355.
		
		
		
		\bibitem{OV07}
		A. Ogus and V. Vologodsky, {\em Nonabelian Hodge theory in characteristic $p$,} Publ. Math. Inst. Hautes \'{E}tudes Sci.
		(2007), 1-–138.
		
		\bibitem{Ou23}
		W. Ou, {\em On generic nefness of tangent sheaves,} Mathematische Zeitschrift. {\bf 304} (2023), No. 4, 58.
		
		\bibitem{Ou24}
		W. Ou, {\em Orbifold modifications of complex analytic varieties,} arXiv:2401.07273 (2024).
		
		\bibitem{ou25}
		W. Ou, {\em Orbifold Chern classes and Bogomolov-Gieseker inequalities,} arxiv: 2512.22273 (2025).
		
		\bibitem{Pan} C. Pan, {\em Gauduchon metrics and Hermitian-Einstein metrics on Non-K\"ahler varieties,}  arXiv:2503.02759 (2025).
		
		\bibitem{Patel23}
		A. Patel, {\em Uniformization of complex projective klt varieties by bounded symmetric domains,} arXiv: 2301.07591 (2023).
		
		\bibitem{PZZ23}
		C. Pan, C. Zhang and X. Zhang, {\em The non-abelian Hodge correspondence on some non-K\"ahler manifolds,} Science China Mathematics {\bf 66} (2023), No. 11, 2545-2588.
		
		
		\bibitem{Satake}
		I. Satake, {\em On a generalization of the notion of manifold,} Proc. Nat. Acad. Sci. U.S.A. {\bf 42} (1956), 359—-363.
		
		\bibitem{Sche24}
		J. Scheffler, {\em Auxiliary Monge-Amp\`{e}re Equations in Orbifold Setting---a Mean-Value Inequality,} arxiv preprint:2404.02812 (2024).
		
		
		\bibitem{Scheja64}
		G. Scheja, {\em Fortsetzungssätze der komplex-analytischen Cohomologie und ihre algebraische Charakterisierung,} Mathematische Annalen {\bf 157} (1964), No. 1, 75--94.
		
		
		\bibitem{SHW24}
		M. Sheng, Hao Sun and J. Wang, {\em A nonabelian Hodge correspondence for principal bundles in positive characteristic,} arXiv preprint arXiv:2405.09947 (2024).
		
		\bibitem{Simpson88}
		C.-T. Simpson, {\em Constructing variations of Hodge structure using Yang-Mills theory and applications to uniformization,} Journal of the American Mathematical Society. (1988), 867--918.
		
		\bibitem{Simpson90}
		C.-T. Simpson C.T, {\em Harmonic bundles on noncompact curves,} J. Amer. Math. Soc. {\bf 3} (1990), 713–770.
		
		\bibitem{SIM2}
		C.T. Simpson, {\em Higgs bundles and local systems}, Inst. Hautes \'{E}tudes Sci. Publ. Math. {\bf75}(1992), 5-95.
		
		
		\bibitem{Siu06}
		Y.-T. Siu and G. Trautmann, {\em Gap-sheaves and extension of coherent analytic subsheaves,} Vol. 172. Springer, 2006.
		
		\bibitem{Taka03}
		S. Takayama, {\em Local simple connectedness of resolutions of log-terminal singularities,} International Journal of Mathematics. {\bf 14} (2003), No. 8, 825--836.
		
		\bibitem{Timorin98}
		V. A. Timorin, {\em The mixed Hodge-Riemann bilinear relations in the linear situation,} Functional Analysis and Its Applications {\bf 32} (1998), No. 4, 268--272.
		
		\bibitem{Tra67}
		G. Trautmann, {\em Ein Kontinuit\"atssatz f\"ur die Fortsetzung koh\"arenter analytischer Garben,} Archiv der Mathematik {\bf 18} (1967), No. 2, 188--196.
		
		\bibitem{UY86}
		K. Uhlenbeck and S.-T. Yau, {\em On the existence of hermitian-Yang-Mills connections in stable vector bundles,} Communications on Pure and Applied Mathematics, {\bf 39} (1986), No. S1, 257--293.
		
		\bibitem{W08}
		J. Włodarczyk, {\em Resolution of singularities of analytic spaces,} Proceedings of G\"okova Geometry-Topology Conference, 2008.
		
		\bibitem{WZ23}
		D. Wu and X. Zhang, {\em Harmonic metrics and semi-simpleness,} to appear in Adv. Math..
		
		\bibitem{Wu23}
		X. Wu, {\em On compact K\" ahler orbifold,} arXiv preprint: 2302.11914 (2023).
		
		
		
		
		\bibitem{Yau77}
		S.-T. Yau, {\em Calabi's conjecture and some new results in algebraic geometry,} Proc. Nat. Acad. Sci. U.S.A. {\bf 74} (1977), No. 5, 1798-1799.
		
		\bibitem{Yau78}
		S.-T. Yau, {\em On the Ricci curvature of a compact K\"{a}hler manifold and the complex Monge-Amp\`ere equation I,} Comm. Pure Appl. Math. {\bf 31} (1978), No. 3, 339-411.
		
		
		\bibitem{ZZZ25}
		C. Zhang, S. Zhang and X. Zhang, {\em The Miyaoka-Yau inequality for minimal K\"{a}hler klt spaces,} arXiv: 2503.13365v2 (2025).
		
	\end{thebibliography}
\end{document}